\documentclass[10pt,article]{amsart}

\usepackage{amssymb}
\usepackage{amsmath}
\usepackage{amsfonts}
\usepackage{amsthm}
\usepackage{stmaryrd}
\usepackage[all]{xy}
\usepackage{mathrsfs}
\usepackage{graphicx}
\usepackage{hyperref}
\usepackage{color}
\usepackage{multirow}
\usepackage{extarrows}
\usepackage{amscd}
\usepackage{scalerel}
\usepackage{stackengine}
\usepackage{bbm}
\usepackage{mathtools}
\usepackage{cite}
\usepackage{tikz-cd}
\usepackage{bm}
\usepackage{gensymb}
\usepackage{verbatim}
\usepackage{enumerate}
\usepackage{amsthm}
\usepackage{stmaryrd}
\usepackage[all]{xy}
\usepackage{mathrsfs}
\usepackage{graphicx}
\usepackage{hyperref}
\usepackage{color}
\usepackage{multirow}
\usepackage{extarrows}
\usepackage{amscd}
\usepackage{scalerel}
\usepackage{stackengine}
\usepackage{bbm}
\usepackage{mathtools}
\usepackage{mathdots}
\usepackage{cite}
\usepackage{tikz-cd}
\usepackage{bm}
\usepackage{gensymb}
\usepackage{verbatim}
\usepackage{enumerate}
\usepackage{pictexwd,dcpic}
\usepackage{fancyhdr}
\usepackage{appendix}
\usepackage{amssymb,amscd,amsxtra,calc}
\usepackage{mathrsfs}
\usepackage{cmmib57}
\usepackage{multirow}
\usepackage[all]{xy}
\usepackage{longtable}

\numberwithin{equation}{section}
\newtheorem{theorem}{Theorem}[section]
\newtheorem{corollary}[theorem]{Corollary}
\newtheorem{lemma}[theorem]{Lemma}
\newtheorem{proposition}[theorem]{Proposition}
\newtheorem{desideratum}[theorem]{Desideratum}

\newtheorem{remark}[theorem]{Remark}
\theoremstyle{definition}

\DeclareMathOperator{\Aut}{Aut}

\DeclareMathOperator{\End}{End}

\DeclareMathOperator{\Hom}{Hom}

\DeclareMathOperator{\WD}{WD}

\DeclareMathOperator{\id}{id}

\DeclareMathOperator{\Ind}{Ind}
\DeclareMathOperator{\ind}{ind}

\DeclareMathOperator{\GL}{GL}

\DeclareMathOperator{\SL}{SL}

\DeclareMathOperator{\SP}{Sp}
\DeclareMathOperator{\SO}{SO}

\DeclareMathOperator{\Irr}{Irr}
\DeclareMathOperator{\Irrt}{Irr_{temp}}
\DeclareMathOperator{\Para}{\Phi_{temp}}
\DeclareMathOperator{\disc}{disc}
\DeclareMathOperator{\JH}{JH}
\begin{document}

\title{Local Langlands Correspondence for Even Orthogonal Groups via Theta Lifts}
\author{Rui Chen \and Jialiang Zou}
	
	\address
	{
		%\textsc{
		%Department of Mathematics} \endgraf
		\textsc{National University of Singapore,
			Singapore 119076, Republic of Singapore
	}}
	\email{e0046839@u.nus.edu.sg}
	
	\address
	{
		%\textsc{
		%Department of Mathematics} \endgraf
		\textsc{National University of Singapore,
			Singapore 119076, Republic of Singapore
	}}
	\email{e0220154@u.nus.edu}
	\begin{abstract}
		Using theta correspondence, we obtain a classification of irreducible representations of an arbitrary even orthogonal group (i.e.\ the local Langlands correspondence) by deducing it from the local Langlands correspondence for symplectic groups due to Arthur. Moreover, we show that our classifications coincide with the local Langlands correspondence established by Arthur and formulated precisely by Atobe--Gan for quasi-split even orthogonal groups. 
	\end{abstract}

	
	
	\maketitle

\section{Introduction}
In the monumental book \cite{MR3135650}, Arthur obtained a classification of irreducible representations of symplectic and quasi-split special orthogonal groups over local fields
of characteristic 0 (the local Langlands correspondence), as well as a description of the automorphic discrete spectra of these groups over number fields (the Arthur's multiplicity formula). He proved these results by using the theory of endoscopy and the stable trace formula. Later in \cite[Desideratum 3.9]{MR3708200}, Atobe--Gan formulated precisely the local Langlands correspondence (LLC for short) for quasi-split even orthogonal groups and their pure inner forms (using Vogan $L$-packet \cite{MR1216197}). They also highlighted that Arthur's results will imply the LLC for quasi-split even orthogonal groups. M\oe glin \cite[\S 1.4 Theorem 1.4.1]{MR2767522} and M\oe glin--Renard \cite{MR3839702} have partially extended Arthur's results to non-quasi-split even orthogonal groups, though we are not sure if all the statements in \cite[Desideratum 3.9]{MR3708200} were verified in their work.

The main goal of this paper is to construct an LLC for all even orthogonal groups (not necessarily quasi-split), and hence to provide an alternative approach to the works of \cite{MR2767522} and M\oe glin--Renard \cite{MR3839702}. More precisely, assuming the LLC for symplectic groups, we deduce the LLC for all pure inner forms of quasi-split even orthogonal groups by using results from theta correspondence. The LLC for symplectic groups is estabished by Arthur \cite{MR3135650}, with supplements by many others. But the experts may see that, to establish the LLC for symplectic groups, Arthur need to establish the LLC for quasi-split special orthogonal groups simultaneously, since they appear in the endoscopic groups for symplectic groups. Therefore, by assuming Arthur's LLC for symplectic groups, one actually implicitly assumes the LLC for quasi-split special orthogonal groups. Hence, comparing to Arthur's results, we deduce the LLC for non-quasi-split even orthogonal groups by assuming Arthur's LLC for symplectic groups and quasi-split special orthogonal groups.  We prove that the LLC we construct satisfies all the properties in \cite[Desideratum 3.9]{MR3708200}. We also prove the local intertwining relation stated in \cite[\S 3.7]{MR3708200}. As in other instances where the LLC was shown using the theta correspondence (such as \cite{MR2999299} and \cite{MR2800725}), 
we do not show the (twisted) endoscopic character relations for the $L$-packets we constructed. To show that our $L$-packets satisfy the endoscopic character relations, one would need to appeal to the stable trace formula (or a simple form of it), as was done in \cite{MR3267112} and \cite{luo2020endoscopic}. 

We would like to mention some related works. In \cite{MR3788848} and \cite{MR3708200}, Atobe--Gan proved the so-called Prasad conjecture, which describes the almost equal rank theta correspondence in terms of the LLC. In this paper, we turn the table around; namely, imitating the prediction of Prasad conjecture, we construct a Vogan version LLC for even orthogonal groups. We also write a parallel paper \cite{chen2020local}, in which we deal with the unitary group case. We write it separately to avoid making notation too complicated. In a sequel to this paper, with these LLC at hand, we shall investigate the Arthur's multiplicity formula for automorphic discrete spectra of even orthogonal and unitary groups.

This paper is organized as follows. First we recall some basic facts in representation theory and local theta correspondence in Sections 2 and 3. Then we formulate the main theorem \ref{desideratumall} (i.e. the desired LLC) in Section 4, taking the chance to recall the results from Arthur \cite{MR3135650} and Atobe--Gan \cite{MR3708200} that we will use. In Sections 5 and 6, we give our construction of the LLC and prove several properties of this construction. The local intertwining relation is stated and proved in Sections 7 and 8. In Section 9, we prove that the LLC we construct coincides with Arthur's LLC for quasi-split even orthogonal groups. Finally we finish the proof of the main theorem in Section 10. In Appendix A, following Atobe-Gan \cite{MR3708200}, we deduce the weak LLC for arbitrary even special orthogonal groups from the LLC for even orthogonal groups. In Appendix B, we recall the definition of the Plancherel measure and prove a lemma on the normalized intertwining operator. 

\section*{Acknowledgments} 
We would like to thank our supervisor Wee Teck Gan for useful advice. We would also thank Hiraku Atobe, Atsushi Ichino, Wen-Wei Li, and Sug Woo Shin for helpful conversations during the conference ``Workshop on Shimura varieties, representation theory and related topics, 2019'' in Hokkaido University. We thank Hiroshi Ishimoto, Caihua Luo, Xiaolei Wan, and Chuijia Wang for helpful discussions. Both authors are supported by an MOE Graduate Research Scholarship. 
\section*{Notation}
Let $F$ be a nonarchimedean local field of characteristic 0 and residue characteristic $p$. We denote by $|\cdot|_F$ the normalized absolute value of $F$. Let $W_F$ be the Weil group of $F$. We fix a non-trivial additive character $\psi$ of $F$, and for $c\in F^{\times}$, we define an additive character $\psi_{c}$ of $F$ by 
$$\psi_{c}(x)=\psi(cx)\quad \mbox{for $x\in F$}.$$ Note that any non-trivial additive character of $F$ is of the form $\psi_c$ for some $c\in F^{\times}$. Let $(\cdot,\cdot)_F$ denote the quadratic Hilbert symbol of $F$. If $G$ is a linear algebraic group over $F$, we denote by $\Irr(G)$ the set of equivalence classes of irreducible smooth representations of $G$, where we identify $G$ with its group of $F$-valued points $G(F)$. We also denote by $\Irrt(G)$ the subset of $\Irr(G)$ consisting of tempered representations. For any representations $\pi$, the contragredient representation of $\pi$ is denote by $\pi^\vee$. We denote by $\widehat{A}$ the Pontryagin dual of a finite abelian group $A$. 

\section{Orthogonal and Symplectic group}
\subsection{Orthogonal space}
Let $V=V_{2m}$ be an orthogonal space of dimension $2m$ over $F,$ i.e., a vector space equipped with a non-degenerate symmetric bilinear form 
$$\langle\cdot, \cdot\rangle_{V} : V \times V \rightarrow F.$$
We take a basis $\left\{e_{1}, \ldots, e_{2m}\right\}$ of $V,$ and define the discriminant of $V$ by

$$\disc(V)=(-1)^{m} \det ((\left\langle e_{i}, e_{j}\right\rangle_{V})_{i, j})  \in F^{ \times} / F^{ \times 2}.$$
Let $$\chi_{V}=(\cdot, \disc(V))_F$$ be the character of $F^{\times}$ associated to $F(\sqrt{\disc(V)})$. We call $\chi_{V}$ the discriminant character of $V$. 
\begin{comment}
Let $q$ be the quadratic form on $V$ defined by 
\begin{align*}
q(v)=\frac{1}{2}\langle v, v\rangle_{V}.
\end{align*}
We define the normalized Hasse--Witt invariant $\epsilon(V)\in \{\pm 1\}$ of $V$ to be the Witt invariant associated to the quadratic form $q$; see \cite[pp.80-81]{MR770063}. The isometry class of $2m$-dimensional orthogonal spaces $V$ is uniquely determined by these two invariants $\disc(V)$ and $\epsilon(V)$. Moreover, for any fixed $d\in  F^{\times}/F^{\times 2}$ and $\epsilon\in \{\pm 1\}$, there exists an orthogonal space $V$ of dimension $2m$ with 
\begin{align*}
\disc(V)=d\quad \mbox{and}\quad \epsilon(V)=\epsilon, 
\end{align*}
except the case when $d\in F^{\times 2}, \epsilon= -1$ and $m=1$.
\end{comment}

Fix $(d,c)\in (F^\times)^2$. Let 
$$V_{(d, c)}=F[X] /(X^{2}-d)$$
be a $2$-dimensional vector space equipped with a bilinear form
$$(\alpha, \beta) \mapsto\langle\alpha, \beta\rangle_{V_{(d, c)}} \coloneqq c \cdot \operatorname{tr}(\alpha \overline{\beta}),$$
where $\beta \mapsto \overline{\beta}$ is the involution on $F[X] /\left(X^{2}-d\right)$ induced by $a+b X \mapsto a-b X$. This involution is regarded as an element $\epsilon \in \mathrm O\left(V_{(d, c)}\right)$. The images of $1, X \in F[X]$ in $V_{(d, c)}$ are denoted by $e, e^{\prime},$ respectively. For $m>1,$ we say that $V_{2m}$ is associated to $(d, c)$ if
\begin{align}\label{127}
V_{2m} \cong V_{(d,c)} \oplus \mathbb{H}^{m-1}
\end{align}
as orthogonal spaces, where $\mathbb{H}$ is the (orthogonal) hyperbolic plane, i.e., $\mathbb{H}=F v_{i}+F v_{i}^{*}$ with
$$\left\langle v_{i}, v_{i}\right\rangle_{V}=\left\langle v_{i}^{*}, v_{i}^{*}\right\rangle_{V}=0\quad \mbox{and}\quad  \left\langle v_{i}, v_{i}^{*}\right\rangle_{V}=\left\langle v_{i}^{*}, v_{i}\right\rangle_{V}= 1. $$  
Note that if $V_{2m}$ is associated to $(d,c)$, then $\disc (V_{2m})= d \bmod F^{\times 2}$. Moreover, 
$$V_{(d, c)} \oplus 
\mathbb{H}^{m-1} \cong V_{\left(d^{\prime}, c^{\prime}\right)} \oplus \mathbb{H}^{m-1}$$
as orthogonal spaces if and only if 
$$d \equiv d^{\prime} \bmod F^{\times 2}\quad \mbox{and}\quad c \equiv c^{\prime} \bmod N_{E / F}\left(E^{ \times}\right),$$ 
where $E= F(\sqrt{d})=F(\sqrt{d^{\prime}})$.

For a fixed pair $(d,c)\in (F^\times)^2$, we denote by $V_{2m}^+$ the orthogonal space associated to $(d,c)$ and $V_{2m}^-$ the orthogonal space such that 
\[
\dim (V_{2m}^+)=\dim (V_{2m}^-)=2m \quad \mbox{and}\quad \disc  (V_{2m}^+)=\disc  (V_{2m}^-)=d,
\]
but $V_{2m}^+\ncong V_{2m}^-$ as orthogonal spaces. Such $V_{2m}^-$ exists uniquely up to isomorphism unless $m=1$ and $d\in F^{\times 2}$. When $d\notin F^{\times 2}$, it is given so that $V_{2m}^-$ is the orthogonal space associated to $(d,c^\prime)$ with $c^\prime \notin cN_{E/F}(E^\times)$, where $E=F(\sqrt{d})$. When $d\in F^{\times 2}$ and $m>1$, it is given by $V_{2m}^-= D\oplus \mathbb H^{m-2} $, where $D$ is the unique quaternion algebra over $F$, which can be regarded as an orthogonal space with a bilinear form
$$(\alpha, \beta) \mapsto \tau\left(\alpha \beta^{\iota}\right),$$
where $\tau$ is the reduced trace and $\beta \mapsto \beta^{\iota}$ is the main involution. The orthogonal group $\mathrm{O}(V_{2m})$ associated to $V_{2m}$ is defined by $$\mathrm{O}(V_{2m})=\left\{g \in \GL(V_{2m}) |\left\langle g v, g v^{\prime}\right\rangle_{V}=\left\langle v, v^{\prime}\right\rangle_{V} \text { for all } v, v^{\prime} \in V_{2m}\right\},$$
and the special orthogonal group $\SO(V_{2m})$ is defined by 
\begin{align*}
\SO(V_{2m})=\{g\in \mathrm O(V_{2m})|\det g=1\}.
\end{align*}
The orthogonal group $\mathrm{O}\left(V_{2m}\right)$ and the special orthogonal group $\SO(V_{2m})$ are quasi-split if and only if $V_{2m}$ is associated to $(d,c)$ for some $d,c\in F^\times$. 
%Note that when $\disc V=d \notin F^{\times 2}$, both $\mathrm O(V_{2m}^+)$ and $\mathrm O(V_{2m}^-)$ are quasi-split.
\begin{comment}
If $d \notin F^{\times 2}$, let $E=F(\sqrt{d})$ and choose a $c_0\notin N_{E/F}(E^\times)$. We define
\begin{align}\label{139}
V_{2m}^+= V_{(d,1)}\oplus \mathbb H^{m-1}\quad \mbox{and}\quad 
V_{2m}^-= V_{(d,c_0)}\oplus \mathbb H^{m-1}
\end{align}
to be the two different orthogonal spaces of dimension $2m$ and discriminant $d$. Note that  
\begin{align*}
\epsilon(V_{2m}^+)=1\quad \mbox{and}\quad  \epsilon(V_{2m}^-)=-1.
\end{align*}
If $d \in F^{ \times 2}$, let $D$ be the unique division algebra over $F$ of dimension $4$, which can be regarded as an orthogonal space with a bilinear form
$$(\alpha, \beta) \mapsto \tau\left(\alpha \beta^{\iota}\right),$$
where $\tau$ is the reduced trace and $\beta \mapsto \beta^{\iota}$ is the main involution. Then we define  
\begin{align}\label{140}
V_{2m}^+= V_{(d,1)}\oplus \mathbb H^{m-1} \quad \mbox{and}\quad
V_{2m}^-= D\oplus \mathbb H^{m-2} \quad \mbox{if $m\geq 2$}
\end{align} 
to be the two different orthogonal spaces of dimension $2m$ and discriminant $d$. Note that  
\begin{align*}
\epsilon(V_{2m}^+)=1\quad \mbox{and}\quad \epsilon(V_{2m}^-)=-1. 
\end{align*}

\end{comment}

\subsection{Symplectic space}\label{sympleticspace}
Let $W=W_{2 n}$ be a symplectic space of dimension $2n$ over $F,$ i.e., a vector space equipped with a non-degenerate symplectic form
$$\langle\cdot, \cdot\rangle_{W} : W \times W \rightarrow F.$$
The symplectic space is always split, i.e.,
\begin{align}\label{128}
W_{2n}\cong \mathbb{H}^{n} ,
\end{align}
where $\mathbb{H}=F w_i+ Fw_{i}^{*}$ is the (symplectic) hyperbolic plane with 
$$\langle w_i,w_i\rangle_{W} =\langle w_{i}^{*},  w_{i}^{*}\rangle_{W}=0 \quad \mbox{and} \quad \langle w_i, w_{i}^{*}\rangle_{W}=-\langle w_{i}^{*}, w_{i}\rangle_{W}=1.$$
The symplectic group  $\SP(W)$ associated to $W$ is defined by 
$$\operatorname{Sp}(W)=\left\{g \in \mathrm{GL}(W) |\left\langle g w, g w^{\prime}\right\rangle_{W}=\left\langle w, w^{\prime}\right\rangle_{W} \text { for all } w, w^{\prime} \in W\right\}.$$

\subsection{Parabolic subgroups}\label{sectionparabolic}
Let $V=V_{2m}$ be an orthogonal space of dimension $2m$ over $F$. Let $r$ be the Witt index of $V$ and $V_{\mathrm{an}}$ be the anisotropic kernel of $V$. Choose a basis $\{v_i,v_i^*|i=1,\cdots,r\}$ of the orthogonal complement of $V_{\mathrm{an}}$ such that  
$$
\langle v_i, v_j\rangle_V=\langle v_i^*, v_j^*\rangle_V=0 \quad \mbox{and}\quad \langle v_i, v_j^*\rangle_V=\delta_{i, j}\quad \mbox{for $1\leq i,j\leq r$}.
$$
Let $k$ be a positive integer with $k\leq r$ and set 
$$
X=X_k=Fv_1+\cdots+Fv_k \quad \mbox{and}\quad X^\vee=X_k^{\vee}=Fv_1^*+\cdots+Fv_k^*.
$$
Let $V_0$ be the orthogonal complement of $X \oplus X^{\vee}$ in $V$ so that $V_0$ is an orthogonal space of dimension $2m_0=2m-2k$ over $F$. We shall write an element in $\mathrm{O}(V)$ as a block matrix relative to the decomposition $V=X\oplus V_0\oplus X^\vee$. Let $P=P_{k}=M_P U_{P}$ be the maximal parabolic subgroup of $\mathrm O(V)$ stabilizing $X$, where $M_{P}$ is the Levi component of $P$ stabilizing $X^{\vee}$. We have 
$$\begin{aligned} M_{P} &=\left\{m_{P}(a) \cdot h_0 | a \in \mathrm{GL}(X), h_0 \in \mathrm{O}(V_0)\right\} ,\\ U_{P} &=\left\{u_{P}(b) \cdot u_{P}(c) | b \in \operatorname{Hom}(V_0, X), c \in \operatorname{Sym}\left(X^{\vee}, X\right)\right\}, \end{aligned}$$ 
where 
\begin{align*}
m_{P}(a)&=\left(\begin{array}{ccc}{a} & {}& \\ {} & {\mathbf 1_{V_0}} & {} \\ {} & {} & {\left(a^{*}\right)^{-1}}\end{array}\right),\\
u_{P}(b)&=\left(\begin{array}{ccc}{\mathbf 1_{X}} & {b} & {-\frac{1}{2} b b^{*}} \\ {} & {\mathbf 1_{V_0}} & {-b^{*}} \\ {} & {} & {\mathbf 1_{X^{\vee}}}\end{array}\right),\\
u_{P}(c)&=\left(\begin{array}{ccc}{\mathbf 1_{X}} & {} & {c} \\ {} & {\mathbf 1_{V_0}} & {} \\ {} & {} & {\mathbf 1_{X^{\vee}}}\end{array}\right),
\end{align*}
and 
$$\operatorname{Sym}\left(X^{\vee}, X\right)=\left\{c \in \Hom\left(X^{\vee}, X\right) | c^{*}=-c\right\}.$$
Here, the elements $a^{*} \in \GL\left(X^{\vee}\right), b^{*} \in \operatorname{Hom}\left(X^{\vee}, V\right),$ and $c^{*} \in \Hom\left(X^{\vee}, X\right)$ are defined by requiring that 
\begin{align*}
\langle a x, x^{\prime}\rangle_{V}=\langle x, a^{*} x^{\prime}\rangle_{V},\quad \langle b v, x^{\prime}\rangle_{V}=\langle v, b^{*} x^{\prime}\rangle_{V}, \quad \mbox{and}\quad 
\langle c x^{\prime}, x^{\prime \prime}\rangle_{V}=\langle x^{\prime}, c^{*} x^{\prime \prime}\rangle_{V}
\end{align*}
for $x \in X, x^{\prime}, x^{\prime \prime} \in X^{\vee}$ and $v \in V_0$. In particular, $M_P\cong \GL(X)\times \mathrm O(V_0)$.  
Put 
$$\rho_{P}=\frac{2m_0+k-1}{2}, \quad w_{P}=\left(\begin{array}{ccc}
{} &{}& -I_{X} \\ 
{} &{\mathbf 1_{V_0}} & {}\\
{-I_{X}^{-1}}&{}& 
\end{array}\right),$$
where $I_{X} \in \Hom\left(X^{\vee}, X\right)$ is defined by $I_{X} v_{i}^{*}=v_{i}$ for $1 \leq i \leq k$. Then the modular character $\delta_{P}$ of $P$ is given by
$$\delta_{P}\left(m_{P}(a) h_0 u_{P}\right)=|\operatorname{det}(a)|_{F}^{2 \rho_{P}}$$
for $a \in \mathrm{GL}(X), h \in \mathrm{O}(V_0)$ and $u_{P} \in U_{P}$. 

Similarly, let $W=W_{2n}$ be a $2n$ dimensional symplectic space and  $k$ be a positive integer with $k\leq n$. Write $W\cong \mathbb H^n$ as in Subsection \ref{sympleticspace} and set 
$$
Y=Y_k=Fw_1+\cdots+Fw_k \quad \mbox{and}\quad Y^\vee=Y_k^\vee=Fw_1^*+\cdots+Fw_k^*.
$$
Let $W_0$ be the orthogonal complement of $Y \oplus Y^{\vee}$ in $W$ so that $W_0$ is a symplectic space of dimensional $2n_0=2n-2k$ over $F$. We define $Q=M_{Q} U_{Q}, m_{Q}(a), u_{Q}(b), u_{Q}(c), \operatorname{Sym}\left(Y^{\vee}, Y\right)$ and $I_Y$ as above. We put  
$$\rho_{Q}=\frac{2n_0+k+1}{2}, \quad w_{Q}=\left(\begin{array}{ccc}
{} &{}& -I_{Y} \\ 
{} &\mathbf 1_{W_0} & {}\\
{I_{Y}^{-1}}&{}& 
\end{array}\right).$$
Then the modular character $\delta_{Q}$ of $Q$ is given by
$$\delta_{Q}\left(m_{Q}(a) g_0 u_{Q}\right)=|\operatorname{det}(a)|_{F}^{2 \rho_{Q}}$$
for $a \in \mathrm{GL}(Y), g_0 \in \mathrm{Sp}(W_0)$ and $u_{Q} \in U_{Q}$.

% Let $W$ with $\operatorname{dim}(W)=2 n, Q=M_{Q} U_{Q} \subset \operatorname{Sp}\left(W^{\prime}\right), m_{Q}(a),u_{Q}(b), u_{Q}(c), \operatorname{Sym}\left(Y^\vee, Y\right)=\left\{c \in \operatorname{Hom}\left(Y^\vee, Y\right) | c^{*}=-c\right\}$ and $I_{Y}$ as above. We put 
%$$\rho_{Q}=\frac{2n+k+1}{2}, \quad \omega_{Q}=\left(\begin{array}{ccc}
%{} &{}& -I_{Y} \\ 
%{} &{1_{W}} & {}\\
%{I_{Y}^{-1}}&{}& 
%\end{array}\right)$$
%Then the modulus character $\delta_{Q}$ of $Q$ is given by
%$$\delta_{Q}\left(m_{Q}\left(a^{\prime}\right) g u_{Q}\right)=|\operatorname{det}(a)|_{F}^{2 \rho_{Q}}$$
%for $a^{\prime} \in \mathrm{GL}\left(Y^{\prime}\right), g \in \mathrm{Sp}(W)$ and $u_{Q} \in U_{Q}$

\section{Theta correspondence}
In this section, we recall some general results on the local theta correspondence for the reductive dual pair $(\mathrm O(V), \SP(W))$. 
\subsection{Weil representation and local theta correspondence}\label{localthteta}
We fix a non-trivial additive character $\psi$ of $F$. Let $V=V_{2m}$ be a $2m$-dimensional orthogonal space over $F$ and $W=W_{2n}$ be a $2n$-dimensional symplectic space over $F$. We denote the Weil representation of $\mathrm O(V)\times Sp(W)$ by $\omega=\omega_{V,W,\psi}$; see \cite{MR1286835} for a detailed description. For $\pi\in 
\Irr(\mathrm O(V))$, the maximal $\pi$-isotypic quotient of $\omega$ is of the form
$$\pi \boxtimes \Theta_{W,V,\psi}(\pi),$$
where $\Theta_{W,V,\psi}(\pi)$ is a finite length smooth representation of $\SP(W)$. Let $\theta_{W,V, \psi}(\pi)$ be the maximal semisimple quotient of $\Theta_{W,V, \psi}(\pi)$. The Howe duality conjecture, which was proved by Waldspurger \cite{MR1159105} when the residue characteristic is not 2 and by Gan-Takeda \cite{MR3454380}, Gan-Sun \cite{MR3753911} in general, says that 
\begin{itemize}
	\item if $\theta_{W, V,  \psi}(\pi)$ is non-zero, then it is  irreducible;
	\item If $\pi_1\ncong \pi_2$ and both $\theta_{W, V,  \psi}(\pi_1)$ and $\theta_{W, V,  \psi}(\pi_2)$ are nonzero, then $\theta_{W, V, \psi}(\pi_1)\ncong \theta_{W, V,  \psi}(\pi_2)$. 
\end{itemize}
Similarly, for $\sigma \in \Irr(\SP(W))$, We obtain smooth finite length representations $\Theta_{V,W, \psi}(\sigma)$ and $\theta_{V, W, \psi}(\sigma)$ of $\mathrm O(V)$.

In the next two subsections, we will recall some general properties of the local theta correspondence. 
\subsection{First occurrence and tower property}
Let $V_{2m_0}$ be an anisotropic orthogonal space over $F$. 
\begin{comment}where
\begin{align*}
V_{2m_0}= \begin{cases*}
0\quad\quad\,\mbox{or}\quad D \quad &\mbox{if $\disc(V_{2m_0})=1\in F^{\times 2}$},\\
V_{(d,c)}\,\,\mbox{for a $c\in F ^\times$}& \mbox{if $\disc(V_{2m_0})=d\notin  F^{\times 2}$}. 
\end{cases*}
\end{align*}
\end{comment}
For any $r\geq 0$, we put 
\begin{align*}
V_{2m_0+2r}=V_{2m_0}\oplus \mathbb H^r,
\end{align*}
where $\mathbb H$ is the (orthogonal) hyperbolic plane. The collection 
\begin{align*}
\mathcal V=\{V_{2m_0+2r}|r\geq 0\}
\end{align*} 
is called the Witt tower of orthogonal spaces associated to $V_{2m_0}$. One can consider a tower of the theta correspondence associated to reductive dual pairs $\{(\SP(W_{2n}),\mathrm O(V_{2m}))|V_{2m}\in \mathcal V\}$. For $\sigma \in \Irr(\SP(W_{2n}))$ and a Witt tower 
$\mathcal V=\{V_{2m_0+2r}|r\geq 0\}$, we define  
\begin{align}\label{09}
m_{\mathcal V}(\sigma)=\min \{2m |\Theta_{V_{2m},W_{2n},\psi}(\sigma)\neq 0\}. 
\end{align}
By \cite[P.67]{MR1041060}, $m_{\mathcal V}(\sigma)$ is finite.

On the other hand, every symplectic space $W_{2n}$ of dimension $2n$ is isomorphic to $\mathbb H^n$, where $\mathbb H$ is the (symplectic) hyperbolic plane. The collection 
\begin{align*}
\mathcal W=\{W_{2r}|r\geq 0\}
\end{align*}
is called a Witt tower of symplectic spaces. One can also consider a tower of the theta correspondence associated to reductive dual pairs $\{(\mathrm O(V_{2m}),\SP(W_{2n}))|W_{2n}\in \mathcal W\}$. For $\pi\in \Irr(\mathrm O(V_{2m}))$ and a Witt tower $\mathcal W=\{W_{2r}|r\geq 0\}$, we define 
\begin{align}\label{10}
m_{\mathcal W}(\pi)=\min \{2n |\Theta_{W_{2n},V_{2m},\psi}(\pi)\neq 0\}. 
\end{align}
By \cite[P.67]{MR1041060}, $m_{\mathcal W}(\pi)$ is also finite.

The following proposition is often referred to as the tower property of theta correspondence; see \cite{MR818351}. 
\begin{proposition}\label{tower}
For $\sigma \in \Irr(\SP(W_{2n}))$ and a Witt tower 
$\mathcal V=\{V_{2m_0+2r}|r\geq 0\}$ (resp. $\pi\in \Irr(\mathrm O(V_{2m}))$ and a Witt tower $\mathcal W=\{W_{2r}|r\geq 0\}$), let $m_{\mathcal V}(\sigma)$ (resp. $m_{\mathcal W}(\pi)$) be as in (\ref{09}) (resp. (\ref{10})). Then we have 
%	\begin{align*}
%	\Theta_{V_{2m},W_{2n},\psi}(\sigma)\neq 0 \quad \mbox{if}\quad& 2m\geq m_{\mathcal V}(\sigma),\\
%	\Theta_{W_{2n},V_{2m},\psi}(\pi)\neq 0 \quad \mbox{if}\quad &2n\geq m_{\mathcal W}(\pi).
%	\end{align*}

\begin{align*}
	\Theta_{V_{2m},W_{2n},\psi}(\sigma)\neq 0 \,\, \mbox{if}\,\, 2m\geq m_{\mathcal V}(\sigma)\quad \left(\mbox{resp.}\,\,  
	\Theta_{W_{2n},V_{2m},\psi}(\pi)\neq 0 \,\, \mbox{if}\,\, 2n\geq m_{\mathcal W}(\pi)\right).
\end{align*}
\end{proposition}
Note that any two spaces in the same Witt tower have the same discriminant. We define the discriminant of a Witt tower to be $\disc(\mathcal V)=\disc(V_{2m})$ for a (hence any) $V_{2m}$ belonging to $\mathcal V$. For a fixed $d\in F^\times /F^{\times 2}$, there are two different Witt towers $\mathcal V^+=\{V^+_{2m}\}$ and $\mathcal V^-=\{V^{-}_{2m}\}$ such that $\disc(\mathcal V^+)=\disc(\mathcal V^-)=d$. More explicitly, for any $c\in F^\times$, we write  
\begin{align*}
\mathcal V^+&=\{V_{(d,c)}\oplus \mathbb H^{m-1}|m\geq 1\}, \qquad\qquad\qquad\qquad\qquad\quad \\
\mathcal V^-&=\begin{cases*}
\{D\oplus \mathbb H^{m-2}|m\geq 2\} \quad &\mbox{if $\disc(V_{2m})=1\in F^{\times 2}$},\\
\{V_{(d,c^\prime)}\oplus \mathbb H^{m-1}|m\geq 1\}\,\, (c^\prime\notin cN_{E/F}(E^\times))\quad &\mbox{if $\disc(V_{2m})=d\notin F^{\times 2}$}. 
\end{cases*} 
\end{align*}
For $\sigma\in \Irr(\SP(W_{2n}))$, we have $m_{\mathcal V^+}(\sigma)$ and $m_{\mathcal V^-}(\sigma)$ as defined in (\ref{09}).  

On the other hand, there is only a single tower of symplectic spaces $\mathcal W=\{W_{2n}\}$. However, since $\pi$ is a representation of the orthogonal group $\mathrm O(V_{2m})$, we may consider its determinant twist $\pi\otimes\det$. Thus we have two towers of theta lifts 
\begin{align*}
\Theta_{W_{2n},V_{2m},\psi}(\pi)\quad \mbox{and}\quad  \Theta_{W_{2n},V_{2m},\psi}(\pi\otimes\det)
\end{align*}
for $\pi\in \Irr(\mathrm O(V_{2m}))$. Similarly, we have $m_{\mathcal W}(\pi)$ and $m_{\mathcal W}(\pi\otimes \det)$. 

The following theorem is often referred to as the conservation relation; see \cite{MR3369906}. 
\begin{theorem}\label{con2}
	For a fixed $d\in F^\times /F^{\times 2}$ and any $\sigma\in \Irr(\SP(W_{2n}))$ (resp. $\pi\in \Irr(\mathrm O(V_{2m}))$, we have 
	\begin{align*}
	m_{\mathcal V^+}(\sigma)+m_{\mathcal V^-}(\sigma)=4n+4 \quad  (\mbox{resp.}\,\, m_{\mathcal W}(\pi)+m_{\mathcal W}(\pi\otimes \det)=4m),
	\end{align*}
	where $\mathcal V^+$ and $\mathcal V^-$ are two Witt towers defined above such that $\disc(\mathcal V^+)=\disc(\mathcal V^-)=d$. 
\end{theorem}

\subsection{Collection of some results}
We collect some results from \cite{MR3166215} which will be frequently used in this paper. We emphasize that the proofs of these results are independent of the local Langlands correspondence.

\begin{lemma}\label{temperedtotempered}
\begin{enumerate}
	\item Let $\pi\in \Irrt \left(\mathrm  O(V_{2n})\right)$.
	\begin{enumerate}[(i)]
		\item If $\Theta_{W_{2n}, V_{2n},\psi}(\pi)\neq 0$, then $\Theta_{W_{2n}, V_{2n},\psi}(\pi)=\theta_{W_{2n}, V_{2n},\psi}(\pi)$ is an irreducible tempered representation of $\SP(W_{2n})$; 
		\item If $\Theta_{W_{2n-2}, V_{2n},\psi}(\pi)\neq 0$, then $\theta_{W_{2n-2}, V_{2n},\psi}(\pi)$ is an irreducible tempered representation of $\SP(W_{2n-2})$.
	\end{enumerate}
\item Let $\sigma\in \Irrt \left(\SP(W_{2n})\right)$. 
\begin{enumerate}[(i)]
	\item If $\Theta_{V_{2n+2},W_{2n}, \psi}(\sigma)\neq 0$, then $\Theta_{V_{2n+2},W_{2n}, \psi}(\sigma)=\theta_{V_{2n+2},W_{2n}, \psi}(\sigma)$ is an irreducible tempered representation of $\mathrm O(V_{2n+2})$; 
	\item If $\Theta_{V_{2n},W_{2n}, \psi}(\sigma)\neq 0$, then $\theta_{V_{2n},W_{2n}, \psi}(\sigma)$ is an irreducible tempered representation of $\mathrm O(V_{2n})$.
\end{enumerate}
\end{enumerate}

\end{lemma}
\begin{proof}
	These follow from \cite[Proposition C.4]{MR3166215}. 
\end{proof}

\begin{lemma}\label{thetadiscrete} 
	\begin{enumerate}
		\item Let $\pi$ be an irreducible discrete series representation of $\mathrm O(V_{2n})$. 
		\begin{enumerate}[(i)]
			\item If $\sigma=\Theta_{W_{2n-2},V_{2n},\psi}(\pi)\neq 0$, then $\sigma$ is an irreducible discrete series representation of $\SP(W_{2n-2})$ and $\Theta_{W_{2n},V_{2n},\psi}(\pi)$ is an irreducible tempered  representation of $\SP(W_{2n})$ such that 
			$$
			\Theta_{W_{2n},V_{2n},\psi}(\pi)\subseteq \Ind_{Q}^{\SP(W_{2n})}\left(\chi_{V}\otimes \sigma \right),
			$$
			where $Q$ is the parabolic subgroup of $\SP(W_{2n})$ with Levi component $$M_{Q}\cong \GL_1(F)\times \SP(W_{2n-2}).$$
			In this case,  $\Theta_{W_{2n},V_{2n},\psi}(\pi)=\theta_{W_{2n},V_{2n},\psi}(\pi)$ is not a discrete series representation.
			\item If $\Theta_{W_{2n-2},V_{2n},\psi}(\pi)= 0$, then $\Theta_{W_{2n},V_{2n},\psi}(\pi)$ is either zero or an irreducible discrete series representation of $\SP(W_{2n})$. 
		\end{enumerate}
		\item Let $\sigma$ be an irreducible discrete series representation of $\SP(W_{2n})$ and $V_{2n+2}=V_{2n}\oplus \mathbb H$. 
		\begin{enumerate}[(i)]
			\item If $\pi=\Theta_{V_{2n},W_{2n},\psi}(\sigma)\neq 0$, then $\pi$ is an irreducible discrete series representation of $\mathrm O(V_{2n})$ and $\Theta_{V_{2n+2},W_{2n},\psi}(\sigma)$ is an irreducible tempered representation of $\mathrm O(V_{2n+2})$ such that 
			$$
			\Theta_{V_{2n+2},W_{2n},\psi}(\sigma)\subseteq \Ind_{P}^{\mathrm O(V_{2n+2})}\left(\mathrm{1}\otimes \pi \right),
			$$
			where $P$ is the parabolic subgroup of $\mathrm O(V_{2n+2})$ with Levi component $$M_{P}\cong \GL_1(F)\times \mathrm O(V_{2n}).$$
			 In this case, $\Theta_{V_{2n+2}^,W_{2n},\psi}(\sigma)=\theta_{V_{2n+2}^,W_{2n},\psi}(\sigma)$ is not a discrete series representation.
			\item  If $\Theta_{V_{2n},W_{2n},\psi}(\sigma)= 0$, then $\Theta_{V_{2n+2},W_{2n},\psi}(\sigma)$ is either zero or an irreducible discrete series representation of $\mathrm O(V_{2n+2})$.
		\end{enumerate}
	\end{enumerate}	
\end{lemma}
\begin{proof}
	See \cite[Corollary C.3]{MR3166215}.
\end{proof}

Next we give a generalization of \cite[Lemma C.4]{MR3166215}. Put
$$V_{2n}=X_k\oplus V_{2n_0}\oplus X_k^{\vee}, \quad W_{2n}=Y_k\oplus W_{2n_0}\oplus Y_k^{\vee}$$ 
as in Subsection \ref{sectionparabolic}. Let $P=P_{k}=M_{P} U_{P}$ and $Q=Q_{k}=M_{Q} U_{Q}$ be the parabolic subgroups defined in Subsection \ref{sectionparabolic} such that 
$$M_{P} \cong \GL(X_k) \times \mathrm O(V_{2n_0}), \quad M_{Q} \cong \GL(Y_k) \times \SP(W_{2n_0}).$$
Then we have 
\begin{lemma}\label{12}
	Let $\pi$ be an irreducible constituent of $\Ind_{P}^{\mathrm O(V_{2n})}(\tau\otimes \pi_0)$, where $\tau$ is an irreducible discrete series representation of $\GL(X_k)$ and $\pi_0$ is an irreducible tempered representation of $\mathrm O(V_{2n_0})$. Then we have 
	$$
	\theta_{W_{2n},V_{2n},\psi}(\pi) \subseteq \Ind_{Q}^{\SP(W_{2n})}(\tau\chi_{V}\otimes \theta_{W_{2n_0},V_{2n_0},\psi}(\pi_0)).
	$$
 Hence $\theta_{W_{2n},V_{2n},\psi}(\pi)$ is either zero or an irreducible tempered representation of $\SP(W_{2n})$. In particular, $\theta_{W_{2n},V_{2n},\psi}(\pi)$ is zero if $\theta_{W_{2n_0},V_{2n_0},\psi}(\pi_0)$ is zero. 
\end{lemma}
\begin{proof}
	In \cite[Lemma C.4]{MR3166215}, they prove this when $\pi_0$ is a discrete series representation. Their proof can be easily extended to our case. We recall the proof here. 
	
	We denote by $\omega$ and $\omega_{00}$ the Weil representations of $\mathrm O(V_{2n})\times \SP(W_{2n})$ and $\mathrm O(V_{2n_0})\times \SP(W_{2n_0})$ respectively. Since $\pi\subseteq \Ind_{P}^{\mathrm O(V_{2n})}(\tau\otimes \pi_0)$, we have 
	\begin{align*}
	\Theta_{W_{2n},V_{2n},\psi}(\pi)^\vee &\cong \Hom_{\mathrm O(V_{2n})}(\omega, \pi)\\
	&\hookrightarrow \Hom_{\mathrm O(V_{2n})}(\omega, \Ind_{P}^{\mathrm O(V_{2n})}(\tau\otimes \pi_0))\\
	& \hookrightarrow \Hom_{\GL(X_k)\times \mathrm O(V_{2n_0})}(R_{P}(\omega), \tau\otimes \pi_0),
	\end{align*} 
	where $R_{P}(\omega)$ is the normalized Jacquet module of $\omega$ with respect to the parabolic $P$ of $\mathrm O(V_{2n})$. The normalized Jacquet module has been computed by Kudla \cite{MR818351}. More precisely, there is an $M_P\times \SP(W_{2n})$-invariant filtration 
	\begin{align*}
	R_{P}(\omega)= R^0 \supset R^1\supset \cdots \supset R^t \supset R^{k+1}
	\end{align*}
	with successive quotients $J^a\coloneqq R^a/R^{a+1}$ ($0\leq a\leq k$). The reader can consult \cite{MR818351} and \cite[Lemma 5.1]{MR3714507} for an explicit formula for $J_a$ ($0\leq a\leq k$). Here we only recall the formula for $J_k$:  
	$$J^k\cong \Ind_{\GL(X_k)\times \mathrm O(V_{2n_0})\times Q}^{\GL(X_k)\times \mathrm O(V_{2n_0})\times \SP(W_{2n})}(\mathscr S(\mathrm{Isom}(Y_k,X_k))\otimes \omega_{00}),$$
	where 
	\begin{itemize}
		\item $\mathrm{Isom}(Y_k,X_k)$ is the set of invertible linear maps from $Y_k$ to $X_k$ and $\mathscr S(\mathrm{Isom}(Y_k,X_k))$ is the space of locally constant, compactly supported functions on $\mathrm{Isom}(Y_k,X_k)$;
		\item the action of $\GL(X_k)\times \mathrm O(V_{2n_0})\times \GL(Y_k)\times \SP(W_{2n_0})$ on $\mathscr S(\mathrm{Isom}(Y_k,X_k))\otimes \omega_{00}$ is given by 
		\begin{itemize}
			\item $\GL(X_k)\times \GL(Y_k)$ acts on $\mathscr S(\mathrm{Isom}(Y_k,X_k))$ by 
			\begin{align*}
			\left((g,h)\cdot f\right)(x)=\chi_{V}(h)\cdot f(g^{-1}\circ x\circ h)
			\end{align*} 
			for $(g,h)\in \GL(X_k)\times \GL(Y_k)$, $f\in \mathscr S(\mathrm{Isom}(Y_k,X_k))$ and $x\in \mathrm{Isom}(Y_k,X_k)$. 
			\item $\mathrm O(V_{2n_0})\times \SP(W_{2n_0})$ acts by the Weil representation $\omega_{00}$. 
		\end{itemize}
	\end{itemize}
	By the same argument as \cite[Lemma C.4]{MR3166215}, we deduce 
	\begin{align*}
	\Hom_{\GL(X_k)\times \mathrm O(V_{2n_0})}(J^a, \tau\otimes \pi_0)=0\quad \mbox{for $0\leq a<k$}. 
	\end{align*}
	This implies 
	\begin{align*}
	\Theta_{W_{2n},V_{2n},\psi}(\pi)^\vee & \hookrightarrow \Hom_{\GL(X_k)\times \mathrm O(V_{2n_0})}(R_{P}(\omega), \tau\otimes \pi_0)\\
	& \hookrightarrow \Hom_{\GL(X_k)\times \mathrm O(V_{2n_0})}(J^k, \tau\otimes \pi_0)\\
	&\cong \left(\Ind_{Q}^{\SP(W_{2n})} (\tau^\vee\chi_{V}\otimes \Theta_{V_{2n_0},W_{2n_0},\psi}(\pi_{0}))\right)^\vee. 
	\end{align*}
	Take the contragredient functor, we get an epimorphism 
	\begin{align*}
	\Ind_{Q}^{\SP(W_{2n})} (\tau^\vee\chi_{V}\otimes \Theta_{V_{2n_0},W_{2n_0},\psi}(\pi_{0}))\twoheadrightarrow \Theta_{V_{2n},W_{2n},\psi}(\pi). 
	\end{align*} 
	It follows from Lemma \ref{temperedtotempered} that $$\Theta_{W_{2n},V_{2n},\psi}(\pi)=\theta_{W_{2n},V_{2n},\psi}(\pi)\quad\mbox{and}\quad    \Theta_{W_{2n_0},V_{2n_0},\psi}(\pi_0)=\theta_{W_{2n_0},V_{2n_0},\psi}(\pi_0).$$ 
	Apply both the contragredient functor and the MVW functor \cite[Lemma 2.2]{MR3714507}, we get 
	\begin{align*}
	\theta_{W_{2n},V_{2n},\psi}(\pi)\subseteq \Ind_{Q}^{\SP(W_{2n})} (\tau\chi_{V}\otimes \theta_{W_{2n_0}, V_{2n_0},\psi}(\pi_0)). 
	\end{align*}
	This finishes the proof. 
\end{proof}

For $\pi\in \Irr(\mathrm O(V_{2n})$, $\sigma\in \Irr(\SP(W_{2n}))$, and $\chi$ a character of $F^\times$, let $\gamma(s,\pi,\chi,\psi)$ and $\gamma(s,\sigma,\chi,\psi)$ be the standard $\gamma$-factors defined by Lapid-Rallis \cite{MR2192828} using doubling method; see also \cite[\S 10,\S 11]{MR3166215}. The following lemma describes how the standard $\gamma$-factors behave under the theta correspondence.
\begin{lemma}\label{comgamma}
	Let $V_{2n}$ be a $2n$-dimensional orthogonal space and $\chi_V$ be the discriminant character of $V_{2n}$. For any $\pi \in \Irrt \mathrm O(V_{2n})$, we have 
	\begin{enumerate}[(i)]
		\item If $\sigma= \theta_{W_{2n}, V_{2n},  \psi}(\pi)\neq 0$, then 	
		$$
		\frac{\gamma(s,\sigma,\chi,\psi)}{\gamma(s,\pi,\chi\chi_V,\psi)}= \gamma(s,\chi\chi_V,\psi),
		$$
		where $\gamma(s,\chi\chi_V,\psi)$ is Tate's $\gamma$-factor.
		\item If $\sigma= \theta_{W_{2n-2}, V_{2n},  ,\psi}(\pi)\neq 0$, then 
		$$
		\frac{\gamma(s,\pi,\chi\chi_V,\psi)}{\gamma(s,\sigma,\chi,\psi)}= \gamma(s,\chi\chi_V,\psi), 
		$$
		where $\gamma(s,\chi\chi_V,\psi)$ is Tate's $\gamma$-factor.
	\end{enumerate}
\end{lemma}
\begin{proof}
See \cite[Theorem 11.5]{MR3166215}. 
\end{proof}

\begin{lemma}\label{pole}
	Let $V_{2n}$ be a $2n$-dimensional orthogonal space and $\chi_V$ be the discriminant character of $V_{2n}$. For any $\sigma\in \Irr(\SP(W_{2n}))$, if  
	$$	 
	\Theta_{V_{2n},W_{2n},\psi}(\sigma) \neq 0,
	$$
	then $\gamma(s,\sigma,\chi_{V},\psi)$ has a pole at $s=1$. 
\end{lemma}
\begin{proof}
	See \cite[Proposition 11.2]{MR3166215}. 
\end{proof}

Next we describe the behaviour of another representation theoretic invariant, called the Plancherel measure, under the theta correspondence. We denote by $\mu_{\psi}(\tau\otimes\pi)$ (resp. $\mu_{\psi}(\tau\otimes\sigma)$) the Plancherel measure associated to  $\pi\in \Irr(\mathrm O(V_{2n}))$ (resp. $\sigma\in \Irr(\SP(W_{2n}))$) and $\tau\in \Irr(\GL_{k}(F))$; see Appendix \ref{plancherelmeasure} for a definition. 
\begin{lemma}\label{complan}
	\begin{enumerate}[(i)]
		\item  Let $\pi$ and $\sigma$ be irreducible smooth representations of $\mathrm O(V_{2n})$ and $\SP(W_{2n})$ respectively, such that $\sigma=\theta_{W_{2n},V_{2n},\psi}(\pi)$. Let $\tau$ be an irreducible smooth representation of $\GL_{k}(F)$ and put $\tau_s=\tau|\cdot|_{F}^{s}$ for $s\in \mathbb{C}$. Then we have:
		$$
		\frac{\mu_{\psi}(\tau_s\chi_{V}\otimes\sigma)}{\mu_{\psi}(\tau_s\otimes\pi)}=\gamma(s,\tau,\psi)\times \gamma(-s,\tau^{\vee},\psi_{-1}).
		$$ 
		\item 
		Let $\pi$ and $\sigma$ be irreducible smooth representations of $\mathrm O(V_{2n})$ and $\SP(W_{2n-2})$ respectively, such that $\sigma=\theta_{W_{2n-2},V_{2n},\psi}(\pi)$. Let $\tau$ be an irreducible smooth representation of $\GL_{k}(F)$ and put $\tau_s=\tau|\cdot|_{F}^{s}$ for $s\in \mathbb{C}$. Then we have:
		$$
		\frac{\mu_{\psi}(\tau_s\otimes\pi)}{\mu_{\psi}(\tau_s\chi_{V}\otimes\sigma)}=\gamma(s,\tau,\psi)\times \gamma(-s,\tau^{\vee},\psi_{-1}).
		$$ 
	\end{enumerate}
\end{lemma}
\begin{proof}
	See \cite[Theorem 12.1]{MR3166215}. 
\end{proof}

\section{Local Langlands correspondence}
In this section, we state the local Langlands correspondence for symplectic groups and even orthogonal groups. We mainly follow \cite[\S 3]{MR3708200}. 
\subsection{$L$-parameters}\label{Lparmeter}
Let $W_F$ be the Weil group of $F$ and $WD_F=W_F\times \SL_2(\mathbb C)$ be the Weil-Deligne group of $F$. We say that a homomorphism
$\phi : WD_{F} \rightarrow \GL_n(\mathbb C)$ is a representation of $WD_F$ if
\begin{itemize}
	\item $\phi\left(\mathrm{Frob}_{F}\right)$ is semi-simple, where $
	\mathrm{Frob}_{F}$ is a geometric Frobenius element in
	$W_{F}$;
	\item the restriction of $\phi$ to $W_{F}$ is smooth;
	\item  the restriction of $\phi$ to $\mathrm{SL}_{2}(\mathbb{C})$ is algebraic.
\end{itemize}
We call $\phi$ tempered if the image of $W_{F}$ is bounded. We say that $\phi$ is orthogonal if there exists a non-degenerate bilinear form $B : \mathbb C^n \times \mathbb C^n \rightarrow \mathbb{C}$ such that
$$\left\{\begin{array}{l}{B(\phi(w) x, \phi(w) y)=B(x, y)}, \\ {B(y, x)= B(x, y)}\end{array}\right.$$
for $x, y \in \mathbb C^n$ and $w \in WD_{F}$. In this case, $\phi$ is self-dual, i.e., $\phi$ is equivalent to its contragredient $\phi^{\vee}$. 
%More precisely, see \cite{MR3202556}[\S 3]

Suppose that $\phi$ is orthogonal. We may decompose it as follows
\begin{align}\label{111}
\phi=m_{1} \phi_{1}+\cdots+m_{l} \phi_{l}+\varphi+\varphi^{\vee},
\end{align}
where $\phi_{1}, \ldots, \phi_{l}$ are pairwise distinct irreducible orthogonal representations of $WD_{F}$ and $\varphi$ is a sum of irreducible representations of $WD_{F}$ which are not orthogonal. We say that a representation $\phi$ is discrete if $m_{i}=1$ for any $i=1, \ldots, l$ and $\varphi=0$.

Let $\phi$ be an orthogonal representation of $WD_F$ with an invariant non-degenerate bilinear form $B$. We denote the space of $\phi$ by $M$. Let $\Aut(\phi,B)$ be the group of elements in $\GL(M)$ which centralize the image of $\phi$ and preserve $B$. Also we put $\Aut(\phi,B)^+= \Aut(\phi,B)\cap \SL(M)$. We define the component groups $\mathcal {S}_{\phi}$ and $\mathcal {S}_{\phi}^+$ of $\phi$ by 
$$
\mathcal {S}_{\phi}=\Aut(\phi,B)/\Aut(\phi,B)^{\circ}\quad \mbox{and}\quad \mathcal {S}_{\phi}^+=\Aut(\phi,B)^+/\left(\Aut(\phi,B)^{\circ}\cap \Aut(\phi,B)^{+}\right),
$$
where $\Aut(\phi,B)^{\circ}$ is the identity component of $\Aut(\phi,B)$. More explicitly, write $\phi=m_{1} \phi_{1}+\cdots+m_{l} \phi_{l}+ \varphi+\varphi^{\vee}$ as in (\ref{111}). Then 
$$\mathcal {S}_{\phi}\cong \bigoplus_{i=1}^{l}(\mathbb{Z} / 2 \mathbb{Z}) a_{i} \cong (\mathbb{Z} / 2 \mathbb{Z})^l,$$
where $a_i$ corresponds to $\phi_i$. The determinant map $\det: \GL(M) \rightarrow \mathbb{C}^{ \times}$ induces a homomorphism
\begin{equation}\label{detmap}
\begin{aligned}
\det : \mathcal {S}_{\phi} \rightarrow \mathbb{Z} / 2 \mathbb{Z},  \quad \sum_{i=1}^{l} \varepsilon_{i} a_{i} \mapsto \sum_{i=1}^{l} \varepsilon_{i} \cdot \dim\left(\phi_{i}\right),
\end{aligned}
\end{equation}
where $\varepsilon_{i} \in\{0,1\}=\mathbb Z/2\mathbb Z$. Then we have $\mathcal {S}_{\phi}^+=\ker(\det)$. Put
$$z_{\phi} \coloneqq\sum_{i=1}^{l} m_i \cdot a_{i} \in \mathcal {S}_{\phi},$$ 
which is the image of $-\mathbf 1\in \Aut(\phi,B)$ in $\mathcal S_\phi$. We call it the central element in $\mathcal {S}_{\phi}$. Also we put $\bar {\mathcal {S}}_{\phi}=\mathcal {S}_{\phi}/\langle z_{\phi}\rangle$. 
\begin{comment}
Note that when $\phi\in \Phi(\SP(W))$, we have $\det(z_\phi)=-1$, so %$z_\phi\notin \mathcal {S}_{\phi}^+$ and 
\begin{align*}
\mathcal S_\phi =\mathcal S_\phi^+\oplus (\mathbb Z/2\mathbb Z)z_\phi \quad \mbox{and}\quad \mathcal S_\phi^+\cong \bar {\mathcal S}_\phi.
\end{align*}
In this case, we have  
\begin{align}\label{componentgroupsmall}
\widehat{\mathcal S^+_\phi}\cong \widehat{\bar {\mathcal {S}}_{\phi}}, 
\end{align}
where we denote by $\widehat{A}$ the pontryagin dual of an abelian group $A$. 
\end{comment}
For $a=a_{i_{1}}+\cdots+a_{i_{k}} \in \mathcal {S}_{\phi}$ with $1 \leq i_{1}<\cdots<i_{k} \leq l$, we put
$$\phi^{a}=\phi_{i_{1}} + \cdots + \phi_{i_{k}}.$$
By \cite[\S 4]{MR3202556}, for each $c \in F^{ \times}$, we define a character $\eta_{\phi,c}$ of $\mathcal {S}_{\phi}$ by
\begin{align}\label{eta}
\eta_{\phi,c}(a)=\det(\phi^{a})(c).
\end{align}
%Note that when $\phi\in \Phi(\SP(W))$, $\eta_{\phi, c}(z_{\phi})=1$. Hence $\eta_{\phi,c}\in \widehat{\bar{\mathcal S}_\phi}$ in this case. 
Also let \begin{align*}
\upsilon: \mathbb Z/2\mathbb Z \hookrightarrow \mathbb C^{\times}
\end{align*} 
be the natural embedding of $\mathbb Z/2\mathbb Z$ into $\mathbb C^\times$. We define a character $\kappa_{\phi}\in \widehat {\mathcal {S}_{\phi}}$ by composing the map $\det$ in (\ref{detmap}) with $\upsilon$, i.e. 
\begin{align}\label{kappa}
\kappa_{\phi}(a)=\upsilon\circ \det(a)\quad \mbox{for}\,\,a\in \mathcal S_{\phi}. 
\end{align}

Next we define the $L$-parameters for $\mathrm O(V_{2n})$ and $\SP(W_{2n})$. For $G=\mathrm O(V_{2n})$ or $\SP(W_{2n})$, following \cite[\S 8]{MR3202556} and \cite[\S 3]{MR3708200}, we define   
\[
\Phi(G)= 
\begin{cases}
\{\phi:\WD_{F} \rightarrow \mathrm{O}(2n,\mathbb{C}) |\det(\phi)=\chi_{V}\} /(\mathrm{O}(2n, \mathbb{C})\mbox{-conjugacy})\quad &\mbox{if $G=\mathrm O(V_{2n})$};\\
\{\phi : \WD_{F} \rightarrow \SO(2n+1, \mathbb{C})\} /(\SO(2n+1, \mathbb{C})\mbox{-conjugacy})\quad &\mbox{if $G=\SP(W_{2n})$}. 
\end{cases}
\]
We call an element $\phi\in \Phi(G)$ an $L$-parameter for $G$. When $G=\mathrm O(V_{2n})$, we may view $\phi\in \Phi(G)$ as a $2n$-dimensional orthogonal representation of $\WD_F$ by composing $\phi$ with the standard embedding $\mathrm{O}(2n,\mathbb{C})\hookrightarrow \GL_{2n}(\mathbb C)$. Similarly, when $G=\SP(W_{2n})$, we may view $\phi\in \Phi(G)$ as a $2n+1$-dimensional orthogonal representation of $\WD_F$.  For $G=\mathrm O(V_{2n})$ or $\SP(W_{2n})$, we denote the subset of $\Phi(G)$ consisting of tempered (resp. discrete)
representations by $\Phi_{\mathrm{temp}}(G)$ (resp. $\Phi_{\mathrm{disc}}(G)$). Then we have a sequence 
\begin{align*}
\Phi_{\disc}(G)\subseteq \Phi_{\mathrm{ temp }}(G)\subseteq \Phi(G).
\end{align*}
When $G=\mathrm O(V_{2n})$, we define $\Phi^{\epsilon}(G)$ to be the subset of $\Phi(G)$ consisting of $\phi$ which contains an irreducible orthogonal subrepresentation of $WD_F$ of odd dimension. 
%We put $\Phi^{\epsilon}_{*}(G)=\Phi^\epsilon(G)\cap \Phi_{*}(G)$ for $*\in \{\disc, \mathrm{ temp }\}$. 

For a representation $\phi$ of $\WD_F$, we define the local factors $L(s,\phi)$, $\varepsilon(s,\phi,\psi)$ and $\gamma(s, \phi, \psi)$ as in \cite{MR546607}. Note that the $\varepsilon$-factor and $\gamma$-factor depend on the choice of the additive character $\psi$ while the $L$-factor does not. 
% If $(\phi, M)$ is an orthogonal (resp. symplectic) representation with a $W D_{F}$ -invariant
%symmetric (resp. alternate) bilinear form $B,$ then we define the adjoint $L$ -function $L(s, \phi, \text { Ad })$ associated to $\phi$ to be the $L$ -function associated to 
% $$\mathrm{Ad}\circ \phi : W D_{F} \rightarrow \operatorname{GL}(\operatorname{Lie}(\operatorname{Aut}(M, B)))$$
% We say that $\phi$ is generic if $L(s, \phi, \mathrm{Ad})$ is regular at $s=1 .$ since $B$ is symmetric (resp. alternate), we have $\mathrm{Ad}\circ \phi \cong \wedge^{2} \phi\left(\text {resp} \mathrm{Ad} \circ \phi \cong \operatorname{Sym}^{2} \phi\right)$. Hence the adjoint L-function $L(s, \phi, \mathrm{Ad})$ is equal to the exterior square $L$ -function $L\left(s, \phi, \wedge^{2}\right)=L\left(s, \wedge^{2} \phi\right)$ (resp. the symmetric square $L$ -function $L\left(s, \phi, \mathrm{Sym}^{2}\right)=L\left(s, \operatorname{Sym}^{2} \phi\right) )$ 

The following lemma in \cite[Lemma 12.3]{MR2999299} and \cite[Lemma A.6]{MR3573972} will be used later.  
\begin{lemma}\label{gammadetermine}
	Let $\phi_{1}$ and $\phi_{2}$ be two tempered orthogonal representations of $\WD_F$ of the same dimension. Assume that 
	$$\gamma\left(s, \phi_{1} \otimes \phi_{\rho}, \psi\right) \cdot \gamma\left(-s, \phi_{1} \otimes \phi_{\rho}^{\vee}, \psi_{-1}\right)=\gamma\left(s, \phi_{2} \otimes \phi_{\rho}, \psi\right) \cdot \gamma\left(-s, \phi_{2} \otimes \phi_{\rho}^{\vee}, \psi_{-1}\right)$$
	for every irreducible representation $\phi_{\rho}$ of $\WD_F$. Then
	$$\phi_{1} \cong \phi_{2}$$
	as representations of $\WD_F$. 
\end{lemma} 

\subsection{Whittaker data}\label{whittaker}
In order to describe the (Vogan version) LLC for symplectic and even orthogonal groups, we need to choose Whittaker data on these groups.
\begin{comment}
Let $G$ be a quasi-split reductive group over $F$. A Whittaker datum of $G$ is a conjugacy class of a pair $\mathfrak W=(B,\mu)$, where 
\begin{itemize}
\item $B=TU$ is a rational Borel subgroup of $G$. 
\item $\mu$ is a generic character of $U(F)$.
\end{itemize}
Note that when $V_{2n}$ is associated to $(d,c)$, $\mathrm O(V_{2n})$ is quasi-spit but not connected, in this case, we define the Whittaker datum of $\mathrm O(V_{2n})$ to be the Whittaker datum of $\SO(V_{2n})$. 
\end{comment}
Firstly we describe the Whittaker datum for symplectic groups. Let $W_{2n}$ be a symplectic space of dimension $2n$. In this case, a Whittaker datum of $\SP(W_{2n})$ is a conjugacy class of the pair $\mathfrak W^\prime=(B^\prime,\mu^\prime)$, where $B^\prime=T^\prime U^\prime$ is an $F$-rational Borel subgroup of  $\SP(W_{2n})$ and $\mu$ is a generic character of $U$. Write $W_{2n}=\mathbb H^n$ as in (\ref{128}). Let $B^{\prime}=T^{\prime} U^{\prime}$ be the $F$-rational Borel subgroup of $\SP(W_{2n})$ stabilizing the complete flag
$$0 \subset\left\langle w_{1}\right\rangle \subset\left\langle w_{1}, w_{2}\right\rangle \subset \cdots \subset\left\langle w_{1}, \ldots, w_{n}\right\rangle= Y_{n},$$
where $T^{\prime}$ is the $F$-split torus stabilizing the lines $F w_{i}$ for $i=1, \ldots, n $. For $c \in F^{\times}$, we define a generic
character $\mu_{c}^{\prime}$ of $U^{\prime}$ by
$$\mu_{c}^{\prime}(u)=\psi\left(\left\langle uw_{2}, w_{1}^{*}\right\rangle_{W}+\cdots+\left\langle u w_{n}, w_{n-1}^{*}\right\rangle_{W}+c\left\langle u w_{n}^{*}, w_{n}^{*}\right\rangle_{W}\right).$$
By \cite{MR3202556} $[\S 12]$ the map $c \mapsto \mu_{c}^{\prime}$ gives a bijection (depending on the choice of $\psi )$
$$F^{\times}/F^{\times 2} \longleftrightarrow \{\mbox{$T^{\prime}$-orbits of generic characters of $U^{\prime}$}\}.$$
For $c\in F^\times$, we define a Whittaker datum of $\SP(W_{2n})$ by $\mathfrak W^\prime_{\psi,c}=(B^\prime,\mu^\prime_c)$. 
\begin{comment}
Then we have a bijection 
\begin{align*}
\{F^{\times} /F^{\times 2}\}&\longleftrightarrow  \{\mbox{Whittaker data of $\SP(W_{2n})$}\}\\
c &\mapsto \mathfrak W^\prime_{\psi,c}= (B^\prime,\mu^\prime_c).
\end{align*}
\end{comment}
Note that $\mathfrak W^\prime_{\psi,c}$ depends on the choice of $\psi$. More precisely, for any $a\in F^\times $, it is easy to check that 
\begin{align}\label{121}
\mathfrak W^{\prime}_{(\psi_a,c)}= \mathfrak W^{\prime}_{(\psi,ac)}. 
\end{align}
We say that $\sigma\in \Irr(\SP(W_{2n}))$ is $\mathfrak W_{\psi,c}^\prime$-generic if 
\begin{align*}
\Hom_{U^\prime}(\sigma, \mu^\prime_c)\neq 0.
\end{align*}

Next we describe the Whittaker datum for even orthogonal groups. In this case, the Whittaker datum is not only associated to the group $\mathrm O(V_{2n})$, but rather the orthogonal space $V_{2n}$ is a part of the Whittaker datum. This is precisely defined in \cite[\S 2.2]{MR3708200}. For convenience of the reader, we briefly recall it here. Fix $d\in F^\times/F^{\times 2}$, a Whittaker datum is an equivalence class of the $4$-tuple $(V_{2n},B_0,T,\mu)$, where
\begin{itemize}
	\item  $V_{2n}$ is an orthogonal space with discriminant $\disc(V_{2n})=d \mod F^{\times 2}$; 
	\item $B$ be the $F$-rational Borel subgroup of $\SO (V_{2n})$;
	\item $T$ is a maximal $F$-torus contained in $B$;
	\item $\mu$ is a generic character of $U$, where $U$ is the unipotent radical of $B$. 
\end{itemize}
The equivalence relation is defined in \cite[\S 2.2]{MR3708200}. For any $c\in F^\times$, let $V_{2n}$ be the orthogonal space associate to $(d,c)$. Write $V_{2n}\cong V_{(d,c)}\oplus \mathbb H^{n-1}$ as in (\ref{127}). We denote by 
$B=TU$ the $F$-rational Borel subgroup of $\SO (V_{2n})$ stabilizing the complete flag
	$$0 \subset\left\langle v_{1}\right\rangle \subset\left\langle v_{1}, v_{2}\right\rangle \subset \cdots \subset\left\langle v_{1}, \ldots, v_{n-1}\right\rangle= X_{n-1},$$
where $T$ is the $F$-torus contained in $B$ stabilizing the lines $F v_{i}$ for $i=1, \ldots, n-1$. We define a generic character $\mu_c$ of $U$ by 
 $$\mu_{c}(u)=\psi\left(\left\langle u v_{2}, v_{1}^{*}\right\rangle_{V}+\cdots+\left\langle u v_{n-1}, v_{n-2}^{*}\right\rangle_{V}+\left\langle u e, v_{n-1}^{*}\right\rangle_{V}\right).$$ %where $U$ is the unipotent radical of $B$. 
By \cite[Proposition 2.1]{MR3708200}, the map $c \mapsto  (V_{2n},B,T,\mu_c)$ gives a bijection (not depending on the choice of $\psi$)
\begin{align*}
F^{\times} /F^{\times 2}\longleftrightarrow  \{\mbox{equivalence classes of tuples $(V_{2n},B,T,\mu)$ with $\disc(V_{2n})=d$}\}. 
\end{align*} 
For $c\in F^\times$, we define a Whittaker datum of even orthogonal groups by $\mathfrak W_{c}=(V_{2n},B,T,\mu_c)$. 
Next we define the notion of generic representation for $\mathrm O(V_{2n})$. We identify $\mathrm{O}\left(V_{(d, c)}\right)$ as the subgroup of $\mathrm{O}\left(V_{2 n}\right)$ which fixes $\mathbb{H}^{n-1} $. Via the canonical embedding $\mathrm{O}\left(V_{(d, c)}\right) \hookrightarrow \mathrm{O}\left(V_{2 n}\right),$ we regard $\epsilon$ as an element in $\mathrm{O}\left(V_{2 n}\right)$. Note that $\epsilon$ normalizes $U$ and fixes $\mu_{c}$, so we can extend $\mu_{c}$ to $\widetilde{U}=U \rtimes\langle\epsilon\rangle$. There are exactly two such extensions $\mu_{c}^{\pm} : \widetilde{U} \rightarrow \mathbb{C}^{\times}$ which are determined by
$$\mu_{c}^{\pm}\left(\epsilon\right)=\pm 1.$$
As in \cite[\S 2.2]{MR3708200}, we say that $\pi\in \Irr(\mathrm{O}(V_{2n}))$ is $\mathfrak W_c^{\pm}$-generic if
$$\Hom_{\widetilde{U}}(\pi, \mu_{c}^{\pm}) \neq 0.$$
It is easy to check that $\pi$ is $\mathfrak W_c^+$ generic if and only if $\pi\otimes\det$ is $\mathfrak W_c^-$-generic.

\subsection{Local Langlands correspondence for symplectic groups}
We first describe the LLC for symplectic groups. This was proved by Arthur\cite{MR3135650}, with supplements by many others. 
 
\begin{theorem}\label{llcsympletic}
Let $W_{2n}$ be a $2n$-dimensional symplectic space.  
	\begin{enumerate}[(1).]
		\item There exists a surjective map
		$$
		\mathcal L:  \Irr \left(\SP(W_{2n})\right) \longrightarrow \Phi(\SP(W_{2n})),$$
		which is finite-to-one. For any $\phi \in \Phi(\SP(W_{2n}))$, we denote $\mathcal L^{-1}(\phi)$ by $\Pi_{\phi}$ and call it the $L$-packet of $\phi$. 
		\item For each Whittaker datum $\mathfrak W^\prime_{\psi,c}$ of $\SP(W_{2n})$, there is a canonical map
		\begin{align}
		\mathcal J_{\mathfrak W^\prime_{\psi,c}} : \Pi_{\phi} \longrightarrow \widehat{\bar {\mathcal {S}}_{\phi}}. 
		\end{align}
		We write $\sigma=\sigma_{\mathfrak W^\prime_{\psi,c}}(\phi,\eta)$ if $\sigma \in \Pi_{\phi}$ corresponds to $\eta \in \widehat{\bar {\mathcal S}_{\phi}}$ under $\mathcal J_{\mathfrak W^\prime_{\psi,c}}$. 
		\item Assume that $\phi$ is a tempered $L$-parameter of $\SP(W_{2n})$. Then for each Whittaker datum $\mathfrak W^\prime_{\psi,c}$ of $\SP(W_{2n})$, there is a unique $\mathfrak W^\prime_{\psi,c}$-generic representation $\sigma$ in $\Pi_{\phi}$, which corresponds to the trivial character of $\bar {\mathcal {S}}_{\phi}$ under $\mathcal J_{\mathfrak W^\prime_{\psi,c}}$. 
		\item 
		Let $\mathfrak W^\prime_{\psi,c_1}$ and $\mathfrak W^\prime_{\psi,c_2}$ be two Whittaker data of $\SP(W_{2n})$. Then for any $\sigma\in \Pi_{\phi}$, we have 
		$$
		\mathcal J_{\mathfrak W^\prime_{\psi,c_2}}(\sigma)=\mathcal J_{\mathfrak W^\prime_{\psi,c_1}}(\sigma)\otimes \eta_{\phi,c_2/c_1},
		$$
		where $\eta_{\phi,c_2/c_1}$ is defined in (\ref{eta}). 
		\item
		The map $\mathcal L$ preserves temperedness and discreteness, i.e., $\sigma\in \Pi_{\phi}$ is a tempered representation if and only if $\phi$ is a tempered parameter and $\sigma\in \Pi_{\phi}$ is a discrete series representation if and only if $\phi$ is a discrete parameter.
		\item
		\textbf{(Local intertwining relation)}
		Assume that $\phi=\phi_{\tau}+\phi_{0}+\phi_{\tau}^{\vee}$, where $\phi_{0}$ is an element in $\Phi_{\mathrm {temp}}\left(\SP(W_{2n_0})\right)$
		and $\phi_{\tau}$ is an irreducible tempered representation of $\WD_{F}$ corresponding to $\tau \in \Irr(\GL_{k}(F))$, so there is a natural embedding $\mathcal S_{\phi_{0}}\hookrightarrow \mathcal {S}_{\phi}$. Let $Q$ be a parabolic subgroup of $\SP(W_{2n})$ with Levi subgroup $$M_{Q}\cong \GL_{k}(F) \times \SP(W_{2n_0})$$ and $\sigma_{0}$ be the irreducible tempered representation of $\SP(W_{2n_0})$ corresponding to $(\phi_{0},\eta_0)$ under $\mathcal L$ and $\mathcal J_{\mathfrak W^\prime_{\psi,1}}$. Then the induced representation $\Ind_{Q}^{\SP(W_{2n})}\left(\tau \otimes \sigma_{0}\right)$ has a decomposition 
		$$
		\Ind_{Q}^{\SP(W_{2n})}(\tau \otimes \sigma_{0})=\bigoplus_{\eta}\sigma_{\mathfrak W^\prime_{\psi,1}}(\phi,\eta),
		$$
		where the sum runs over all $\widehat{\bar{ \mathcal S}_{\phi}}$ such that $\eta|_{ \bar{\mathcal S}_{\phi_0}}=\eta_0$. Moreover if $\phi_{\tau}$ is self-dual and of orthogonal type, let 
		$$
		R_{\mathfrak W^\prime_{\psi,1}}(w, \tau \otimes \sigma_0)
		\in \End_{\SP(W_{2n})}\left(
		\Ind_{Q}^{\SP(W_{2n})}(\tau \otimes \sigma_{0})\right)$$ be the normalized intertwining operator associated to the Whittaker datum $\mathfrak W^\prime_{\psi,1}$ (see Section \ref{normalizingintertwing} below), where $w$ is the unique non-trivial element in the relative Wely group for $M_Q$. Then
		$$R_{\mathfrak W^\prime_{\psi,1}}(w, \tau \otimes \sigma_0) | _{\sigma}=\mathcal J_{\mathfrak W^\prime_{\psi,1}}(\sigma)(a),$$
		where $a \in \mathcal {S}_{\phi}$ corresponds to $\phi_{\tau}$.
		\item \textbf{(Compatibility with Langlands quotients)}
		Assume that
		$$\phi=\phi_1|\cdot|_{F}^{s_{1}}+ \cdots+ \phi_{r}|\cdot|_{F}^{s_{r}}+\phi_{0}+ \phi_{r}^\vee |\cdot|_{F}^{-s_{r}}+ \cdots+ \phi_1^\vee|\cdot|_{F}^{-s_{1}},$$ 
		where 
		\begin{itemize}
			\item $\phi_i$ is an irreducible tempered representation of $\WD_F$ of dimension $d_i$; 
			\item $\phi_0\in \Phi_{\mathrm{temp}}(\SP(W_{2n_0}))$;
			\item $s_1\geq \cdots \geq s_r>0$;
			\item $d_1+\cdots +d_r+n_0=n$. 
		\end{itemize}
		We denote by $\tau_{i}$ the irreducible tempered representation of $\GL_{d_i}(F)$ corresponding to $\phi_i$. Then the $L$-packet $\Pi_{\phi}$ consists of the unique irreducible quotients $\sigma$ of the standard modules $$\Ind_{Q}^{\SP(W_{2n})}\left(\tau_{1}|\cdot|_{F}^{s_{1}} \otimes \cdots \otimes \tau_{r}|\cdot|_{F}^{s_{r}} \otimes \sigma_{0}\right),$$
		where $Q$ is a parabolic subgroup of $\SP(W_{2n})$ with Levi subgroup $$M_{Q}\cong \GL_{k_{1}}(F) \times \cdots \times \GL_{k_{r}}(F) \times \SP(W_{2n_0})$$ and $\sigma_{0}$ runs over elements of $\Pi_{\phi_{0}}(\SP(W_{2n_0}))$. Moreover, the natural embedding $\mathcal S_{\phi_0}\hookrightarrow  \mathcal S_{\phi}$ is an isomorphism and $$\mathcal J_{\mathfrak W^\prime_{\psi,c}}(\sigma)=\mathcal J_{\mathfrak W^\prime_{\psi,c}}(\sigma_0)$$ if we identify $\mathcal S_{\phi_0}$ with $\mathcal S_{\phi}$ via the above isomorphism. 
		\item
		The map $\mathcal L$ respects standard $\gamma$-factors. Namely, we have 
		$$
		\gamma(s,\sigma,\chi,\psi)=\gamma(s,\phi\otimes\chi,\psi)
		$$
		for $\sigma\in \Pi_{\phi}$ and any character $\chi$ of $F^{\times}$. Here $\gamma(s,\sigma,\chi,\psi)$ is the standard $\gamma$-factor defined by Lapid-Rallis \cite{MR2192828} and $\gamma(s,\phi\otimes\chi,\psi)$ is the $\gamma$-factor defined in \cite{MR546607}.  
		\item 
		The map $\mathcal L$ respects Plancherel measures. Namely, we have 
		\begin{align*}
		\mu_{\psi}(\tau_s\otimes \sigma)&=\gamma(s,\phi_{\tau}\otimes \phi^{ \vee},\psi)\times \gamma(-s,\phi_{\tau}^{\vee}\otimes\phi, \psi_{-1} )\\
		& \times \gamma (2s, \wedge^2\circ \phi_{\tau}, \psi)\times \gamma(-2s, \wedge^{2} \circ \phi_{\tau}^{\vee},\psi_{-1})
		\end{align*}
		for any $\sigma\in \Pi_{\phi}$ and $\tau\in \Irr (\GL_k(F))$ with $L$-parameter $\phi_{\tau}$. 
	\end{enumerate}
\end{theorem}
\begin{remark}
	The local intertwining relation we used here is the same as in \cite[\S 6.6]{MR3788848}, which is a little bit different from the local intertwining relation formulated by Arthur \cite[\S 2.4]{MR3135650}. In \cite[Theorem 2.2]{MR3801418}, Atobe proved that the local intertwining relation we used here is a consequence of the local intertwining relation formulated by Arthur. 
\end{remark}

%$$
%\iota_{c_2}(\sigma).\iota_{c_1}(\sigma)^{-1}=\eta_{\phi,c_2/c_1}
%$$

%(2) $\phi$ is generic, i.e., $L(s,\phi,Ad)$ is regular at $s=1$ if and only if $\Pi_{\phi}$ contains a $\mu-$ generic representation $\sigma$ for each generic character $\mu$. 

%(3) Suppose $\phi$ is generic, then for each $c\in F^{\times}$, $\sigma\in \Pi_{\phi}$ is $\mu_{c}$ generic if and only if $l_c(\sigma)$ is the trivial representation of $\mathcal {S}_{\phi}^{+}$.

%(4) If $\phi$ are unramified, then $\Pi_{\phi}$ contains a unique unramified representation $\sigma$, and it corresponds to the trivial representation of $\mathcal {S}_{\phi}^{+}$ under $\iota_{1}$. 

%The adjoint group $\SP_{ad}(W)$ acts on $\SP(W)$ by conjugation, and hence acts on $\Irr(\SP(W))$. This action factor through $\SP_{ad}(W)/\im \SP(W)$. This quotient is isomorphic to the cohomology group 
%$$
%E=\SP_{ad}(W)/\im \SP(W)=\ker (H^1(F,Z)\rightarrow H^1(F,\SP(W)))=F^{\times}/F^{\times 2}
%$$
%Here $Z\cong \mu_2$ is the center of $\SP(W)$ and the last equliaty comes from 	
%$$
%H^1(F,\mu_2)= F^{\times}/F^{\times 2} \quad H^1(F,\SP(W)) =1
%$$
%For $d\in F^{\times}/F^{\times 2}=E$, we denote the action of $d$ on $\sigma\in Irr(\SP(W))$ by $\sigma^{d}$. We have the following proposition describing the adjoint action on $\Irr(\SP(W))$ in terms of $L$-parameter and character. 
 
For any $a\in F^\times$, let $\delta_a\in \GL(W_{2n})$ be such that 
\begin{align}\label{112}
\langle \delta_a w, \delta_a w^\prime\rangle_W=a \langle w,w^\prime\rangle_W \,\,\mbox{for all $w,w^\prime\in W_{2n}$}. 
\end{align}
For any $\sigma\in \Irr(\SP(W_{2n}))$, we define a new representation $\sigma^{\delta_a}$ of $\SP(W_{2n})$ by 
\begin{align}\label{113}
\sigma^{\delta_a}(g)=\sigma(\delta_a^{-1}g\delta_a). 
\end{align}
Note that a different choice of $\delta_a$ differs by an element in $\SP(W_{2n})$. Hence the isomorphic class of $\sigma^{\delta_a}$ is independent of the choice of $\delta_a$. The map  
\begin{align*}
\SP(W_{2n})&\rightarrow \SP(W_{2n})\\
g &\mapsto \delta_a^{-1}g\delta_a
\end{align*}
transfers $\mathfrak W^\prime_{\psi,c}$ to $\mathfrak W^\prime_{\psi,ac}$. It follows from \cite[Theorem 4.3]{MR3194648} that 
\begin{align}\label{103}
\mathcal L(\sigma^{\delta_a})=\mathcal L(\sigma)\quad\mbox{and}\quad \mathcal J_{\mathfrak W^\prime_{\psi,c}}(\sigma^{\delta_a})=\mathcal J_{\mathfrak W^\prime_{\psi,ac}}(\sigma)=\mathcal J_{\mathfrak W^\prime_{\psi,c}}(\sigma)\otimes \eta_{\phi,a}, 
\end{align}
where $\phi=\mathcal L(\sigma)$. Recall that according to \cite[Chapter 4. II.1]{MR1041060}, we have $\sigma^\vee\cong \sigma^{\delta_{-1}}$. Hence in particular, it follows that  
\begin{align}\label{104}
\mathcal L(\sigma^\vee)=\mathcal L(\sigma)\quad\mbox{and}\quad \mathcal J_{\mathfrak W^\prime_{\psi,c}}(\sigma^\vee)=\mathcal J_{\mathfrak W^\prime_{\psi,c}}(\sigma)\otimes \eta_{\phi,-1}. 
\end{align}

\subsection{Local Langlands correspondence for even orthogonal groups}\label{LLCeven}
Next we state the LLC for even orthogonal groups, which is the main theorem of this paper.  

\begin{theorem}\label{desideratumall}
	Fix $(d,c)\in (F^\times)^2$. Let $V_{2n}=V_{2n}^+$ be the orthogonal space associated to $(d,c)$, and $\chi_V=(\cdot, d)_F$ be the discriminant character of $V_{2n}$. 
%	Fix $d\in F^\times /F^{\times 2}$. Let $\chi=(\cdot, d)_F$ be the character of $F^\times$ associated to $E=F(\sqrt d)$. Let $V=V_{2n}$ be a $2n$-dimensional orthogonal space with discriminant character $\chi_V=\chi$.
	\begin{enumerate}[(1).]
		\item There exists a surjective map
		$$
		\mathcal L: \bigsqcup_{\delta\in \{\pm 1\}} \Irr \left(\mathrm O(V_{2n}^\delta)\right) \longrightarrow \Phi(\mathrm O(V_{2n})),$$
		which is finite-to-one. For any $\phi \in \Phi(\mathrm O(V_{2n}))$, we denote $\mathcal L^{-1}(\phi)$ by $\Pi_{\phi}$ and call it the $L$-packet of $\phi$. We also write $\Pi_{\phi}(\mathrm O(V_{2n}))=\Pi_{\phi}\cap \Irr(\mathrm O(V_{2n}))$. 
		\item For each $c^\prime \in F^\times$ and Whittaker datum $\mathfrak W_{c^\prime}$, there exists a canonical bijection
		\begin{align}\label{51}
		\mathcal J_{\mathfrak W_{c^\prime}} : \Pi_{\phi} \longrightarrow \widehat {\mathcal {S}_{\phi}}. 
		\end{align}
		We write $\pi=\pi_{\mathfrak W_{c^\prime}}(\phi,\eta)$ if $\pi\in  \Pi_{\phi}$ corresponds to $\eta \in \widehat{\mathcal {S}_{\phi}}$ under $\mathcal J_{\mathfrak W_{c^\prime}}$.
		\item Assume that $\phi$ is a tempered $L$-parameter of $\mathrm O(V_{2n})$. Then for each Whittaker datum $\mathfrak W_{c^\prime}$, 
		\begin{itemize}
			\item there is a unique $\mathfrak W^+_{c^\prime}$-generic representation $\pi$ in $\Pi_{\phi}$, which corresponds to the trivial character of $\mathcal {S}_{\phi}$ under $\mathcal J_{\mathfrak W_{c^\prime}}(\pi)$; 
			\item there is a unique $\mathfrak W^{-}_{c^\prime}$-generic representation $\pi$ in $\Pi_{\phi}$, which corresponds to $\kappa_{\phi}$ under $\mathcal J_{\mathfrak W_{c^\prime}}(\pi)$. Here $\kappa_{\phi}$ is defined in (\ref{kappa}). 
		\end{itemize}
		\item 
		Let $\mathfrak W_{c_1}$ and $\mathfrak W_{c_2}$ be two Whittaker data. Then for any $\pi\in \Pi_{\phi}$, we have 
		$$
		\mathcal J_{\mathfrak W_{c_2}}(\pi)=\mathcal J_{\mathfrak W_{c_1}}(\pi)\otimes \eta_{\phi\chi_{V},c_2/c_1},
		$$
		where $\phi\chi_{V}=\phi\otimes\chi_{V}$ and $\eta_{\phi\chi_{V},c_2/c_1}$ is defined in (\ref{eta}). 
		\item
		The map $\mathcal L$ preserves temperedness and discreteness, i.e., $\pi\in \Pi_{\phi}$ is a tempered representation if and only if $\phi$ is a tempered parameter and $\pi\in \Pi_{\phi}$ is a discrete series representation if and only if $\phi$ is a discrete parameter.
		\item
		For each Whittaker datum $\mathfrak W_{c^\prime}$, $\pi$ is a representation of $\mathrm O(V_{2n}^+)$ if and only if
		$$\mathcal J_{\mathfrak W_{c^\prime}}(\pi)(z_\phi)=  \chi_{V}(c^\prime/c).$$
		\item
		The following are equivalent:
		\begin{itemize}
			\item 	$\phi \in \Phi^{\epsilon}\left(\mathrm{O}\left(V_{2 n}\right)\right)$;
			\item some $\pi\in \Pi_{\phi}$ satisfies $\pi \otimes \det \neq \pi$;
			\item all $\pi\in \Pi_{\phi}$ satisfy $\pi \otimes\det 
			\neq \pi$. 
		\end{itemize}
	    \item
		 For $\pi\in \Pi_{\phi}$, the determinant twist $\pi\otimes \det$ also belongs to $\Pi_{\phi}$, and 
		$$\mathcal J_{\mathfrak W_{c^\prime}}(\pi \otimes \det)=\mathcal J_{\mathfrak W_{c^\prime}}(\pi)\otimes  \kappa_{\phi}.$$
		\item
		\textbf{(Local intertwining relation)}
		Assume that $\phi=\phi_{\tau}+\phi_{0}+\phi_{\tau}^{\vee}$, where $\phi_{0}$ is an element in $\Phi_{\mathrm {temp}}\left(\mathrm O(V_{2n_0})\right)$
		and $\phi_{\tau}$ is an irreducible tempered representation of $WD_{F}$ corresponding to $\tau \in \Irr(\GL_{k}(F))$, so there is a natural embedding $S_{\phi_{0}}\hookrightarrow \mathcal {S}_{\phi}$. Let $P$ be a parabolic subgroup of $\mathrm O(V_{2n}^\delta)$ with Levi subgroup $$M_{P}\cong \GL_{k}(F) \times \mathrm O(V_{2n_0}^\delta)$$ and $\pi_{0}$ be the irreducible tempered representation of $\mathrm O(V_{2n_0}^\delta)$ corresponding to $(\phi_{0},\eta_0)$ under $\mathcal L$ and $\mathcal J_{\mathfrak W_{\psi,c}}$. Then the induced representation $\Ind_{P}^{\mathrm O(V_{2n}^\delta)}\left(\tau \otimes \pi_{0}\right)$ has a decomposition 
		$$
		\Ind_{P}^{\mathrm O(V_{2n}^\delta)}\left(\tau \otimes \pi_{0}\right)=\bigoplus_{\eta}\pi_{\mathfrak W_{c^\prime}}(\phi,\eta),
		$$
		where the sum runs over all $\widehat{\mathcal {S}_{\phi}}$ such that $\eta|_{\mathcal S_{\phi_0}}=\eta_0$. Moreover if $\phi_{\tau}$ is self-dual and of orthogonal type, let 
		$$
		R_{\mathfrak W_{c^\prime}}(w, \tau \otimes \pi_0)
		\in \End_{\mathrm O(V_{2n}^\delta)}
		\left(\Ind_{P}^{\mathrm O(V_{2n}^\delta)}(\tau\otimes\pi_0)\right)$$ be the normalized intertwining operator associated to the Whittaker datum $\mathfrak W_{c^\prime}$ (see Section \ref{normalizingintertwing} below), where $w$ is the unique non-trivial element in the relative Wely group for $M_P$. Then
		$$R_{\mathfrak W_{c^\prime}}(w, \tau \otimes \pi_0) |_{\pi}=\mathcal J_{\mathfrak W_{c^\prime}}(\pi)(a),$$
		where $a \in \mathcal {S}_{\phi}$ corresponds $\phi_{\tau}$. 
	\item \textbf{(Compatibility with Langlands quotients)}
	Assume that
	$$\phi=\phi_1|\cdot|_{F}^{s_{1}}+\cdots+ \phi_{r}|\cdot|_{F}^{s_{r}}+\phi_{0}+ \phi_{r}^\vee |\cdot|_{F}^{-s_{r}}+ \cdots+ \phi_1^\vee|\cdot|_{F}^{-s_{1}},$$ 
	where 
	\begin{itemize}
		\item $\phi_i$ is an irreducible tempered representation of $\WD_F$ of dimension $d_i$; 
		\item $\phi_0\in \Phi_{\mathrm{temp}}(\mathrm O(V_{2n_0}))$;
		\item $s_1\geq \cdots \geq s_r>0$;
		\item $d_1+\cdots +d_r+n_0=n$. 
	\end{itemize}
	We denote by $\tau_{i}$ the irreducible tempered representation of $\GL_{d_i}(F)$ corresponding to $\phi_i$. Then the $L$-packet $\Pi_{\phi}$ consists of the unique irreducible quotients $\pi$ of the standard modules $$\Ind_{P}^{\mathrm O(V_{2n}^\delta)}\left(\tau_{1}|\cdot|_{F}^{s_{1}} \otimes \cdots \otimes \tau_{r}|\cdot|_{F}^{s_{r}} \otimes \pi_{0}\right),$$
	where $P$ is a parabolic subgroup of $\mathrm O(V_{2n}^\delta)$ with Levi subgroup $$M_{P}\cong \GL_{k_{1}}(F) \times \cdots \times \GL_{k_{r}}(F) \times \mathrm O(V_{2n_0}^\delta)$$ and $\pi_{0}$ runs over elements of $\Pi_{\phi_{0}}$. Moreover, the natural embedding $\mathcal S_{\phi_0}\hookrightarrow  \mathcal {S}_{\phi}$ is an isomorphism and $$\mathcal J_{\mathfrak W_{c^\prime}}(\pi)=\mathcal J_{\mathfrak W_{c^\prime}}(\pi_0)$$ if we identify $\mathcal S_{\phi_0}$ with $\mathcal {S}_{\phi}$ via the above isomorphism. 
	\item
		The map respects standard $\gamma$-factors. Namely, we have 
		$$
		\gamma(s,\pi,\chi,\psi)=\gamma(s,\phi\otimes\chi,\psi)
		$$
		for $\pi\in \Pi_{\phi}$ and any character $\chi$ of $F^{\times}$. Here $\gamma(s,\pi,\chi,\psi)$ is the standard $\gamma$-factor defined by Lapid-Rallis \cite{MR2192828} and $\gamma(s,\phi\otimes\chi,\psi)$ is the $\gamma$-factor defined in \cite{MR546607}.  
		\item 
		The map respects Plancherel measures. Namely, we have 
		\begin{align*}
		\mu_{\psi}(\tau_s\otimes \pi)&=\gamma(s,\phi_{\tau}\otimes \phi^{ \vee},\psi)\times \gamma(-s,\phi_{\tau}^{\vee}\otimes\phi, \psi_{-1} )\\
		& \times \gamma (2s, \wedge^2\circ \phi_{\tau}, \psi)\times \gamma(-2s, \wedge^{2} \circ \phi_{\tau}^{\vee},\psi_{-1})
		\end{align*}
		for any $\pi\in \Pi_{\phi}$ and $\tau\in \Irr (\GL_k(F))$ with $L$-parameter $\phi_{\tau}$. 
	\end{enumerate}
\end{theorem}

The LLC for quasi-split even orthogonal groups has been proved by Arthur\cite{MR3135650} and was explicated in Atobe--Gan\cite{MR3708200}. More precisely, they proved
\begin{theorem}\label{Arthurorth}
	Fix $(d,c)\in (F^\times)^2$. Let $V_{2n}=V_{2n}^+$ be the orthogonal space associated to $(d,c)$, and $\chi_V=(\cdot, d)_F$ be the discriminant character of $V_{2n}$. Put $E=F(\sqrt d)$. 
%	Fix $d\in F^\times /F^{\times 2}$. Let $\chi=(\cdot, d)_F$ be the character of $F^\times$ associated to $E=F(\sqrt d)$. Let $V^+_{2n}=V_{2n}$ be the orthogonal space associated to $(d,1)$. 
	\begin{enumerate}[(1).]
		\item There exists a surjective map
		$$
		\mathcal L^{A}: \Irr\left(\mathrm O(V_{2n}^{+})\right) \longrightarrow \Phi\left(\mathrm O(V_{2n})\right),$$
		which is finite-to-one. For $\phi \in \Phi(\mathrm O(V_{2n})),$ we denote the inverse image of $\phi$ by $\Pi_{\phi}(\mathrm O(V^+_{2n}))$. 
		\item For any $c^\prime \in c N_{E/F}(E^\times)$, there exists a canonical bijection
		$$\mathcal J_{\mathfrak W_{c^\prime}}^A: \Pi_{\phi}(\mathrm O(V_{2n}^+)) \longrightarrow \widehat{\bar {\mathcal {S}}_{\phi}}.$$
	\end{enumerate}  
	 Moreover, the maps $\mathcal L^A$ and $\mathcal J_{\mathfrak W_{c^\prime}}^A$ satisfy all the other properties in Theorem  \ref{desideratumall}. 
\end{theorem}
\begin{remark}
\begin{enumerate}
	\item M\oe glin \cite[\S 1.4 Theorem 1.4.1]{MR2767522} and M\oe glin-Renard \cite{MR3839702} have partially extended Theorem \ref{Arthurorth} to pure inner forms as well, though we are not sure if all the statements in Theorem \ref{desideratumall} were verified in their work.
	\item For a fixed pair $(d,c)\in (F^\times)^2$, if $d\notin F^{\times 2}$, then both $\mathrm O(V_{2n}^+)$ and $\mathrm O(V_{2n}^-)$ are quasi-split. Indeed, we may take $V_{2n}^-=c_0\cdot  V_{2n}^+$ for a $c_0\notin N_{E/F}(E^\times)$, where $c_0\cdot V_{2n}^+=V_{2n}^+$ as a vector space and 
\begin{align*}
\langle x, y\rangle_{c_0\cdot V^+_{2n}}= c_0\cdot \langle x,y\rangle_{V_{2n}^+}. 
\end{align*}
As subgroups of $\GL(V_{2n}^+)$, we have an identification 
\begin{align*}
\id: \mathrm O(V_{2n}^+)=\mathrm O(V_{2n}^-). 
\end{align*}
This induces a bijection 
\begin{align*}
\id^*: \Irr \left(\mathrm O(V_{2n}^-)\right)\rightarrow  \Irr \left(\mathrm O(V_{2n}^+)\right). 
\end{align*}
Under this identification, we can extend the map $\mathcal L^A$ to $\Irr\left(\mathrm O(V_{2n}^{-})\right)$ by defining 
	\begin{align*}
	\mathcal L^A: \Irr\left(\mathrm O(V_{2n}^{-})\right)&\longrightarrow \Phi(\mathrm O(V_{2n}))\\
\pi&\mapsto\mathcal L^A(\pi)\coloneqq \mathcal L^A(\id^*(\pi)). 
	\end{align*}
	We can also extend the map $\mathcal J_{\mathfrak W_{c^\prime}}^A$ as follows: 
	\begin{itemize}
		\item If $c^\prime \in c N_{E/F}(E^\times)$, we extend the map $\mathcal J_{\mathfrak W_{c^\prime}}^A$ to $\Irr\left(\mathrm O(V_{2n}^{-})\right)$ by defining 
		$$\mathcal J_{\mathfrak W_{c^\prime}}^A(\pi)= \mathcal J_{\mathfrak W_{c^\prime}}^A(\id^*(\pi))\otimes \eta_{\phi\chi_{V},c_0}.$$
		\item If $c^\prime \notin c N_{E/F}(E^\times)$, we define the map $\mathcal J_{\mathfrak W_{c^\prime}}^A$ by  
		\begin{align*}
		\mathcal J_{\mathfrak W_{c^\prime}}^A(\pi)\coloneqq \mathcal J_{\mathfrak W_{c}}^A(\pi)\otimes \eta_{\phi\chi_{V},c^\prime/c}(\pi)\quad \mbox{for}\,\, \pi\in \bigsqcup_{\delta\in \{\pm 1\}} \Irr\left(\mathrm O(V_{2n}^{\delta})\right).
		\end{align*} 
	\end{itemize}
	One can check that these extensions satisfy all the properties in Theorem \ref{desideratumall}. Hence when $d\notin F^{\times 2}$, Theorem \ref{desideratumall} follows from Theorem \ref{Arthurorth} and the constructions above. But our proof for Theorem \ref{desideratumall} in later sections will not use these constructions. In particular, it provides another way to prove Theorem \ref{desideratumall} for the case when $d\notin F^{\times 2}$.
\end{enumerate}	
\end{remark}
The following sections are devoted to proving Theorem \ref{desideratumall}. As mentioned in the introduction, we will use Theorem \ref{llcsympletic} and the local theta correspondence to construct the maps $\mathcal L$ and $\mathcal J_{\mathfrak W_{c^\prime}}$ in Theorem \ref{desideratumall}. We shall prove the maps $\mathcal L$ and $\mathcal J_{\mathfrak W_{c^\prime}}$ constructed by us satisfy all the properties in Theorem \ref{desideratumall}. In Section 9, we will also show that our classification coincides with Arthur's in the quasi-split case, i.e.,  
\begin{align*}
\mathcal L|_{\Irr\left(\mathrm O(V_{2n}^{+})\right)} &= \mathcal L^A|_{\Irr\left(\mathrm O(V_{2n}^{+})\right)},\\
\mathcal J_{\mathfrak W_{c^\prime}}|_{\Pi_{\phi}(\mathrm O(V_{2n}^+))}&=\mathcal J_{\mathfrak W_{c^\prime}}^A|_{\Pi_{\phi}(\mathrm O(V_{2n}^+))} \,\, \mbox{for}\,\,c^\prime \in c N_{E / F}(E^{ \times}).
\end{align*}

\section{Construction}\label{construction}
We construct the LLC for even orthogonal groups in this section. We first construct the correspondence for tempered representations and then extend our construction to non-tempered representations by using the Langlands classification for $p$-adic groups. Several properties in Theorem \ref{desideratumall} will be proved along the way. 
\subsection{Construction of $\mathcal L_{\psi}$}\label{constructionparameter}
Fix $(d,c)\in (F^\times) ^2$. Let $V_{2n}=V_{2n}^+$ be the $2n$-dimensional orthogonal space associated to $(d,c)$, and $\chi_{V}=(\cdot, d)_F$ be the discriminant character of $V_{2n}$. Fix a non-trivial character $\psi$ of $F$. In this subsection, we will construct a map 
$$
\mathcal L_{\psi}: \bigsqcup_{\delta\in \{\pm 1\}} \Irrt \left(\mathrm O(V_{2n}^\delta)\right) \longrightarrow \Para(\mathrm O(V_{2n})).$$
Later in Section \ref{changeofpsi}, we will show that $\mathcal L_\psi$ is independent of the choice of $\psi$. So we get our desired $\mathcal L$.  

For any $\pi\in \Irrt \mathrm O(V^\delta_{2n})$, consider the following two representations

\begin{equation}\label{con}
\begin{cases}
\sigma_1\coloneqq\theta_{W_{2n}, V^\delta_{2n} ,\psi}(\pi)\quad &\mbox{of}\,\, \SP(W_{2n});\\
\sigma_2\coloneqq\theta_{W_{2n-2}, V^\delta_{2n}, \psi}(\pi\otimes \det)\quad &\mbox{of}\,\, \SP(W_{2n-2}).
\end{cases}
\end{equation}
By the conservation relation (Theorem \ref{con2}), exactly one of $\sigma_{i}$ is non-vanishing. We shall attach an $L$-parameter to $\pi$ in terms of the $L$-parameter of $\sigma_{i}$.

$\bullet$ \underline{Case I}: If $\sigma_1=\theta_{W_{2n}, V^\delta_{2n}, \psi}(\pi)\neq 0$, then we have: 
	\begin{lemma}
		Let $\phi^+$ be the $L$-parameter of $\sigma_1=\theta_{W_{2n}, V^\delta_{2n}, \psi}(\pi)$. Then $\chi_{V}\subseteq \phi^+$. 
	\end{lemma}
	\begin{proof}
		By Lemma \ref{pole}, we know that $\gamma(s,\sigma_1, \chi_V, \psi)$ has a pole at $s=1$. On the other hand, it follows from Theorem \ref{llcsympletic} (8) that 
		$$\gamma(s,\sigma_1, \chi_V, \psi)= \gamma(s,\phi^+\otimes\chi_V,\psi).$$ 
		So $\gamma(s,\phi^+ \otimes\chi_V,\psi)$ has a pole at $s=1$. Note that  
		\begin{align*}
		\gamma(s,\phi^+ \otimes\chi_V,\psi)=&\varepsilon(s,\phi^+\otimes \chi_V,\psi)\times \frac{L(1-s,(\phi^+ \otimes\chi_V)^\vee)}{L(s,\phi^+\otimes\chi_V)}\\
		=&\varepsilon(s,\phi^+\otimes \chi_V,\psi)\times \frac{L(1-s,\phi^+ \otimes\chi_V)}{L(s,\phi^+\otimes\chi_V)}.
		\end{align*}
		Since $\varepsilon(s,\phi^+ \otimes\chi_V,\psi)$ is holomorphic at $s=1$ and $L(s,\phi^+\otimes \chi_V)$ is non-zero at $s=1$, we deduce that $L(1-s,\phi^+ \otimes\chi_V)$ has a pole at $s=1$. Since $\phi^+\otimes\chi_{V}$ is tempered, this implies that $\phi^+\otimes\chi_V$ contains the trivial representation, which is equivalent to saying that $\phi^+$ contains $\chi_{V}$.
	\end{proof}
	In this case, we define the $L$-parameter of $\pi$ to be the unique orthogonal representation $\phi$ of $\WD_F$ such that 
	$$\phi^+=(\phi\otimes\chi_V) +\chi_V.$$
	
$\bullet$ \underline{Case II}: If $\sigma_2=\theta_{W_{2n-2}, V^\delta_{2n}, \psi}(\pi\otimes \det)\neq 0 $, we simply define the $L$-parameter of $\pi$ to be 
	$$
	\phi=(\phi^{-} \otimes\chi_{V})+ \mathbbm{1},
	$$
	where $\phi^{-}$ is the $L$-parameter of $\sigma_2$ and $\mathbbm{1}$ is the trivial representation of $\WD_F$. 

Note that we have $\phi\in \Phi(\mathrm O(V_{2n}))$ in both cases. Moreover, it follows from Lemma \ref{temperedtotempered} and Theoreom \ref{llcsympletic} (5) that both $\phi^+$ and $\phi^-$ are tempered parameters, so $\phi\in \Para(\mathrm O(V_{2n}))$ in both cases. Combining these two cases, we have defined a map
$$
\mathcal L_{\psi}: \bigsqcup_{\delta\in \{\pm 1\}} \Irrt \left(\mathrm O(V_{2n}^\delta)\right) \longrightarrow \Para(\mathrm O(V_{2n})).$$
For a parameter $\phi\in \Para(\mathrm O(V_{2n})) $, we define the packet $\Pi_{\phi,\psi}$ to be the fiber $\mathcal L_{\psi}^{-1}(\phi)$ and $\Pi_{\phi,\psi}(\mathrm O(V_{2n}^{\delta}))= \Pi_{\phi,\psi}\cap \Irrt \mathrm O(V_{2n}^{\delta})$.

\subsection{Local factors}
So far we have associated a tempered $L$-parameter $\phi$ to every  $\pi\in \Irrt(\mathrm O(V^\delta_{2n}))$. We then show this assignment respects the standard $\gamma$-factors and the Plancherel measures. These imply Theorem \ref{desideratumall} (11) and (12) in the tempered case.

\begin{lemma}\label{respectgamma}
	Let $\pi\in \Irrt \mathrm O(V^\delta_{2n})$ and $\mathcal L_{\psi}(\pi)=\phi$. For any character $\chi$ of $F^{\times}$, we have $\gamma(s,\pi, \chi,\psi)=\gamma(s,\phi\otimes \chi, \psi)$. 
\end{lemma}
\begin{proof}
According to our construction, we need to consider the following two cases: 

$\bullet$ \underline{Case I}: Assume $\sigma_1=\theta_{W_{2n}, V^\delta_{2n}, \psi}(\pi)\neq 0$. It follows from Lemma \ref{comgamma} that  
$$
\frac{\gamma(s,\sigma_1,\chi,\psi)}{\gamma(s,\pi,\chi\chi_V,\psi)}= \gamma(s,\chi\chi_V,\psi).
$$
By Theorem \ref{llcsympletic} (8), we have 
$$
\gamma(s,\sigma_1,\chi,\psi)=\gamma(s,\phi^+\otimes \chi ,\psi),
$$
where $\phi^+=(\phi\otimes\chi_{V})\oplus \chi_{V}$ is the $L$-parameter of $\sigma_1$. So 
\begin{align*}
\gamma(s,\pi,\chi,\psi)&=\frac{\gamma(s,\sigma_1,\chi\chi_{V}^{-1},\psi)}{\gamma(s,(\chi\chi_{V}^{-1})\chi_V,\psi)}\\
&=\frac{\gamma(s,\phi^+\otimes\chi \chi_{V}^{-1} ,\psi)}{\gamma(s,\chi,\psi)}\\
&=\gamma(s,\phi \otimes\chi ,\psi).
\end{align*}

$\bullet$ \underline{Case II}: Assume $\sigma_2=\theta_{W_{2n-2}, V^\delta_{2n}, \psi}(\pi\otimes \det)\neq 0$. This follows from a similar calculation in Case I and the fact that
\begin{align*}
\gamma(s,\pi, \chi,\psi)= \gamma(s,\pi\otimes\det, \chi,\psi).
\end{align*}
These complete the proof.  
\end{proof}

\begin{lemma}\label{respectplancherel}
Let $\pi\in \Irrt \mathrm O(V^\delta_{2n})$ and $\mathcal L_{\psi}(\pi)=\phi$. For any irreducible smooth representation $\tau$ of $\GL_k(F)$ with $L$-parameter $\phi_{\tau}$, we have   
$$
\begin{aligned} \mu_{\psi}\left(\tau_{s} \otimes \pi\right) &=\gamma\left(s, \phi_{\tau} \otimes \phi^{\vee}, \psi\right) \times \gamma\left(-s, \phi_{\tau}^{\vee} \otimes \phi, \psi_{-1}\right) \\
 &\times \gamma\left(2 s, \wedge^{2} \circ \phi_{\tau}, \psi\right) \times \gamma\left(-2 s, \wedge^{2} \circ \phi_{\tau}^{\vee}, \psi_{-1}\right) .\end{aligned}
$$
\end{lemma}
\begin{proof}
Similarly to the proof of Lemma \ref{respectgamma}, we need to consider the following two cases: 

$\bullet$ \underline{Case I}: If $\sigma_1=\theta_{W_{2n}, V^\delta_{2n}, \psi}(\pi)\neq 0$, then it follows from Lemma \ref{complan} that 
$$
\frac{\mu_{\psi}(\tau_s\chi_{V}\otimes\sigma_1)}{\mu_{\psi}(\tau_s\otimes\pi)}=\gamma(s,\tau,\psi)\times \gamma(-s,\tau^{\vee},\psi_{-1}).
$$ 
Let $\phi^+$ be the $L$-parameter for $\sigma_1$. Then by Theorem \ref{llcsympletic} (9), we have 
\begin{align*}
\mu_{\psi}(\tau_s\chi_{V}\otimes \sigma)&=\gamma(s,\phi_{\tau}\chi_{V}\otimes (\phi^+)^ {\vee},\psi)\times \gamma(-s, \phi_{\tau}^{\vee} \chi_{V}^{-1}\otimes\phi^+, \psi_{-1} )\\
& \times \gamma (2s, \wedge^2\circ (\phi_{\tau}\chi_{V}), \psi)\times \gamma(-2s, \wedge^{2} \circ (\phi_{\tau}^{\vee} \chi_{V}^{-1}),\psi_{-1}).
\end{align*}
So 
\begin{align*}
\mu_{\psi}(\tau_s\otimes \pi)&= \mu_{\psi}(\tau_s\chi_{V}\otimes \sigma)\times \gamma(s,\tau,\psi)^{-1}\times \gamma(-s,
\tau^{\vee},\psi_{-1})^{-1}\\ 
&=\gamma(s,\phi_{\tau}\chi_{V} \otimes (\phi^+)^{\vee},\psi)\times \gamma(-s, \phi_{\tau}^{\vee}\chi_{V}^{-1} \otimes \phi^+, \psi_{-1} )\\
&  \times \gamma (2s, \wedge^2\circ (\phi_{\tau}\chi_{V}), \psi)\times  \gamma(-2s, \wedge^{2} \circ (\phi_{\tau}^{\vee} \chi_{V}^{-1}),\psi_{-1})\\
& \times \gamma(s,\phi_{\tau},\psi)^{-1}\times \gamma(-s,
\phi_{\tau}^{\vee},\psi_{-1})^{-1} \\ 
&= \gamma(s,\phi_{\tau}\otimes\phi^{\vee},\psi)\times \gamma(-s, \phi_{\tau}^{\vee}\otimes \phi, \psi_{-1} )\\
&  \times \gamma (2s, \wedge^2\circ \phi_{\tau}, \psi)\times \gamma(-2s, \wedge^{2} \circ \phi_{\tau}^{\vee},\psi_{-1}).
\end{align*}
Here we use the fact that $\phi^+=\phi\otimes\chi_{V}+\chi_{V}$. 

$\bullet$ \underline{Case II}: If $\sigma_2=\theta_{W_{2n-2}, V^\delta_{2n}, \psi}(\pi\otimes \det)\neq 0 $, then this follows from a similar calculation to Case I and the fact that 
$$\mu_{\psi}(\tau_s \otimes \pi)= \mu_{\psi}(\tau_s \otimes (\pi\otimes \det)).$$
Readers may also consult \cite[Lemma B.1]{MR3166215} for the proof of this fact. 
\end{proof}

\subsection{Counting sizes of packets}
Our next goal is to attach a character of the component group to each $\pi\in \Irrt(\mathrm O(V_{2n}^\delta))$. To do this, we need some preparations. In this subsection we consider the behaviour of $L$-parameters under local theta correspondence and count the sizes of $L$-packets. We emphasize that when we talk about representations of symplectic groups, the $L$-parameter of a representation is in the sense of Theorem \ref{llcsympletic}, whereas when we talk about representations of even orthogonal groups, the $L$-parameter of a tempered representation is in the sense of $\mathcal L_\psi$ constructed in Subsection \ref{constructionparameter}.

The following lemma will be used later. 
\begin{lemma}\label{plancherel}
	Let $G=\mathrm O(V^\delta_{2n})$ (resp. $G=\SP(W_{2n})$) and $\pi_1$ and $\pi_2$ (resp. $\sigma_1$ and $\sigma_2$) be two irreducible tempered representations of $G$ with $\mathcal L_{\psi}(\pi_1)=\phi_1, \mathcal L_{\psi}(\pi_2)=\phi_2$ (resp. $\mathcal L(\sigma_1)=\phi_1, \mathcal L(\sigma_2)=\phi_2$). Assume that 
	$\mu_{\psi}(\tau_{s}\otimes \pi_1)=\mu_{\psi}(\tau_s\otimes \pi_2)$ (resp. $ 
    \mu_{\psi}(\tau_{s}\otimes \sigma_1)=\mu_{\psi}(\tau_s\otimes \sigma_2)$) for all $k\geq 1$ and all irreducible smooth representation $\tau$ of $\GL_{k}(F)$. Then we have 
	$$
	\phi_1=\phi_2. 
	$$
\end{lemma}
\begin{proof}
	It follows from Theorem \ref{llcsympletic} (9) and Lemma \ref{respectplancherel} that the maps $\mathcal L_\psi$ and $\mathcal L$ respect the Plancherel measures, so we have  
	$$
    \gamma\left(s, \phi_{\tau} \otimes \phi_1^{\vee}, \psi\right) \times \gamma\left(-s, \phi_{\tau}^{\vee} \otimes \phi_1, \psi_{-1}\right) =
     \gamma\left(s, \phi_{\tau} \otimes \phi_2^{\vee}, \psi\right) \times \gamma\left(-s, \phi_{\tau}^{\vee} \otimes \phi_2, \psi_{-1}\right) 
	$$
	for any irreducible smooth representations $\tau$ of $\GL_k(F)$ with $L$-parameter $\phi_\tau$. Then it follows from Lemma \ref{gammadetermine} that $\phi_1=\phi_2$.  
\end{proof}

Next we prove the $L$-parameters we attached to $\pi$ and $\pi\otimes \det$ are the same. This proves the first part of Theorem \ref{desideratumall} (8) for tempered representations. 
\begin{proposition}\label{det}
Let $\pi\in \Irrt(\mathrm O(V^\delta_{2n}))$. Then we have  
\begin{align*}
\mathcal L_{\psi}(\pi)=\mathcal L_{\psi}(\pi\otimes\det). 
\end{align*}
\end{proposition}
\begin{proof}
	By \cite[Lemma B.1]{MR3166215}, we know that  
	$$\mu_{\psi}(\tau_s \otimes \pi)= \mu_{\psi}(\tau_s \otimes (\pi\otimes \det))$$
	for any $k\geq 1$ and any irreducible smooth representations $\tau$ of $\GL_k(F)$. Then this proposition follows from Lemma \ref{plancherel}.
\end{proof}

The next proposition describes the behaviour of $L$-parameters under the local theta correspondence. 
\begin{proposition}\label{prasad} 
\begin{enumerate}[(i)]
	\item Let $\pi \in \Irrt \left( \mathrm O(V^\delta_{2n})\right)$ and $\mathcal L_{\psi}(\pi)=\phi$. 
	\begin{itemize}
		\item If $\sigma\coloneqq\theta_{W_{2n}, V^\delta_{2n},  \psi}(\pi)\neq 0$, then 
		$$\mathcal L(\sigma)=(\phi\otimes \chi_V)+ \chi_{V}.$$
		\item If $\sigma\coloneqq\theta_{W_{2n}, V^\delta_{2n},  \psi}(\pi\otimes\det )\neq 0$, then 
		$$\mathcal L(\sigma)=(\phi\otimes \chi_V)+ \chi_{V}.$$
	\end{itemize}
	\item Let $\sigma\in \Irrt \left( \SP(W_{2n-2})\right)$ and $\mathcal L(\sigma)=\phi^-$. If $\pi\coloneqq\theta_{V^\delta_{2n},W_{2n-2},\psi}(\sigma)\neq 0$, then  $$\mathcal L_{\psi}(\pi)=(\phi^-\otimes\chi_{V})+ \mathbbm 1.$$
\end{enumerate}
\end{proposition}
\begin{proof}
We first prove (i). By Proposition \ref{det}, we have 
$$\mathcal L_{\psi}(\pi\otimes\det)=\mathcal L_\psi(\pi)=\phi.$$ Then (i) follows from our construction of $\mathcal L_{\psi}$. The proof for (ii) is similar.
\end{proof}

The following corollary is a consequence of Proposition \ref{prasad}.  
\begin{corollary}\label{1inphi}
	\begin{enumerate}[(i)]
		\item Let $\pi \in \Irrt \left(\mathrm O(V^\delta_{2n})\right)$ and $\mathcal L_{\psi}(\pi)=\phi$. If $\mathbbm 1 \nsubseteq \phi$, then 
		\begin{align*}
			\theta_{W_{2n-2}, V_{2n}^{\delta} ,\psi}(\pi)= \theta_{W_{2n-2}, V_{2n}^{\delta} ,\psi}(\pi\otimes\det)=0.
		\end{align*}
			So by the conservation relation (Theorem \ref{con2}), both $\theta_{W_{2n}, V_{2n}^{\delta},\psi}(\pi)$ and $\theta_{W_{2n}, V_{2n}^{\delta}, \psi}(\pi\otimes \det)$ are non-zero. 
	\item Let $\sigma\in \Irrt \SP(W_{2n})$ and $\mathcal L(\sigma)=\phi^+$. If $\chi_{V}\nsubseteq \phi^+$, then
		\begin{align*}
	\theta_{V^{+}_{2n},W_{2n},\psi}(\sigma)=\theta_{V^{-}_{2n},W_{2n},\psi}(\sigma)=0.
	\end{align*}
	So by the conservation relation (Theorem \ref{con2}), both $\theta_{V^{+}_{2n+2},W_{2n},\psi}(\sigma)$ and $\theta_{V^{-}_{2n+2},W_{2n},\psi}(\sigma)$ are non-zero. 
	\end{enumerate}
\end{corollary}

\begin{proposition}\label{1notinphi}
	\begin{enumerate}[(i)]
		\item Let $\pi \in \Irrt \left(\mathrm O(V^\delta_{2n})\right)$ and $\mathcal L_{\psi}(\pi)=\phi$. If $\mathbbm 1 \subseteq \phi$, then exactly one of $\theta_{W_{2n-2}, V_{2n}^{\delta} ,\psi}(\pi)$ and $ \theta_{W_{2n-2}, V_{2n}^{\delta} ,\psi}(\pi\otimes\det)$ is non-zero. In particular, $\pi\ncong \pi\otimes\det$ in this case. 
		\item Let $\sigma\in \Irrt \left(\SP(W_{2n})\right)$ and $\mathcal L(\sigma)=\phi^+$. If $\chi_{V}\subseteq \phi^+$, then exactly one of $\theta_{V^{+}_{2n},W_{2n},\psi}(\sigma)$ and $\theta_{V^{-}_{2n},W_{2n},\psi}(\sigma)$ is non-zero.  
	\end{enumerate}
\end{proposition}
\begin{proof}
We first prove (i). Let $\phi^{+}=(\phi\otimes\chi_{V})\oplus\chi_{V}$ and $\phi^-$ be the unique $L$-parameter such that $\phi=(\phi^-\otimes \chi_V)\oplus 1$. We define a map:
	$$
	\theta_{\psi,2n}: \bigsqcup_{V_{2n}^\delta} \Irrt \left(\mathrm O(V_{2n}^{\delta})\right)\longrightarrow \Irrt \left(\SP(W_{2n})\right) 
	$$
	by 
	$$
	\theta_{\psi,2n}(\pi)=
	\begin{cases}
	\theta_{W_{2n}, V_{2n}^\delta, \psi}(\pi) \quad &\mbox{if}\,\,\theta_{W_{2n}, V_{2n}^\delta,  \psi}(\pi)\neq 0, \\
	\theta_{W_{2n}, V_{2n}^\delta,  \psi}(\pi\otimes\det) &\mbox{otherwise}.
	\end{cases}
	$$
	Then $\theta_{\psi,2n}(\pi)$ is well defined  by the conservation relation (Theorem \ref{con2}). By Proposition \ref{prasad}, the restriction of $\theta_{\psi,2n}$ to the packet $\Pi_{\phi,\psi}$ gives 
	\begin{equation}\label{theta2n}
	\theta_{\psi,2n}: \Pi_{\phi,\psi} \longrightarrow \Pi_{\phi^{+}}.
	\end{equation}
	It follows from the Howe duality that
	\begin{align}\label{15}
	\theta_{\psi,2n}(\pi_1)\neq \theta_{\psi,2n}(\pi_2)\quad  \mbox{if}\,\, \pi_1\neq \pi_2\,\, \mbox{and}\,\, \pi_1 \neq \pi_2\otimes \det.
	\end{align}
	So for every $\sigma \in \Pi_{\phi^{+}}$, the fibre $\theta_{\psi,2n}^{-1}(\sigma)$ contains at most two elements. Hence 
\begin{align}\label{1111}
	|\Pi_{\phi,\psi}|\leq 2|\Pi_{\phi^{+}}|. 
\end{align}

	Next we shall consider two different cases, depending on $\mathbbm 1\oplus\mathbbm 1
	\subseteq \phi$ or not.  

$\bullet$ \underline{Case I}: If $\mathbbm 1\oplus\mathbbm 1 \subseteq \phi$, then $\chi_{V}\subseteq \phi^{-}$. Define
		\begin{align*}
		\theta_{\psi,2n-2}: \Irrt \left(\SP(W_{2n-2})\right) \longrightarrow  \bigsqcup_{V_{2n}^\delta} \Irrt \left(\mathrm O(V_{2n}^{\delta})\right)
		\end{align*} 
		by 
		$$
		\theta_{\psi,2n-2}(\sigma)= 
		\begin{cases}
		\theta_{V^{+}_{2n},W_{2n-2},\psi}(\sigma) \quad &\mbox{if}\,\, \theta_{V^{+}_{2n},W_{2n-2},\psi}(\sigma)\neq 0,\\
		\theta_{V^{-}_{2n},W_{2n-2},\psi}(\sigma) &\mbox{otherwise}.
		\end{cases}
		$$
		Then $\theta_{\psi,2n-2}(\sigma)$ is well defined  by the conservation relation (Theorem \ref{con2}). By Proposition \ref{prasad}, the restriction of $\theta_{\psi,2n-2}$ to the packet $\Pi_{\phi^{-}}$ gives  
		\begin{equation}\label{theta2n-2}
		\theta_{\psi,2n-2}: \Pi_{\phi^{-}}\longrightarrow \Pi_{\phi,\psi}.
		\end{equation}
		Moreover, it follows from the Howe duality that 
			\begin{equation}\label{141}
			\theta_{\psi,2n-2}(\sigma_1)\neq \theta_{\psi,2n-2}(\sigma_2) \quad \mbox{if $\sigma_1\neq\sigma_2$},
			\end{equation} 
		and from the conservation relation (Theorem \ref{con2}) that
		\begin{equation}\label{14}
		\theta_{\psi,2n-2}(\sigma_1)\neq \theta_{\psi,2n-2}(\sigma_2)\otimes \det\quad \mbox{for all  $\sigma_1,\sigma_2$}.
		\end{equation} 
%		So the image of $\Pi_{\phi^{-}}$ under $\theta_{\psi,2n-2}$ is at most half size of $\Pi_{\phi,\psi}$, i.e., 
Hence 
\begin{align}\label{1112}
	|\Pi_{\phi,\psi}|\geq 2|\Pi_{\phi^{-}}|. 
\end{align}		
\begin{comment}		
		Composing (\ref{theta2n}) and (\ref{theta2n-2}), we get 
		$$
		\theta_{\psi,2n}\circ \theta_{\psi,2n-2}: \Pi_{\phi^{-}} \longrightarrow \Pi_{\phi^{+}}. 
		$$
		It follows from (\ref{15}) and (\ref{14}) that $\theta_{\psi,2n}\circ \theta_{\psi,2n-2}$ is injective. 
\end{comment}		
On the other hand, since $\mathbbm 1\oplus\mathbbm 1\subseteq \phi$, the inclusion map $\mathcal S_{\phi^{-}}\hookrightarrow \mathcal S_{\phi^+}$ is an isomorphism. Then it follows from Theorem \ref{llcsympletic} (2) that 
\begin{align}\label{135}
		|\Pi_{\phi^{-}}|=|\widehat{\bar {\mathcal S}_{\phi^-}}|=|\widehat{\bar {\mathcal S}_{\phi^+}}|=|\Pi_{\phi^{+}}|. 
\end{align}
%		Hence $\theta_{\psi,2n}\circ \theta_{\psi,2n-2}$ is a bijection. From the surjectivity of $\theta_{\psi,2n}$ and (\ref{15}), we deduce 
%		\begin{align*}
%		|\Pi_{\phi,\psi}|\leq 2|\Pi_{\phi^{+}}|. 
%		\end{align*}
%		On the other hand, it follows from (\ref{14}) that 
%		\begin{align*}
%		|\Pi_{\phi,\psi}|\geq 2|\Pi_{\phi^{-}}|. 
%		\end{align*}
Combining (\ref{1111}), (\ref{1112}) and (\ref{135}), we deduce  
		\begin{align*}
		|\Pi_{\phi,\psi}|=2|\Pi_{\phi^{+}}|=  2|\Pi_{\phi^{-}}|=2|\widehat{\bar {\mathcal S}_{\phi^{-}}}|=|\widehat{\mathcal {S}_{\phi}}|. 
		\end{align*}
		Hence for $\pi\in \Pi_{\phi,\psi}$, exactly one of $\pi$ and $\pi\otimes\det$ is in the image of $\theta_{\psi,2n-2}$. This proves Case I. 
	
	$\bullet$ \underline{Case II}: If $\mathbbm 1\oplus\mathbbm 1\nsubseteq \phi$, then $\chi_{V}\nsubseteq \phi^{-}$. We define a map 
		\begin{equation}\label{theta2n-22}
		\theta_{\psi,2n-2}^{+}\sqcup\theta_{\psi,2n-2}^{-}: \Pi_{\phi^{-}}\sqcup \Pi_{\phi^{-}} \longrightarrow \Pi_{\phi,\psi}
		\end{equation}
		by 
		$$
		\begin{cases*}
		\sigma\mapsto  \theta_{V^{+}_{2n},W_{2n-2},\psi}(\sigma)\quad \mbox{for $\sigma$ in the first copy of $\Pi_{\phi^{-}}$},\\
		\sigma\mapsto \theta_{V^{-}_{2n},W_{2n-2},\psi}(\sigma) \quad \mbox{for $\sigma$ in the second copy of $\Pi_{\phi^{-}}$}. 
		\end{cases*}
		$$
		Since $\chi_{V}\nsubseteq \phi^{-}$, both $\theta_{V^{+}_{2n},W_{2n-2},\psi}(\sigma)$ and $\theta_{V^{-}_{2n},W_{2n-2},\psi}(\sigma)$ are non-zero by Corollary \ref{1inphi}. Hence the map $\theta_{\psi,2n-2}^{+}\sqcup\theta_{\psi,2n-2}^{-}$ is well defined. Again, it follows from the Howe duality that 
		\begin{equation}\label{16}
	\begin{aligned}
	\left(\theta_{\psi,2n-2}^{+}\sqcup\theta_{\psi,2n-2}^{-}\right)(\sigma_1)\neq \left(\theta_{\psi,2n-2}^{+}\sqcup\theta_{\psi,2n-2}^{-}\right)(\sigma_2) \quad	\mbox{if}\,\, \sigma_1\neq \sigma_2,
	\end{aligned}
		\end{equation}
		and from the conservation relation (Theorem \ref{con2}) that 
			\begin{equation}\label{161}
			\begin{aligned}
			\left(\theta_{\psi,2n-2}^{+}\sqcup\theta_{\psi,2n-2}^{-}\right)(\sigma_1)\neq \left(\theta_{\psi,2n-2}^{+}\sqcup\theta_{\psi,2n-2}^{-}\right)(\sigma_2)\otimes \det \quad
			\mbox{for all}\,\, \sigma_1, \sigma_2. 
			\end{aligned}
			\end{equation}
	Hence 
	\begin{align}\label{1113}
	|\Pi_{\phi,\psi}|\geq 4|\Pi_{\phi^{-}}|. 
	\end{align}	
	On the other hand, we have $|\mathcal S_{\phi^{+}}|=2|\mathcal S_{\phi^{-}}|$ in this case. Then it follows from Theorem \ref{llcsympletic} (2) that 
	\begin{align}\label{1114}
	2|\Pi_{\phi^{-}}|=2|\widehat{\bar {\mathcal S}_{\phi^-}}|=|\widehat{\bar {\mathcal S}_{\phi^+}}|=|\Pi_{\phi^{+}}|. 
	\end{align}
	Combining (\ref{1111}), (\ref{1113}) and (\ref{1114}), we deduce  
	\begin{align*}
	|\Pi_{\phi,\psi}|=2|\Pi_{\phi^{+}}|=  4|\Pi_{\phi^{-}}|=4|\widehat{\bar {\mathcal S}_{\phi^{-}}}|=|\widehat{\mathcal {S}_{\phi}}|. 
	\end{align*}
	Hence for $\pi\in \Pi_{\phi,\psi}$, exactly one of $\pi$ and $\pi\otimes\det$ is in the image of $(\theta_{\psi,2n-2}^{+}\sqcup\theta_{\psi,2n-2}^{-})$. This proves Case II. 
\begin{comment}
		Composing (\ref{theta2n}) and (\ref{theta2n-22}), we get 
		\begin{align*}
		\theta_{\psi,2n}\circ(\theta_{\psi,2n-2}^{+}\sqcup\theta_{\psi,2n-2}^{-}): \Pi_{\phi^{-}}\sqcup \Pi_{\phi^{-}} \longrightarrow \Pi_{\phi^{+}},
		\end{align*}
		and it follows from (\ref{15}) and (\ref{16}) that $\theta_{\psi,2n}\circ (\theta_{\psi,2n-2}^{+}\sqcup\theta_{\psi,2n-2}^{-})$ is injective. On the other hand, we have $|\mathcal S_{\phi^{+}}|=2|\mathcal S_{\phi^{-}}|$ in this case. Then it follows from Theorem \ref{llcsympletic} (2) that 
		\begin{align}
		2|\Pi_{\phi^{-}}|=2|\widehat{\bar {\mathcal S}_{\phi^-}}|=|\widehat{\bar {\mathcal S}_{\phi^+}}|=|\Pi_{\phi^{+}}|. 
		\end{align}
		So $\theta_{\psi,2n}\circ (\theta_{\psi,2n-2}^{+}\sqcup\theta_{\psi,2n-2}^{-})$ is a bijection. Similarly to Case I, together with (\ref{15}) and (\ref{16}), we deduce  
		\begin{align*}
		|\Pi_{\phi,\psi}|=|\Pi_{\phi^{+}}|=  2|\Pi_{\phi^{-}}|=2|\widehat{\bar {\mathcal S}_{\phi^{-}}}|=|\widehat{\mathcal {S}_{\phi}}|. 
		\end{align*}
	Hence for $\pi\in \Pi_{\phi,\psi}$, exactly one of $\pi$ and $\pi\otimes\det$ is in the image of $(\theta_{\psi,2n-2}^{+}\sqcup\theta_{\psi,2n-2}^{-})$. This proves Case II.
\end{comment}
	The proof for (ii) is similar to (i) and we omit it here.
\end{proof}

As a consequence of Corollary \ref{1inphi} and Proposition \ref{1notinphi}, we can count the sizes of the fibers of $\mathcal L_{\psi}$. 
\begin{corollary}\label{size}
	Let $\phi\in \Phi_{\mathrm {temp}}(\mathrm O(V_{2n}))$. Then the size of the packet $\Pi_{\phi,\psi}$ is exactly the same as the size of $\widehat{\mathcal {S}_{\phi}}$. 
\end{corollary}
\begin{proof}
The case when $\mathbbm 1\subseteq \phi$ follows directly from the proof of Proposition \ref{1notinphi}. So it suffices to prove the case when 
$\mathbbm 1\nsubseteq \phi$. We define 
	\begin{align*}
	\theta_{\psi,2n}: \bigsqcup_{\delta\in \{\pm 1\}}\Irrt \left(\mathrm O(V_{2n}^{\delta})\right)&\longrightarrow \Irrt \left(\SP(W_{2n})\right)\\
	\pi &\mapsto \theta_{W_{2n}, V_{2n}^{\delta} ,\psi}(\pi). 
	\end{align*}
	By Propsition \ref{prasad}, Proposition \ref{1notinphi} and the Howe duality, we deduce that the restriction of $\theta_{\psi,2n}$ to $\Pi_{\phi,\psi}$ gives a bijection 
	\begin{equation}\label{theta2n1}
	\theta_{\psi,2n}: \Pi_{\phi,\psi} \rightarrow \Pi_{\phi^+},
	\end{equation}
	where $\phi^+=(\phi\otimes\chi_{V})\oplus \chi_{V}$.
	It follows from Theorem \ref{llcsympletic} (2) that 
	\begin{align}\label{11}
	|\Pi_{\phi^+}|=|\widehat{\bar {\mathcal S}_{\phi^+}}|=\frac{1}{2}|\widehat{ \mathcal S_{\phi^+}}|.
	\end{align}
	On the other hand, since $\mathbbm 1\nsubseteq \phi$, we deduce  
	\begin{align}\label{13}
	\mathcal S_{\phi^+}\cong \mathcal {S}_{\phi}\oplus (\mathbb{Z}/2\mathbb{Z}) e \quad \mbox{and}\quad |\widehat{\mathcal {S}_{\phi}}|=\frac{1}{2}|\widehat{\mathcal S_{\phi^+}}|,
	\end{align}
	where $e\in \mathcal S_{\phi^+}$ corresponds to $\chi_{V} \subseteq \phi^+$. Combining (\ref{theta2n1}), (\ref{11}) and (\ref{13}), we deduce $$|\Pi_{\phi,\psi}|=|\Pi_{\phi^+}|=\frac{1}{2}|\widehat{\mathcal S_{\phi^+}}|=|\widehat{\mathcal {S}_{\phi}}|.$$ This finishes the proof. 
\end{proof}

%The following corollary follows from the lemma 
%\begin{corollary}
%When $\mathbbm 1\nsubseteq \phi$, the size of the packet $\Phi_{\phi}$ is exactly the same with the size of $\widehat{A}_{\phi}$. 
%\end{corollary}

%Combine corollary and ?? , we have

\subsection{Construction of $\mathcal J^\psi_{\mathfrak W_{c^\prime}}$}\label{constructJc}
Given a tempered parameter $\phi \in \Para \mathrm O(V_{2n})$, we have shown that the size of the packet $\Pi_{\phi,\psi}$ equals to the size of $\widehat{\mathcal {S}_{\phi}}$. In this subsection, we are going to define a bijection
\begin{align*}
\mathcal J^{\psi}_{\mathfrak W_{c^\prime}}: \Pi_{\phi,\psi}\longrightarrow \widehat{\mathcal {S}_{\phi}} 
\end{align*} 
for each choice of the Whittaker data $\mathfrak W_{c^\prime}$. We will prove $\mathcal J^{\psi}_{\mathfrak W_{c^\prime}}$ is independent of the choice of $\psi$ in Section \ref{changeofpsi}, and hence get our desired $\mathcal J_{\mathfrak W_{c^\prime}}$.
 
Fix $c^\prime \in F^\times$. We shall separate the construction of $\mathcal J^{\psi}_{\mathfrak W_{c^\prime}}$ into two cases, depending on $\mathbbm 1 \subseteq \phi$ or not. 

$\bullet$ \underline{Case I}: If $\mathbbm 1 \nsubseteq \phi$, then by (\ref{theta2n1}), we know that   
    \begin{equation}
    \begin{aligned}\label{52}
    \theta_{\psi,2n}: \Pi_{\phi,\psi} &\rightarrow \Pi_{\phi^+}\\
    \pi &\mapsto \sigma= \theta_{W_{2n}, V_{2n}^{\delta} ,\psi}(\pi)
    \end{aligned}
    \end{equation}
  is a bijection, where $\phi^+=\left(\phi\otimes \chi_{V}\right)\oplus \chi_{V}$. On the other hand, we have  
	$$
	\mathcal S_{\phi^+}\cong \mathcal {S}_{\phi}\oplus (\mathbb Z/2\mathbb Z)e,
	$$
	where $e$ is the element in $\mathcal S_{\phi^+}$ corresponding to $\chi_{V}\subseteq \phi^+$. This induces an isomorphism 
	\begin{equation}\label{55}
	\ell:\mathcal S_{\phi}\hookrightarrow \mathcal S_{\phi^+}\twoheadrightarrow\bar{ \mathcal S}_{\phi^+}. 
	\end{equation}
	For $\pi\in \Pi_{\phi,\psi}$, we define 
	$$
	\mathcal J^{\psi}_{\mathfrak W_{c^\prime}}(\pi)\coloneqq  \ell^*(\mathcal J_{\mathfrak W^\prime_{\psi,c^\prime}}(\sigma)), 
	$$
	where $\sigma= \theta_{\psi,2n}(\pi)$ in (\ref{52}). 
	
	Note that in this case, the following diagram  
	\[
	\begindc{\commdiag}[500]
	\obj(-1,1)[aa]{$\Pi_{\phi^+}$}
	\obj(1,1)[bb]{$\widehat{\bar{\mathcal{S}}_{\phi^+}}$}
	\obj(-2,0)[cc]{$\Pi_{\phi,\psi}$}
	\obj(2,0)[dd]{$\widehat{\mathcal{S}_\phi}$}
	\mor{aa}{bb}{$\mathcal{J}_{\mathfrak{W}'_{\psi,c^\prime}}$}
	\mor{cc}{aa}{$\theta_{\psi,2n}$}
	\mor{bb}{dd}{$\ell^*$}
	\mor{cc}{dd}{$\mathcal{J}_{\mathfrak{W}_{c^\prime}}^\psi$}
	\enddc
	\]
	is commutative and every arrow in this diagram is a bijection. 
	
$\bullet$ \underline{Case II}: If $\mathbbm 1\subseteq \phi$, let $\phi^+=(\phi\otimes\chi_V)\oplus \chi_V$. Then the map 
\begin{equation}\label{501}
\ell:\mathcal S_{\phi}\hookrightarrow \mathcal S_{\phi^+}\twoheadrightarrow\bar{ \mathcal S}_{\phi^+}
\end{equation}
is surjective with the kernel isomorphic to $\mathbb Z/2\mathbb Z$. Let $\Pi_{\phi,\psi}^{+}$ be the subset of all representations $\pi\in \Pi_{\phi,\psi}$ such that $\theta_{W_{2n}, V_{2n}^{\delta} ,\psi}(\pi)\neq 0$. By Proposition \ref{1notinphi}, for $\pi\in \Pi_{\phi,\psi}$, 
$$\pi\ncong \pi\otimes\det
$$
and exactly one of $\pi$ and $\pi\otimes \det$ lies in $\Pi_{\phi,\psi}^{+}$, so $\Pi_{\phi,\psi}^{+}$ is half the size of $ \Pi_{\phi,\psi}$. It follows from Proposition \ref{prasad} and Proposition \ref{1notinphi} that the map 
	\begin{equation}\label{56}
	\begin{aligned}
	\theta_{\psi,2n}: \Pi_{\phi,\psi}^{+} &\longrightarrow \Pi_{\phi^+}\\
	\pi &\mapsto \sigma=\theta_{W_{2n}, V_{2n}^{\delta} ,\psi}(\pi)
	\end{aligned}
	\end{equation}
	is a bijection. For each $\pi\in \Pi_{\phi,\psi}^{+}$, we define 
	\begin{align*}
 \mathcal J^{\psi}_{\mathfrak W_{c^\prime}}(\pi)\coloneqq \ell^*(\mathcal J_{\mathfrak W^\prime_{\psi,c^\prime}}(\sigma)),
	\end{align*}
	where $\sigma=\theta_{\psi,2n}(\pi)$ in (\ref{56}). Next we define $\mathcal J^{\psi}_{\mathfrak W_{c^\prime}}$ on the other half of $\Pi_{\phi,\psi}$. For $\pi \in \Pi_{\phi,\psi}\backslash \Pi_{\phi,\psi}^{+}$, by the conservation relation (Theorem \ref{con2}), we have $\pi\otimes \det\in \Pi^+_{\phi,\psi}$. We define 
	\begin{equation*}
	\mathcal J^{\psi}_{\mathfrak W_{c^\prime}}(\pi)\coloneqq \mathcal J^{\psi}_{\mathfrak W_{c^\prime}}(\pi\otimes\det)\otimes\kappa_{\phi},
	\end{equation*}
	where $\kappa_{\phi}$ is defined by (\ref{detmap}). It is easy to check that the map 
	\begin{align*}
	\mathcal J^{\psi}_{\mathfrak W_{c^\prime}}: \Pi_{\phi,\psi} \longrightarrow \widehat {\mathcal {S}_{\phi}}
	\end{align*} 
	is also a bijection in this case. 

Combining these two cases, we deduce that our construction satisfies Theorem \ref{desideratumall} (2) for tempered representations:  
\begin{proposition}\label{bijection}
Let $\phi\in \Para(\mathrm O(V_{2n}))$. For each Whittaker datum $\mathfrak W_{c^\prime}$, the map 
	\begin{align*}
	\mathcal J^{\psi}_{\mathfrak W_{c^\prime}}: \Pi_{\phi,\psi}\longrightarrow \widehat{\mathcal {S}_{\phi}}
	\end{align*} 
is a bijection. 
\end{proposition}
\subsection{From tempered to non-tempered}\label{nontempered}
So far, we have attached an $L$-parameter and a character of the component group to $\pi\in \Irrt \mathrm O(V_{2n}^{\delta})$. In this subsection, we will extend this construction to non-tempered representations by using the Langlands classification.  

Let $\pi\in \Irr(\mathrm O(V_{2n}^{\delta}))$. By the Langlands classification for $p$-adic groups \cite{MR507262}, \cite{MR2050093}, we know that $\pi$ is the unique irreducible quotient of the standard module 
\begin{equation*}
\Ind_{P}^{\mathrm O(V_{2n}^{\delta})}\left(\tau_{1}|\cdot|_F^{s_1}\otimes \cdots\otimes \tau_{r}|\cdot|_F^{s_r}\otimes\pi_{0} \right),
\end{equation*}
where 
\begin{itemize}
	\item $P$ is a parabolic subgroup of $\mathrm O(V_{2n}^{\delta})$ with Levi component $\GL_{k_{1}}(F)\times \cdots \times \GL_{k_{r}}(F)\times \mathrm O(V_{2n_0}^{\delta})$;
	\item $\tau_{i}$ is an irreducible tempered representation of $\GL_{k_{i}}(F)$ for $i=1,\cdots,r$;
	\item $\pi_0$ is an irreducible tempered representation of $\mathrm O(V_{2n_0}^\delta)$;
	\item $n=k_1+\cdots k_r +n_0 $ and $s_1>\cdots >s_r>0$. 
\end{itemize}
Then we define
$$
\mathcal L_{\psi}(\pi)\coloneqq \phi_{1}|\cdot|^{s_1}+\cdots + \phi_{r}|\cdot|^{s_r}+ \phi_0+ \phi_{r}^{\vee}|\cdot|^{-s_r}+ \cdots +  \phi_{1}^{\vee}|\cdot|^{-s_1},
$$
where $\phi_{i}$ is the $L$-parameter of $\tau_{i}$ and $\phi_0=\mathcal L_{\psi}(\pi_0)$. The natural embedding $\mathcal S_{\phi_0}\hookrightarrow \mathcal {S}_{\phi}$ is an isomorphism in this case and we identify $\mathcal {S}_{\phi}$ with $\mathcal S_{\phi_{0}}$ via this isomorphism. For a fixed Whittaker datum $\mathfrak W_{c^\prime}$, we define  
$$
\mathcal J^\psi_{\mathfrak W_{c^\prime}}(\pi)\coloneqq \mathcal J^\psi_{\mathfrak W_{c^\prime}}(\pi_0).
$$
Since the standard module is unique up to Weyl group conjugation, the maps $\mathcal L_{\psi}$ and $\mathcal J_{\mathfrak W_{c^\prime}}^\psi$ are well defined. It then follows from our construction that $\mathcal L_{\psi}$ and $\mathcal J_{\mathfrak W_{c^\prime}}^\psi$ are compatible with Langlands quotients. 
\begin{proposition}\label{Langlandsclasification}
For each Whittaker datum $\mathfrak W_{c^\prime}$, the maps $\mathcal L_{\psi}$ and $\mathcal J^\psi_{\mathfrak W_{c^\prime}}$ we constructed are compatible with Langlands quotients, i.e., satisfy Theorem \ref{desideratumall} (10). 
\end{proposition}

The following corollary follows from Proposition \ref{bijection} and our construction in this subsection.   
\begin{corollary}\label{bijection2}
	For each Whittaker datum $\mathfrak W_{c^\prime}$, the map 
	\begin{align*}
	\mathcal J^{\psi}_{\mathfrak W_{c^\prime}}: \Pi_{\phi,\psi}\longrightarrow \widehat{\mathcal {S}_{\phi}}
	\end{align*} 
	is a bijection. Hence Theorem \ref{desideratumall} (2) holds for our construction.  
\end{corollary}

Next we extend Lemma \ref{respectgamma} and Lemma \ref{respectplancherel} from tempered representations to general cases.  
\begin{proposition}\label{localfactor}
	The map $\mathcal L_{\psi}$ we constructed respect the standard $\gamma$-factors
	 and the Plancherel measures, i.e., it satisfies Theorem \ref{desideratumall} (11) and (12).   
\end{proposition}
\begin{proof}
	This follows from the tempered case and the multiplicativity of the standard $\gamma$-factors \cite{MR2192828} \& Plancherel measures \cite[Section 10.2, Appendix B.5]{MR3166215}. 
\end{proof}

\subsection{Some properties}
We shall prove the LLC we constructed satisfies Theorem \ref{desideratumall} (4) and (5) in this subsection.

\begin{proposition}\label{changeofwhittakerdatum}
	Let $\mathfrak W_{c_1}$ and $\mathfrak W_{c_2}$ be two Whittaker data. Then for any
	$\pi\in \Pi_{\phi,\psi}$, we have 
	$$
	\mathcal J^\psi_{\mathfrak W_{c_2}}(\pi)=\mathcal J^\psi_{\mathfrak W_{c_1}}(\pi)\otimes \eta_{\phi\chi_{V},c_2/c_1},
	$$
	where $\phi\chi_{V}=\phi\otimes\chi_{V}$ and $\eta_{\phi\chi_{V},c_2/c_1}$ is defined in (\ref{eta}). Hence Theorem \ref{desideratumall} (4) holds for our construction.
\end{proposition}
\begin{proof}
By Proposition \ref{Langlandsclasification}, it is enough to prove the case when  $\phi\in \Para(\mathrm O(V_{2n}))$. Let $\pi\in \Pi_{\phi,\psi}(\mathrm O(V_{2n}^\delta))$. We divide it into two cases: 
	
	$\bullet$ \underline{Case I}: Assume $\theta_{W_{2n},V_{2n}^\delta, \psi}(\pi)\neq 0$. Let $\sigma=\theta_{W_{2n},V_{2n}^\delta, \psi}(\pi)$ and $\phi^+$ be the $L$-parameter of $\sigma$. Then by Proposition \ref{prasad}, 
	\begin{align*}
	\phi^+=(\phi\otimes\chi_{V})+ \chi_{V}.
	\end{align*}
	Moreover, by our construction of $\mathcal J^\psi_{\mathfrak W_{c^\prime}}$, we have 
	\begin{align*}
	\mathcal J^\psi_{\mathfrak W_{c_1}}(\pi)=\ell^*(\mathcal J_{\mathfrak W^\prime_{\psi,c_1}}(\sigma))\quad \mbox{and}\quad 
	\mathcal J^\psi_{\mathfrak W_{c_2}}(\pi)=\ell^*(\mathcal J_{\mathfrak W^\prime_{\psi,c_2}}(\sigma)), 
	\end{align*}
	where $\ell: \mathcal {S}_{\phi}\rightarrow \bar S_{\phi^+}$ is the isomorphism  in (\ref{55}). It follows from Theorem \ref{llcsympletic} (4) that 
	\begin{align*}
	\mathcal J_{\mathfrak W^\prime_{\psi,c_2}}(\sigma)= \mathcal J_{\mathfrak W^\prime_{\psi,c_1}}(\sigma)\otimes \eta_{\phi^+,c_2/c_1}. 
	\end{align*}
	Hence 
	\begin{align*}
	\mathcal J^\psi_{\mathfrak W_{c_2}}(\pi)=&\ell^*(\mathcal J_{\mathfrak W^\prime_{\psi,c_2}}(\sigma))=\ell^*(\mathcal J_{\mathfrak W^\prime_{\psi,c_1}}(\sigma)\otimes \eta_{\phi^+,c_2/c_1})
	=\mathcal J^\psi_{\mathfrak W_{c_1}}(\pi)\otimes\eta_{\phi\chi_{V},c_2/c_1}. 
	\end{align*}
	
	$\bullet$ \underline{Case II}: Assume $\theta_{W_{2n},V_{2n}^\delta, \psi}(\pi)=0$. Then $\theta_{W_{2n},V_{2n}^\delta, \psi}(\pi\otimes\det)\neq 0$ by the conservation relation (Theorem \ref{con2}). Hence by Case I, we have 
	\begin{align}\label{137}
	\mathcal J^\psi_{\mathfrak W_{c_2}}(\pi\otimes\det)=\mathcal J^\psi_{\mathfrak W_{c_1}}(\pi\otimes
	\det)\otimes\eta_{\phi\chi_{V},c_2/c_1}. 
	\end{align}
	On the other hand, it follows from our construction of $\mathcal J^\psi_{\mathfrak W_{c^\prime}}$ that 
	\begin{equation}\label{136}
	\begin{aligned}
	\mathcal J^\psi_{\mathfrak W_{c_1}}(\pi)=\mathcal J^\psi_{\mathfrak W_{c_1}}(\pi\otimes\det)\otimes\kappa_{\phi}\quad \mbox{and}\quad 
	\mathcal J^\psi_{\mathfrak W_{c_2}}(\pi)=\mathcal J^\psi_{\mathfrak W_{c_2}}(\pi\otimes\det)\otimes\kappa_{\phi}. 
	\end{aligned}
	\end{equation} 
	 So by (\ref{137}) and (\ref{136}), we deduce  
	\begin{align*}
	\mathcal J^\psi_{\mathfrak W_{c_2}}(\pi)=\mathcal J^\psi_{\mathfrak W_{c_1}}(\pi)\otimes\eta_{\phi\chi_{V},c_2/c_1}
	\end{align*}
	as desired. 
\end{proof}
\begin{proposition}\label{tempered}
	Let $\pi\in \Irr(\mathrm O(V_{2n}^\delta))$. Then we have  
	\begin{enumerate}[(i)]
		\item $\pi\in \Irrt(\mathrm O(V_{2n}^\delta))$ if and only if $\phi=\mathcal L_\psi(\pi)\in \Para(\mathrm O(V_{2n}))$; 
		\item $\pi$ is a discrete series representation if and only if $\phi=\mathcal L_\psi(\pi)\in \Phi_{\disc}(\mathrm O(V_{2n}))$. 
	\end{enumerate} 
Hence Theorem \ref{desideratumall} (5) holds for our construction.
\end{proposition}
\begin{proof}
	(i) automatically follows from our construction. We then prove (ii). First, we prove $\pi\in \Pi_{\phi,\psi}$ is a discrete series representation if $\phi$ is a discrete parameter. We shall separate the proof into two cases, depending on $\mathbbm 1\subseteq \phi$ or not.  

$\bullet$ \underline{Case I}: If $\mathbbm 1\nsubseteq \phi$, then by Proposition \ref{prasad} and Corollary \ref{1inphi}, we have  $$\sigma=\theta_{W_{2n},V_{2n}^\delta,\psi}(\pi)\neq 0$$ 
and the $L$-parameter of $\sigma$ is $$\phi^+=(\phi\otimes\chi_{V})+\chi_{V}.$$
Since $\phi$ is a discrete parameter and $\mathbbm 1\nsubseteq \phi$, we deduce that $\phi^+$ is also a discrete parameter. Then $\sigma$ is a discrete series representation of $\SP(W_{2n})$ by Theorem \ref{llcsympletic} (5). Hence it follows from Lemma \ref{thetadiscrete} that  $\pi=\theta_{V_{2n}^\delta,W_{2n},\psi}(\sigma)$ is a discrete series representation of $\mathrm O(V_{2n})$. 
	
$\bullet$ \underline{Case II}: If $\mathbbm 1\subseteq \phi$, then by Proposition \ref{1notinphi}, exactly one of 
	$$\theta_{W_{2n-2},V_{2n}^\delta,\psi}(\pi) \quad \mbox{and}\quad \theta_{W_{2n-2},V_{2n}^\delta,\psi}(\pi\otimes\det)$$
is nonzero. Since $\pi$ is a discrete series representation if and only if $\pi\otimes\det$ is a discrete series representation, without loss of generality, we may assume that  $$\theta_{W_{2n-2},V_{2n}^\delta,\psi}(\pi)\neq 0 .$$ 
Let $\sigma=\theta_{W_{2n-2},V_{2n}^\delta,\psi}(\pi)$. Then by Proposition \ref{prasad}, the $L$-parameter of $\sigma$ is $\phi^-$ such that 
$$\phi=(\phi^-\otimes\chi_{V})+\mathbbm 1.$$
Since $\phi$ is a discrete parameter, so is $\phi^-$. So it follows from Theorem \ref{llcsympletic} (5) that $\sigma$ is a discrete series representation of $\SP(W_{2n-2})$. Note that $\mathbbm 1\oplus \mathbbm 1 \nsubseteq \phi$ since $\phi$ is a discrete parameter, and hence $\chi_{V}\nsubseteq \phi^+$. It follows from Corollary \ref{1inphi} that $$\theta_{V_{2n-2}^\delta,W_{2n-2},\psi} (\sigma)=0.$$
Then by Lemma \ref{thetadiscrete}, $\pi=\theta_{V_{2n}^\delta,W_{2n-2},\psi}(\sigma)\neq 0$ is a discrete series representation of $\mathrm O(V_{2n}^\delta)$. 

Next, for $\pi\in \Irrt(\mathrm O(V_{2n}^\delta))$, we prove $\phi=\mathcal L_\psi(\pi)$ is a discrete parameter if $\pi$ is a discrete series representation. We also separate the proof into two cases, depending on $\mathbbm 1\subseteq \phi$ or not: 
	
$\bullet$ \underline{Case I}: If $\mathbbm 1\nsubseteq \phi$, then by Corollary \ref{1inphi}, we have 
$$\theta_{W_{2n-2},V_{2n}^\delta,\psi} (\pi)=0\quad \mbox{and}\quad \theta_{W_{2n},V_{2n}^\delta,\psi}(\pi)\neq 0.$$
Let $\sigma=\theta_{W_{2n},V_{2n}^\delta,\psi}(\pi)$ and $\phi^+$ be the $L$-parameter of $\sigma$. Then $\sigma$ is a discrete series representation of $\SP(W_{2n})$ by Lemma \ref{thetadiscrete}. Hence $\phi^+$ is a discrete parameter by Theorem \ref{llcsympletic} (5). On the other hand, it follows from Proposition \ref{prasad} that 
$$\phi^+=(\phi\otimes\chi_{V})+\mathbbm 1. 
$$
So $\phi$ is also a discrete parameter.

$\bullet$ \underline{Case II}: If $\mathbbm 1\subseteq \phi$, then by Proposition \ref{1notinphi}, exactly one of  $$\theta_{W_{2n-2},V_{2n}^\delta,\psi}(\pi)\quad\mbox{and}\quad \theta_{W_{2n-2},V_{2n}^\delta,\psi}(\pi\otimes\det)$$
is non-zero. Since 
$$
\mathcal L_\psi(\pi)=\mathcal L_\psi(\pi\otimes\det),
$$
without loss of generality, we may assume
\begin{align}\label{1231}
\theta_{W_{2n-2},V_{2n}^\delta,\psi}(\pi)\neq 0 \quad \mbox{and}\quad  \theta_{W_{2n-2},V_{2n}^\delta,\psi}(\pi\otimes\det)= 0.
\end{align} 
Let $\sigma=\theta_{W_{2n-2},V_{2n}^\delta,\psi}(\pi)$ and $\phi^-$ be the $L$-parameter of $\sigma$. It follows from Lemma \ref{thetadiscrete} that $\sigma$ is a discrete series representation of $\SP(W_{2n-2})$. Then by Theorem \ref{llcsympletic} (5), $\phi^-$ is a discrete  parameter. Note that by Proposition \ref{prasad}, we have
\begin{align}\label{1234}
\phi=(\phi^-\otimes\chi_{V})\oplus \mathbbm 1.
\end{align}
We shall prove that $\chi_{V}\nsubseteq \phi^-$, which will imply $\phi$ is a discrete parameter by (\ref{1234}). We prove it by contradiction. Suppose $\chi_{V}\subseteq \phi^-$. Then by the conservation relation (Theorem \ref{con2}), Proposition \ref{1notinphi} and (\ref{1231}), we have 
$$\theta_{V_{2n-2}^\delta,W_{2n-2},\psi} (\sigma)\neq 0.$$
But by Lemma \ref{thetadiscrete}, this will imply $\pi=\theta_{V_{2n}^\delta,W_{2n-2},\psi} (\sigma)$ is a tempered but not discrete series representation, which contradicts with our assumption. This completes the proof. 
\end{proof}

\section{Variation of $\psi$}\label{changeofpsi}
 In this section, we shall prove that $\mathcal L_\psi$ and $\mathcal J^\psi_{\mathfrak W_{c^\prime}}$ we constructed are independent of the choice of $\psi$. Note that every non-trivial additive character of $F$ is of the form $\psi_a$ for some $a\in F^\times$. The method is to study the behaviour of theta correspondence under the change of $\psi$. More precisely, we would like to compare $\theta_{W_{2n}, V_{2n}^\delta, \psi_a}(\pi)$ and $\theta_{W_{2n},V_{2n}^\delta,\psi}(\pi)$ for $\pi\in \Irr(\mathrm O(V_{2n}^\delta))$ and $a\in F^\times$. For any $\sigma\in \Irr(\SP(W_{2n}))$, let $\sigma^{\delta_a}\in \Irr(\SP(W_{2n}))$ be defined in (\ref{113}).   
\begin{lemma}\label{scaling}
	Let $\pi\in \Irr \left(\mathrm O(V_{2n}^\delta)\right)$ and $a\in F^{\times}$. We have  
	$$
	\theta_{W_{2n}, V_{2n}^\delta, \psi_a}(\pi)=
	\left( \theta_{W_{2n},V_{2n}^\delta,\psi}(\pi)\right)^{\delta_a}. 
	$$
\end{lemma}
\begin{proof}
	See \cite[II Corollary 6.2]{kudla1996notes} and \cite[IV Proposition 1.9]{kudla1996notes}
\end{proof}

\begin{proposition}\label{notdependonpsi}
	For any $\pi\in \Irr\left( \mathrm O(V_{2n}^\delta)\right)$ and $a\in F^{\times}$, we have 
	\begin{enumerate}[(i)]
		\item $\mathcal L_{\psi}(\pi)=\mathcal L_{\psi_a}(\pi);$
		\item $\mathcal J^\psi_{\mathfrak W_{c^\prime}}(\pi)= \mathcal J^{\psi_a}_{\mathfrak W_{c^\prime}}(\pi)$. 	
	\end{enumerate}	
\end{proposition}
\begin{proof}
Since the LLC we construct is compatible with the Langlands quotients, it is enough to prove these for tempered representations. Hence we may assume that $\pi\in \Irrt \left(\mathrm O(V_{2n}^\delta)\right)$.  
We divide it into two cases: 

$\bullet$ \underline{Case I}: Assume $\theta_{W_{2n},V_{2n}^\delta,\psi}(\pi)\neq 0$. Let $\sigma=\theta_{W_{2n},V_{2n}^\delta,\psi}(\pi)$. Then by Lemma \ref{scaling}, we have 
\begin{equation*}
\begin{aligned}
\theta_{W_{2n}, V_{2n}^\delta,  \psi_a}(\pi)=\left(\theta_{W_{2n}, V_{2n}^\delta, \psi}(\pi)\right)^{\delta_a}=\sigma^{\delta_a}. 
\end{aligned}
\end{equation*}
It follows from (\ref{103}) that the $L$-parameters for $\sigma$ and $\sigma^{\delta_a}$ are the same. Then by our constructions of $\mathcal L_\psi$ and $\mathcal L_{\psi_a}$, we have 
$$\mathcal L_{\psi}(\pi)=\mathcal L_{\psi_a}(\pi).$$ Next we consider the map $\mathcal J^{\psi}_{\mathfrak W_{c^\prime}}$. By our constructions of $\mathcal J^{\psi}_{\mathfrak W_{c^\prime}}$ and $\mathcal J^{\psi_a}_{\mathfrak W_{c^\prime}}$, we have  
		\begin{equation}\label{126}
		\begin{aligned}
		\mathcal J^\psi_{\mathfrak W_{c^\prime}}(\pi)=\ell^*\left(\mathcal J_{\mathfrak W^\prime_{\psi,c^\prime}}(\sigma)\right)\quad \mbox{and}\quad 
		\mathcal J^{\psi_a}_{\mathfrak W_{c^\prime}}(\pi)=\ell^*\left(\mathcal J_{\mathfrak W^\prime_{\psi_a,c^\prime}}(\sigma^{\delta_a})\right),
		\end{aligned}
		\end{equation}
		where $\ell: \mathcal {S}_{\phi}\rightarrow \bar {\mathcal S}_{\phi^+}$ is defined in (\ref{55}). By (\ref{103}), we have 
		\begin{equation}\label{124}
		\begin{aligned}
		\mathcal J_{\mathfrak W^\prime_{\psi_a,c^\prime}}(\sigma^{\delta_a})=\mathcal J_{\mathfrak W^\prime_{\psi,c^\prime}}(\sigma).
		\end{aligned}
		\end{equation}
		So $\mathcal J^\psi_{\mathfrak W_{c^\prime}}(\pi)= \mathcal J^{\psi_a}_{\mathfrak W_{c^\prime}}(\pi)$ by (\ref{126}) and (\ref{124}).

$\bullet$ \underline{Case II}:  If $\theta_{W_{2n}, V_{2n}^\delta, \psi}(\pi)=0$, then $\theta_{W_{2n}, V_{2n}^\delta, \psi}(\pi\otimes\det)\neq 0$ by the conservation relation (Theorem \ref{con2}). Hence, by Case I, we have 
	\begin{align*}
	\mathcal L_{\psi}(\pi\otimes\det)=\mathcal L_{\psi_a}(\pi\otimes\det) \quad 
\mbox{and}\quad 
	\mathcal J^{\psi}_{\mathfrak W_{c^\prime}}(\pi\otimes\det)=\mathcal J^{\psi_a}_{\mathfrak W_{c^\prime}}(\pi\otimes\det) .
	\end{align*}
On the other hand, it follows from Proposition \ref{det} and our construction of $\mathcal J^{\psi}_{\mathfrak W_{c^\prime}}$ that 
\begin{align*}
	\mathcal L_\psi(\pi\otimes\det)=\mathcal L_\psi(\pi)\quad \mbox{and}\quad 
	\mathcal J^\psi_{\mathfrak W_{c^\prime}}(\pi\otimes\det)= \mathcal J^\psi_{\mathfrak W_{c^\prime}}(\pi)\otimes\kappa_{\phi}. 	
\end{align*}	
Hence (i) and (ii) also hold for Case II. 	
\end{proof}

Since the maps $\mathcal L_\psi$ and $\mathcal J^\psi_{\mathfrak W_{c^\prime}}$ do not depend on the choice of $\psi$, we may drop the symbol $\psi$ and denote them by $\mathcal L$ and $\mathcal J_{\mathfrak W_{c^\prime}}$. These give the desired LLC for even orthogonal groups. So far we have constructed the maps $\mathcal L$ and $\mathcal J_{\mathfrak W_{c^\prime}}$ for each Whittaker datum $\mathfrak W_{c^\prime}$. We also proved this construction satisfies Theorem \ref{desideratumall} (1), (2), (4), (5), (10), (11), (12).

\section{Preparations for local intertwining relation }\label{LIR}
Our next goal is to prove our construction satisfies the local intertwining relation. In this section, we do some preparations. We first recall the definition of the normalized intertwining operators, and then construct an important equivariant map, which is key to the proof.

\subsection{Normalized intertwining operators}\label{normalizingintertwing}
In this subsection, we define the normalized intertwining operators. We mainly follow \cite[\S 2.3]{MR3135650}, \cite[\S 3.7]{MR3708200} and \cite[\S 6]{MR3788848}.

Fix a non-trivial character $\psi$ of $F$ and $(d,c)\in (F^\times) ^2$. Let $V_{2n}=V_{2n}^+$ be the $2n$-dimensional orthogonal space associated to $(d,c)$ and $V\in \{V_{2n}^+, V_{2n}^-\}$. 
%Fix $d\in F^{\times}/F^{\times 2}$. Let $V=V_{2n}$ be a $2n$-dimensional orthogonal space with discriminant $d$ and discriminant character $\chi_V=(\cdot,d)_F$. 
Let $W=W_{2n}$ be a $2n$-dimensional symplectic space. Let $r$ be the Witt index of $V$ and $k$ be a positive integer with $k\leq r$. As in Section \ref{sectionparabolic}, we put 
$$V=X\oplus V_0\oplus X^{\vee}, \quad W=Y\oplus W_0\oplus Y^{\vee}$$
with $X=X_{k}, X^{\vee}=X_{k}^{\vee}$ and $Y=Y_{k}, Y^{\vee}=Y_{k}^{\vee}$. Hence $\dim(V_0)=\dim(W_0)=2n-2k$. Let $P=P_{k}=M_{P} U_{P}$ and $Q=Q_{k}=M_{Q} U_{Q}$ be the parabolic subgroups defined in Subsection \ref{sectionparabolic} such that 
$$M_{P} \cong \GL(X) \times \mathrm O(V_0) \quad\mbox{and}\quad  M_{Q} \cong \GL(Y) \times \SP(W_0).$$
Using the basis $\left\{v_{1}, \ldots, v_{k}\right\}$ of $X$ (resp. $\left\{w_{1}, \ldots, w_{k}\right\}$ of $Y$), we identify $\GL(X)$ (resp. $\GL(Y)$) with $\GL_k(F)$. Hence we can define an isomorphism $i: \GL(X) \rightarrow \GL(Y) $ via these identifications.  

Let $\tau$ be an irreducible tempered representation of $\GL_{k}(F)$ on a space $\mathscr{V}_{\tau}$ with a central character $\omega_{\tau}$. We may regard $\tau$ as a representation of $\GL(X)$ or $\GL(Y)$ via the above identifications. For any $s \in \mathbb{C}$, we realize the representation $\tau_{s} \coloneqq \tau \otimes|\det|_{F}^{s}$ on $\mathscr{V}_{\tau}$ by setting $\tau_{s}(a) v \coloneqq |\det(a)|_{F}^{s} \tau(a) v$ for all $v \in \mathscr{V}_{\tau}$ and $a \in \GL_{k}(F)$. Let $\pi_0$ (resp. $\sigma_0$) be an irreducible tempered representation of $\mathrm{O}(V_0)$ (resp. $\SP(W_0)$) on a space $\mathscr{V}_{\pi_0}$ (resp. $\mathscr{V}_{\sigma_0}$). We consider the induced representations 
$$\Ind_{P}^{\mathrm O\left(V\right)}\left(\tau_{s} \otimes \pi_0 \right) \quad \mbox{and} \quad \Ind_{Q}^{\SP\left(W\right)}\left(\tau_{s} \otimes \sigma_0 \right)$$
of $\mathrm O(V)$ and $\SP(W)$. They are realized on the spaces of smooth functions 
$$\Psi_{s} : \mathrm O(V) \rightarrow \mathscr{V}_{\tau} \otimes \mathscr{V}_{\pi_0}\quad  \mbox{and} \quad \Phi_{s}: \operatorname{Sp}\left(W\right) \rightarrow \mathscr{V}_{\tau} \otimes \mathscr{V}_{\sigma_0}$$
such that
\begin{align*} \Psi_{s}\left(u_{P} m_{P}(a) h_0 h\right) =|\det(a)|_{F}^{s+\rho_{P}} \tau(a) \pi_0(h_0) \Psi_{s}\left(h\right)
\end{align*}
and 
\begin{align*}
\Phi_{s}\left(u_{Q} m_{Q}\left(a^{\prime}\right) g_0 g\right) =\left|\operatorname{det}\left(a^{\prime}\right)\right|_{F}^{s+\rho_{Q}} \tau\left(a^{\prime}\right) \sigma_0(g_0) \Phi_{s}\left(g\right)
\end{align*}
for any $u_{P} \in U_{P}, a\in \GL(X), h_0\in \mathrm{O}(V_0), h \in \mathrm{O}(V)$, and  $u_{Q} \in U_{Q}, a^{\prime} \in \mathrm{GL}(Y), g_0 \in \SP(W_0),g \in \SP(W)$, respectively. Let $A_{P}$ (resp. $A_{Q}$) be the split component of the center of $M_{P}$ (resp. $M_{Q}$) and $W\left(M_{P}\right)=N_{\mathrm O(V)}(A_P)/M_P$
(resp. $W(M_{Q})=N_{\SP(W)}(A_Q)/M_Q$) be the relative Weyl group for $M_{P}$ (resp. $M_{Q}$). Note that 
$$W(M_{P}) \cong W(M_{Q}) \cong \mathbb{Z} / 2 \mathbb{Z}.$$ 
We denote by $w$ (resp. $ w^{\prime}$) the non-trivial element in $W(M_{P})$ (resp. $W(M_{Q}) $). For any representative $\widetilde{w} \in \mathrm{O}(V)$ of $w$ and $\widetilde{w}^{\prime}\in \SP(W)$ of $w^{\prime}$, we define the unnormalized intertwining operators 
\begin{align*}
\mathcal{M}\left(\widetilde{w}, \tau_{s} \otimes \pi_0\right) : \Ind_{P}^{\mathrm{O}(V)}\left(\tau_{s} \otimes \pi_0 \right) \longrightarrow \operatorname{Ind}_{P}^{\mathrm{O}(V)}\left(w\left(\tau_{s} \otimes \pi_0 \right)\right)
\end{align*}
and 
\begin{align*}
\mathcal{M}\left(\widetilde{w}^{\prime}, \tau_{s} \otimes \sigma \right) : \operatorname{Ind}_{Q}^{\mathrm{Sp}(W)}\left(\tau_{s} \otimes \sigma_0 \right) \longrightarrow \operatorname{Ind}_{Q}^{\mathrm{Sp}(W)}\left(w^{\prime}\left(\tau_{s} \otimes \sigma_0 \right)\right)
\end{align*}
by (the meromorphic continuations of) the integrals
\begin{align*}
 \mathcal{M}\left(\widetilde{w}, \tau_{s} \otimes \pi_0 \right) \Psi_{s}\left(h\right) = \int_{U_{P}} \Psi_{s}\left(\widetilde{w}^{-1} u_{P} h\right) d u_{P}
 \end{align*}
 and 
 \begin{align*}
\mathcal{M}\left(\widetilde{w}^{\prime}, \tau_{s} \otimes \sigma_0 \right) \Phi_{s}\left(g\right) =\int_{U_{Q}} \Phi_{s}\left(\widetilde{w}^{\prime-1} u_{Q} g\right) d u_{Q}, \end{align*}
where 
\begin{itemize}
	\item $du_P$ and $du_Q$ are the Haar measures given in \cite[\S 6.3]{MR3788848});
	\item $w\left(\tau_{s} \otimes \pi_0 \right)$ and $w^{\prime}\left(\tau_{s} \otimes \sigma_0 \right) $ are the representations of $M_{P}$ on $\mathscr{V}_{\tau} \otimes \mathscr{V}_{\pi_0}$ and $M_{Q}$ on $\mathscr{V}_{\tau} \otimes \mathscr{V}_{\sigma_0} $ given by 
	\begin{align*}
	w\left(\tau_{s} \otimes \pi_0 \right)\left(m_{P}\right)=\left(\tau_{s} \otimes \pi_0 \right)\left(\widetilde{w}^{-1} m_{P} \widetilde{w}\right)
	\end{align*}
and 
\begin{align*}
	w^{\prime}\left(\tau_{s} \otimes \sigma_0 \right)\left(m_{Q}\right) =\left(\tau_{s} \otimes \sigma_0 \right)\left(\widetilde{w}^{\prime-1} m_{Q} \widetilde{w}^{\prime}\right)
	\end{align*}
	for $m_{P} \in M_{P}$ and $m_{Q}\in M_Q$. 
\end{itemize}

We shall normalize the intertwining operators $\mathcal{M}\left(\widetilde{w}, \tau_{s} \otimes \pi_0\right)$ and $\mathcal{M}\left(\widetilde{w}^{\prime}, \tau_{s} \otimes \sigma \right)$ depending on the choice of the Whittaker data. Having fixed the additive character $\psi$ of F, for any $c^\prime\in F^\times$, we use the Whittaker datum $\mathfrak W_{c^\prime}$ of even orthogonal group and $\mathfrak W^\prime_{\psi,1}$ of symplectic group as in Subsection \ref{whittaker}. The normalized intertwining operators 
\begin{align*}
\mathcal R_{\mathfrak W_{c^\prime}}(w, \tau_s \otimes \pi_0) &: \Ind_{P}^{\mathrm{O}(V)}\left(\tau_s\otimes \pi_0\right)\rightarrow  \operatorname{Ind}_{P}^{\mathrm{O}(V)}\left(w(\tau_s \otimes \pi_0)\right),\\
\mathcal R_{\mathfrak W^\prime_{\psi,1}}\left(w^{\prime}, \tau_s \otimes \sigma_0\right) &: \operatorname{Ind}_{Q}^{\SP(W)}\left(\tau_s \otimes \sigma_0\right) \rightarrow \operatorname{Ind}_{Q}^{\SP(W)}\left(w^\prime(\tau_s \otimes \sigma_0)\right)
\end{align*}
will be defined as follows: 
\begin{align*}
\mathcal R_{\mathfrak W_{c^\prime}}(w, \tau_s \otimes \pi_0) \Psi_s(h) &=\epsilon(V)^{k} \times \chi_{V}(c^\prime/c)^{k}\times |c^\prime|_F^{k\rho_P}\cdot r(w,\tau_s\otimes
\pi_0)^{-1}\times \mathcal{M}(\widetilde{w}_{c^\prime}, \tau_s \otimes \pi_0) \Psi_s(h), \\ 
\mathcal R_{\mathfrak W^\prime_{\psi,1}}(w^\prime, \tau_s \otimes \sigma_0) \Phi_s(g) &=r(w^\prime,\tau_s\otimes
\sigma_0)^{-1}\times \mathcal{M}(\widetilde{w}_1^\prime, \tau_s \otimes \sigma_0) \Phi_s(g),
\end{align*}
where 
\begin{itemize}
 \item $\epsilon(V)=\mathcal J_{\mathfrak W_{c}}(z_{\phi_{\pi_0}})$, where $\phi_{\pi_0}$ is the $L$-parameters for $\pi_0$. Hence 
 \[
 \epsilon(V)=\begin{cases}
 1 \quad  &\mbox{if $V=V_{2n}^+$},\\
 -1 \quad  & \mbox{if $V=V_{2n}^-$}.
 \end{cases}
 \]
\item $\widetilde{w}_{c^\prime}$ and $\widetilde{w}_1^{\prime}$ are defined by
\begin{equation}\label{131}
\begin{aligned}
\widetilde{w}_{c^\prime}&=w_{P} \cdot m_{P}\left(c^\prime \cdot a\right) \cdot\left((-1)^{k} \mathbf{1}_{V_0}\right),\\
\widetilde{w}_1^{\prime}&=w_{Q} \cdot m_{Q}\left( a^{\prime}\right) \cdot\left((-1)^k \mathbf{1}_{W_0}\right),
\end{aligned}
\end{equation}
where $a \in \mathrm{GL}_{k}(F) \cong \mathrm{GL}(X)$ and $a^{\prime} \in \mathrm{GL}_{k}(F) \cong \mathrm{GL}(Y)$ is given by
$$
a=a^\prime=\left(\begin{array}{ccc}
{} &{}& (-1)^{n-k+1} \\ 
{} &{\iddots} & {}\\
{(-1)^n}&{}& {}
\end{array}\right). 
$$
Note that  
\begin{equation}\label{determinantminus1}
\det(\widetilde{w}_{c^\prime})=\det(w_{P})\cdot \det\left(m_{P}(c^\prime \cdot a)\right)\cdot \det\left((-1)^{k} \mathbf{1}_{V_0}\right))=(-1)^{k}. 
\end{equation}
 \item  Following \cite{MR3135650}, \cite[\S 6.2]{MR3788848} and \cite[\S 3.7]{MR3708200}, $r\left(w, \tau_{s} \otimes \pi_0\right)$ and $r\left(w^{\prime}, \tau_{s} \otimes \sigma_0\right)$ are 
    defined as 
    \begin{equation}\label{normalizationfactor}
	\begin{aligned}
	r\left(w, \tau_{s} \otimes \pi_0\right)&=\lambda(E / F, \psi)^{k}\times \frac{L\left(s, \phi_{\tau} \otimes \phi_{\pi_0}\right)}{\varepsilon\left(s, \phi_{\tau} \otimes \phi_{\pi_0}, \psi\right) L\left(1+s, \phi_{\tau} \otimes \phi_{\pi_0}\right)}\\ 
	&\quad \times \frac{L\left(2 s, \wedge^{2} \circ \phi_{\tau}\right)}{\varepsilon\left(2 s, \wedge^{2} \circ \phi_{\tau}, \psi\right) L\left(1+2 s, \wedge^{2} \circ \phi_{\tau}\right)},\\
	r\left(w^\prime, \tau_{s} \otimes \sigma_0\right)
	&= \frac{L\left(s, \phi_{\tau} \otimes \phi_{\sigma_0}\right)}{\varepsilon\left(s, \phi_{\tau} \otimes \phi_{\sigma_0}, \psi\right) L\left(1+s, \phi_{\tau} \otimes \phi_{\sigma_0}\right)} \\
	& \quad \times \frac{L\left(2 s, \wedge^{2} \circ \phi_{\tau}\right)}{\varepsilon\left(2 s, \wedge^{2} \circ \phi_{\tau}, \psi\right) L\left(1+2 s, \wedge^{2} \circ \phi_{\tau}\right)},&
	\end{aligned}
	\end{equation}
	where $\lambda(E/F, \psi)$ is the Langlands $\lambda$-factor associated to $E=F(\sqrt{d})$, and $\phi_{\tau}, \phi_{\sigma_0}$ and $\phi_{\pi_0}$ are the $L$-parameters for $\tau,\sigma_0$ and $\pi_0$ respectively. Note that 
	\begin{align}\label{12345}
	\lambda(E/F, \psi)^2=\chi_{V}(-1).
	\end{align}  
\end{itemize}
\begin{remark}
\begin{enumerate}
	\item The representatives $\widetilde{w}_{c^\prime}$ and $\widetilde{w}^\prime_1$ are chosen according to an $F$-splittings of $\mathrm O(V)$ and $\SP(W)$; see \cite[\S 6.2]{MR3788848} for details. 
	\item In Arthur \cite{MR3135650}, the Haar measures of $u_P$ and $u_Q$ are chosen with respect to an $F$-splittings of $\mathrm O(V)$ and $Sp(W)$. Readers can consult \cite[\S 6.3]{MR3788848} for details. The factors $|c^\prime|_F^{k\rho_{P}}$ appear because of the different choices of the Haar measures. 
\end{enumerate}
\end{remark}
\begin{lemma}\label{homomorphicatsequal0}
	$ \mathcal{R}_{\mathfrak W_{c^\prime}}\left(w, \tau_{s} \otimes \pi_0\right) $ and  $\mathcal{R}_{\mathfrak W^\prime_{\psi,1}}\left(w^{\prime}, \tau_{s} \otimes \sigma_0 \right)$ are holomorphic at $s=0$ and 
	\begin{align*}
	\mathcal{R}_{\mathfrak W_{c^\prime}}\left(w, w(\tau_{s} \otimes \pi_0)\right)&\circ \mathcal{R}_{\mathfrak W_{c^\prime}}\left(w, \tau_{s} \otimes \pi_0\right)=1, \\
	\mathcal{R}_{\mathfrak W^\prime_{\psi,1}}\left(w^\prime, w^\prime(\tau_{s} \otimes \sigma_0)\right)&\circ \mathcal{R}_{\mathfrak W^\prime_{\psi,1}}\left(w^\prime, \tau_{s} \otimes \sigma_0\right)=1.
	\end{align*}
\end{lemma}
\begin{proof}
	The case when $G$ is symplectic group or quasi-split even orthogonal group is proved in \cite[Proposition 2.3.1]{MR3135650}. When $G$ is an even orthogonal group, we will give a proof in Appendix \ref{appendixa1} based on the explicit formula for Plancherel measures. 
\end{proof}

Now assume that $w(\tau \otimes \pi_0)\cong \tau\otimes \pi_0$ and $w^\prime(\tau \otimes \sigma_0)\cong \tau\otimes \sigma_0$, both of which are equivalent to $\tau^\vee \cong \tau$. We take  
\begin{align*}
\mathcal{A}_{w}=\mathcal A_{\tau,w}\otimes\mathbf 1_{\mathscr V_{\pi_0}} : \mathscr{V}_{\tau} \otimes \mathscr{V}_{\pi_0} \rightarrow \mathscr{V}_{\tau} \otimes \mathscr{V}_{\pi_0}
\end{align*}
and 
\begin{align*}
\mathcal{A}_{w^\prime}=\mathcal A_{\tau,w^\prime}\otimes\mathbf 1_{\mathscr V_{\sigma_0}} : \mathscr{V}_{\tau} \otimes \mathscr{V}_{\sigma_0} \rightarrow \mathscr{V}_{\tau} \otimes \mathscr{V}_{\sigma_0}
\end{align*}
be the unique intertwining isomorphisms such that 
\begin{itemize}
	\item for any $m_P \in M_{P}, m_Q\in M_Q$, 
	\begin{align*}
	\mathcal{A}_{w} \circ w(\tau \otimes \pi_0)(m_P)=(\tau \otimes \pi_0)(m_P) \circ \mathcal{A}_{w}
	\end{align*}
	and 
	\begin{align*}
	\mathcal{A}_{w^\prime} \circ w^\prime(\tau \otimes \sigma_0)(m_Q)=(\tau \otimes \sigma_0)(m_Q) \circ \mathcal{A}_{w^\prime}.
	\end{align*}
\item $\Lambda \circ \mathcal{A}_{\tau,w}=\Lambda$ and $\Lambda \circ \mathcal{A}_{\tau,w^\prime}=\Lambda$ for a fixed  non-zero Whittaker functional $\Lambda$ on $\mathscr V_{\tau}$ with respect to the Whittaker datum $\left(B_{k}, \psi_{U_{k}}\right),$ where $B_{k}$ is the Borel subgroup consisting of upper triangular matrices in $\mathrm{GL}_{k}(F)$ and $\psi_{U_{k}}$ is the generic character of the unipotent radical $U_{k}$ of $B_{k}$ given by $\psi_{U_{k}}(x)=\psi\left(x_{1,2}+\cdots+x_{k-1, k}\right)$. Here we identify $\GL(X)$ and $\GL(Y)$ with $\GL_k(F)$ via the identification described at the beginning of this subsection. 
\end{itemize}
Note that $\mathcal{A}_{w}^2=\mathbf 1_{\mathscr V_{\tau}\otimes \mathscr V_{\pi_0}}$ and $\mathcal{A}_{w^\prime}^2=\mathbf 1_{\mathscr V_{\tau}\otimes \mathscr V_{\sigma_0}}$. We define the self-intertiwining operators 
\begin{align*}
R_{\mathfrak W_{c^\prime}}(w, \tau \otimes \pi_0) &: \Ind_{P}^{\mathrm{O}(V)}(\tau\otimes \pi_0)\rightarrow  \Ind_{P}^{\mathrm{O}(V)}(\tau \otimes \pi_0),\\
R_{\mathfrak W^\prime_{\psi,1}}(w^{\prime}, \tau \otimes \sigma_0) &: \Ind_{Q}^{\SP(W)}(\tau \otimes \sigma_0) \rightarrow \Ind_{Q}^{\SP(W)}(\tau \otimes \sigma_0)
\end{align*}
by 
\begin{equation}\label{100}
\begin{aligned} 
R_{\mathfrak W_{c^\prime}}(w, \tau \otimes \pi_0) \Psi(h) &=\mathcal{A}_{w}\left(\mathcal{R}_{\mathfrak W_{c^\prime}}(w, \tau_s \otimes \pi_0) \Psi_s(h)|_{s=0}\right),\\ 
R_{\mathfrak W^\prime_{\psi,1}}(w^{\prime}, \tau \otimes \sigma_0) \Phi(g) &=\mathcal{A}_{w^{\prime}}\left(\mathcal{R}_{\mathfrak W^\prime_{\psi,1}}(w^{\prime}, \tau_s \otimes \sigma_0 ) \Phi_s(g)|_{s=0}\right),
\end{aligned}
\end{equation}
where $\Psi_s\in \Ind_{P}^{\mathrm{O}(V)}\left(\tau_s\otimes \pi_0\right)$ and $\Phi_s\in \Ind_{Q}^{\SP(W)}\left(\tau_s\otimes \sigma_0\right)$ are holomorphic sections satisfying $\Psi_s|_{s=0}=\Psi$ and $\Phi_s|_{s=0}=\Phi$ respectively. By Lemma \ref{homomorphicatsequal0} and our construction, we have  
\begin{equation}\label{twiceequalone}
\begin{aligned}
R_{\mathfrak W_{c^\prime}}(w, \tau \otimes \pi_0)^2&=1,\\
R_{\mathfrak W^\prime_{\psi,1}}(w^{\prime}, \tau \otimes \sigma_0)^2&=1.
\end{aligned}
\end{equation}

\begin{remark}
\begin{enumerate}
	\item The definition of the self-intertwining operator $R_{\mathfrak W_{c^\prime}}(w, \tau \otimes \pi_0)$ involves the additive character $\psi$, but an easy computation shows that different choices of $\psi$ give the same $R_{\mathfrak W_{c^\prime}}(w, \tau \otimes \pi_0)$. One can also check that 
	\begin{align*}
	R_{\mathfrak W_{c_2}}(w, \tau \otimes \pi_0)=R_{\mathfrak W_{c_1}}(w, \tau \otimes \pi_0)\times \omega_{\tau}(c_2/c_1)\times \chi_V(c_2/c_1)^k \quad \mbox{for $c_1,c_2\in F^\times$}. 
	\end{align*}
	Hence $R_{\mathfrak W_{c^\prime}}(w, \tau \otimes \pi_0)$ depends only on the choice of Whittaker datum $\mathfrak W_{c^\prime}$. Similarly, one can show 
	\begin{align*}
	R_{\mathfrak W^\prime_{\psi_a,1}}(w^{\prime}, \tau \otimes \sigma_0)=R_{\mathfrak W^\prime_{\psi,1}}(w^{\prime}, \tau \otimes \sigma_0)\times \omega_{\tau}(a)
	\end{align*}
	for any $a\in F^\times$. In particular, 
	\begin{align*}
	R_{\mathfrak W^\prime_{\psi_a,1}}(w^{\prime}, \tau \otimes \sigma_0)=R_{\mathfrak W^\prime_{\psi,1}}(w^{\prime}, \tau \otimes \sigma_0) \quad \mbox{if $a\in F^{\times 2}$}, 
	\end{align*}
	so $R_{\mathfrak W^\prime_{\psi,1}}(w^{\prime}, \tau \otimes \sigma_0)$ depends only on the choice of the Whittaker datum $\mathfrak W^\prime_{\psi,1}$ of $\SP(W_{2n})$. 
	\item Following \cite[\S 7.3]{MR3573972}, we could also use the normalizing factors $r\left(w, \tau_{s} \otimes \pi_0\right)$ and $r\left(w^\prime, \tau_{s} \otimes \sigma_0\right)$ defined by
	\begin{equation}\label{1134}
	\begin{aligned}
	r\left(w, \tau_{s} \otimes \pi_0\right)&=\lambda(E / F, \psi)^{k}\times \gamma(s,\phi_{\tau}\otimes\phi_{\pi_0},\psi)^{-1}\times \gamma(2s,\wedge^{2} \circ \phi_{\tau},\psi)^{-1},\\
	r\left(w^\prime, \tau_{s} \otimes \sigma_0\right)&= \gamma(s,\phi_{\tau}\otimes\phi_{\sigma_0},\psi)^{-1}\times \gamma(2s,\wedge^{2} \circ \phi_{\tau},\psi)^{-1}. 
	\end{aligned}
	\end{equation}
	These normalizing factors are not the same as the normalizing factors we defined in (\ref{normalizationfactor}). But the normalizing factors defined in (\ref{1134}) and (\ref{normalizationfactor}) do share the same analytic behaviours near $s=0$. So the final self-intertwining operators $R_{\mathfrak W_{c^\prime}}(w, \tau \otimes \pi_0)$ and $R_{\mathfrak W^\prime_{\psi,1}}(w^{\prime}, \tau \otimes \sigma_0)$ will not change if we use the normalizing factors defined in (\ref{1134}).  
\end{enumerate}
\end{remark}

We close this subsection by comparing 
$R_{\mathfrak W_{c^\prime}}(w,\tau\otimes(\pi_0\otimes\det))$ and $R_{\mathfrak W_{c^\prime}}(w,\tau\otimes\pi_0)$. Note that there is an isomorphism of $\mathrm O(V)$-representations: 
\begin{equation}
\begin{aligned}\label{59}
\mathcal P_s:\Ind_{P}^{\mathrm O(V)}(\tau_s\otimes \pi_0)\otimes
\det  &\cong  \Ind_{P}^{\mathrm O(V)}(\tau_s\otimes (\pi_0\otimes\det))\\
\Psi_s &\mapsto \mathcal P_s(\Psi_s)
\end{aligned}
\end{equation}
for $\Psi_s\in \Ind_{P}^{\mathrm O(V)}(\tau_s\otimes \pi_0)$, where 
\begin{align*}
\mathcal P_s(\Psi_s)(h)=\Psi_s(h)\det(h)\quad \mbox{for $h\in \mathrm O(V)$}. 
\end{align*}
Here we realize $\Ind_{P}^{\mathrm O(V)}(\tau_s\otimes \pi_0)\otimes
\det$ on the same space with $\Ind_{P}^{\mathrm O(V)}(\tau_s\otimes \pi_0)$, but with the action twisted by $\det$. If we identify these two representations via this isomorphism, then 
\begin{align*}
\mathcal{M}\left(\widetilde{w}_{c^\prime}, \tau_{s} \otimes (\pi_0\otimes\det) \right) (\Psi_{s})\left(h\right) &=\int_{U_{P}} \Psi_{s}\left(\widetilde{w}_{c^\prime}^{-1} u_{P} h\right)\det\left(\widetilde{w}_{c^\prime}^{-1} u_{P} h\right) d u_{P}\\
&=\det(\widetilde{w}_{c^\prime}^{-1})\cdot \det(h)\cdot \int_{U_{P}} \Psi_{s}\left(\widetilde{w}_{c^\prime}^{-1} u_{P} h\right) d u_{P}\\
&=(-1)^{\dim \phi_\tau}\cdot \det(h)\cdot \int_{U_{P}} \Psi_{s}\left(\widetilde{w}_{c^\prime}^{-1} u_{P} h\right) d u_{P}\\
&=(-1)^{\dim \phi_\tau}\cdot  \mathcal{M}\left(\widetilde{w}_{c^\prime}, \tau_{s} \otimes \pi\right) \Psi_{s}(h)\cdot \det(h). 
\end{align*}
Here we use the equation (\ref{determinantminus1}). This implies
\begin{equation}\label{compareunnormlize}
\mathcal{M}\left(\widetilde{w}_{c^\prime}, \tau_{s} \otimes (\pi\otimes\det) \right) =(-1)^{\dim \phi_\tau} \mathcal{M}\left(\widetilde{w}_{c^\prime}, \tau_{s} \otimes \pi \right) 
\end{equation}
via the isomorphism $\mathcal P_s$. Since the Langlands paramters for $\pi_0$ and $\pi_0\otimes\det$ are the same, $\mathcal{M}\left(\widetilde{w}_{c^\prime}, \tau_{s} \otimes \pi_0) \right) $ and $\mathcal{M}\left(\widetilde{w}_{c^\prime}, \tau_{s} \otimes (\pi_0\otimes\det) \right) $ share the same normalizing factors. So we deduce from (\ref{compareunnormlize}) that  
\begin{equation}\label{comparenormalize}
R_{\mathfrak W_{c^\prime}}(w,\tau\otimes(\pi_0\otimes\det))=(-1)^{\dim \phi_\tau}R_{\mathfrak W_{c^\prime}}(w,\tau\otimes\pi_0) .
\end{equation}

\subsection{Weil representations and mixed models}\label{mixed model}
In this section, we recall some explicit formulas for the Weil representations. 
\begin{comment}
Let $\mathbb W$ be a finite dimensional vector space equipped with a non-degenerate symplectic form $\langle\cdot, \cdot\rangle _{\mathbb W}: \mathbb W\times \mathbb W \rightarrow F$. Let  $\mathcal{H}(\mathbb W)=\mathbb W \oplus F$ be the associated Heisenberg group, i.e., the multiplication law is given by
$$(w, t) \cdot\left(w^{\prime}, t^{\prime}\right)=\left(w+w^{\prime}, t+t^{\prime}+\frac{1}{2}\left\langle w, w^{\prime}\right\rangle_{\mathbb W}\right)$$
for $w, w^{\prime} \in W$ and $t, t^{\prime} \in F$. Fix a maximal totally isotropic subspaces $\mathbb X$ and $\mathbb X^\vee$ such that $\mathbb W=\mathbb X\oplus \mathbb X^\vee$. Let $\rho$ be the Heisenberg representation of $\mathcal H(\mathbb W)$ on $\mathscr S(\mathbb X^\vee)$ with central character $\psi$. Namely,
$$\rho\left(x+x^{\prime}, t\right) \varphi\left(x_0^\prime\right)=\psi\left(t+\left\langle x_0, x\right\rangle_{\mathbb W}+\frac{1}{2}\left\langle x^{\prime}, x\right\rangle_{\mathbb W}\right) \varphi\left(x_0^\prime+x^{\prime}\right)$$
for $\varphi \in \mathscr{S}(\mathbb X^\vee), x \in \mathbb X, x^{\prime}, x_0^\prime \in \mathbb X^\vee$ and $t \in F$. 
\end{comment}
We retain the notation in Subsection \ref{normalizingintertwing}.  For simplicity, we write:
\begin{itemize}
	\item  $\omega_{00}$ for the Weil representation $\omega_{\psi, V_0, W_0}$ of $\mathrm{Sp}(W_0) \times \mathrm{O}(V_0)$ on a space $\mathscr{S}_{00}$;
	\item $\omega_{0}$ for the Weil representation $\omega_{\psi, V_0, W}$ of $\operatorname{Sp}(W) \times \mathrm{O}\left(V_0\right)$ on a space $\mathscr{S}_{0}$;
	\item $\omega$ for the Weil representation $\omega_{\psi, V, W}$ of $\operatorname{Sp}\left(W\right) \times \mathrm{O}\left(V\right)$ on a space $\mathscr{S}$.
\end{itemize}
We take a mixed model 
$$\mathscr{S}_0=\mathscr{S}\left(V_0\otimes Y^\vee\right)\otimes\mathscr{S}_{00} $$
of $\omega_0$, where we regard $\mathscr{S}_0$ as a space of functions on $V_0 \otimes Y^\vee$ with values in $\mathscr{S}_{00}$. Similarly, we take a mixed model 
$$\mathscr{S}=\mathscr{S}\left(X^\vee \otimes W\right) \otimes \mathscr{S}_0$$
of $\omega$, where we regard $\mathscr{S}$ as a space of functions on $X^\vee\otimes W$ with values in $\mathscr{S}_0$. Also, we write 

$\bullet \rho_{00}$ for the Heisenberg representation of $\mathcal{H}(V_0 \otimes W_0)$ on $\mathscr{S}_{00}$ with the central character $\psi$;

$\bullet \rho_0$ for the Heisenberg representation of $\mathcal{H}\left(V_0 \otimes W\right)$ on $\mathscr{S}_0$ with the central character $\psi$.
\begin{comment}
$\bullet \rho$ for the Heisenberg representation of $\mathcal{H}\left(V \otimes W\right)$ on $\mathcal{S}$ with the central character $\psi$ .
\end{comment}

Similarly to those of \cite[Lemma 6.2,6.3,6.4]{MR3788848} and \cite[\S 7.4]{MR3573972}, we obtain some explicit formulas for these Weil representations. 
\begin{comment}
For $\varphi_{00} \in \mathcal{S}_{00}$ and $x \in V_0 \otimes Y_{n_0}^{*}$
$$\begin{array}{rlr}{[\omega_{00}(h_0) \varphi_{00}](x)} & {=\varphi_{00}\left(h_{0}^{-1} x\right),} & {h_0 \in \mathrm{O}(V_0)} \\
{\left[\omega_{00}\left(m(a)\right) \varphi_{00}\right](x)} & {=\chi_{V}\left(\operatorname{det}(a)\right)\left|\operatorname{det}(a)\right|_{F}^{n_0} \varphi_{00}\left(a^{*} x\right),} & {a \in \operatorname{GL}\left(Y_{n_0}\right)} \\ {\left[\omega_{00}\left(u(c)\right) \varphi_{00}\right](x)} & {=\psi\left(\frac{1}{2}\left\langle c x, x\right\rangle\right) \varphi_{00}(x),} & {c \in \operatorname{Sym}\left(Y_{n_0}^{*}, Y_{n_0}\right)} \\ 
{[\omega_{00}(w) \varphi_{00}](x)} & {=\gamma_{V}^{-n} \int_{V_0 \otimes Y_{n_0}} \varphi_{00}\left(I^{-1} z\right) \psi(-\langle z, x\rangle) d z}\end{array}$$
Here 
$$
w=\left (\begin{array}{cc}
0& -I_{Y_{n_0}}\\
I_{Y_{n_0}}^{-1}& 0
\end{array}\right)
$$
and  $\gamma_{V}$ is a 4 th root of unity satisfying $\gamma_{V}^{2}=\chi_{V}(-1)$. 
%$dz$ is the self-dual measure on $V_0\otimes Y_{n_0}$ with respect to the pairing $(x,y)\mapsto \psi(\langle y, I^{-1}x \rangle )$, and
\end{comment}

For $\varphi_0 \in \mathscr{S}_0$ and $x\in V_0\otimes Y^\vee$, we have 
\begin{align*}
\left[\omega_0(g_0) \varphi_0\right]\left(x\right) & =\omega_{00}(g_0)\varphi_0(x), & g \in \mathrm{Sp}(W_0), \\
\left[\omega_0(h_0) \varphi_0\right](x) & =\omega_{00}(h_0)\varphi_0(h_{0}^{-1}x),  &  h_0 \in  \mathrm{O}(V_0), \\
\left[\omega_0\left( m_{Q}(a)\right) \varphi_0\right]\left(x\right) & =\chi_{V}(\det a)|\operatorname{det}(a)|_{F}^{n_0} \varphi_0(a^{*} x), & {a \in \mathrm{GL}(Y)} ,\\ 
\left[\omega_0\left( u_{Q}(b)\right) \varphi_0\right]\left(x\right) & =\rho_{00}\left(b^{*} x, 0\right)\varphi_0(x) , & {b \in \operatorname{Hom}(W_0, Y)},\\
\left[\omega_0\left( u_{Q}(c)\right) \varphi_0\right]\left(x\right) &=\psi\left(\frac{1}{2}\left\langle c x, x\right\rangle\right) \varphi_0\left(x\right),  & c \in \operatorname{Sym}\left(Y^\vee, Y\right) ,\\
\left[\omega_0\left(w_{Q}\right) \varphi_0\right]\left(x\right) &=\gamma_{V}^{-k} \int_{V_0 \otimes Y} \varphi_0\left(I_{Y}^{-1} z\right) \psi\left(-\left\langle z, x\right\rangle\right) d z, \end{align*}
where $\gamma_{V}$ is the Weil constant associated to $V$, which is a $4$-th root of unity satisfying $\gamma_{V}^{2}=\chi_{V}(-1)$.

For $\varphi\in  \mathscr{S}$ and $x\in X^\vee\otimes W$, we have 
\begin{align*}
\left[\omega(g) \varphi\right]\left(x \right)&= \omega_0(g)\varphi(g^{-1}x),  & g \in \mathrm{Sp}(W) , \\
\left[\omega(h_0) \varphi\right](x) & =\omega_0(h_0)\varphi(x), & {h \in \mathrm{O}\left(V_0\right)}, \\ 
{\left[\omega\left(m_{P}\left(a\right)\right) \varphi\right]\left(x\right)} & =\left|\operatorname{det}\left(a\right)\right|_{F}^{n} \varphi(a^{ *}x), & {a \in \mathrm{GL}(X)},\\
{\left[\omega\left(u_{P}\left(b\right)\right) \varphi\right]\left(x \right)}& =\rho_{0}\left(b^{ *} x, 0\right) \varphi\left(x\right),  & {b \in \operatorname{Hom}(V_0, X)}, \\ \left[\omega\left(u_{P}\left(c\right)\right) \varphi\right]\left(x\right) & =\psi\left(\frac{1}{2}\left\langle c x, x\right\rangle\right) \varphi\left(x\right) , & c \in \operatorname{Sym}\left(X^\vee, X\right),\\
\left[\omega\left(w_{P}\right) \varphi\right]\left(x \right)&=\ \int_{X\otimes W } \varphi\left(I_{X}^{-1} z\right) \psi\left(-\left\langle z, x\right\rangle\right) d z. 
 \end{align*}
%and 
%$$\begin{aligned}\left[\rho_{0}\left(\left(y+y^{\prime}, 0\right)\right) \varphi_0\right](x)&=\psi\left(\langle x, y\rangle+\frac{1}{2}\left\langle y^{\prime}, y\right\rangle\right) \varphi_0\left(x+y^{\prime}\right), & y \in V_0 \otimes Y \\
%&{} &y^{\prime} \in V_0 \otimes Y^\vee \\
%\left[\rho_0\left(\left(y_{0}, 0\right) \right)\varphi_0\right](x),&=\rho_{00}\left(\left(y_{0}, 0\right)\right) \varphi_0(x) & y_{0} \in V_0 \otimes W_0 \end{aligned}$$
%For $\varphi_0\in S_0$ and $x\in V_0\otimes Y^\vee$.

%$$\left[\rho^{\prime}\left(v+v_{0}+v^{*}, 0\right) \varphi^{\prime}\right]\left(x_{2}, x_{3}\right)=\psi\left(\left\langle x_{3}, v\right\rangle+\frac{1}{2}\left\langle v^{*}, v\right\rangle\right)\left[\rho\left(v_{0}, 0\right) \varphi_{2}\right]\left(x_{2}\right) \cdot \varphi_{3}\left(x_{3}+v^{*}\right)$$

%Combine the above explicit formulas, we get a formula for the mixed model $\mathcal{S}^{\prime \prime}=\mathcal{S}\left(V^{\prime} \otimes Y^\vee\right) \otimes \mathcal{S}\left(V \otimes Y_{n}^{*}\right) \otimes \mathcal{S}\left(X^\vee \otimes W\right)$ of $\omega^{''}$

\subsection{Construction of the equivariant map}\label{constructionof}
Recall that we put 
$$V=X\oplus V_0\oplus X^{\vee}, \quad W=Y\oplus W_0\oplus Y^{\vee}$$
with $\dim(X)=\dim(Y)=k $ in Subsection \ref{normalizingintertwing}. Using the basis $\left\{v_{1}, \ldots, v_{k}\right\}$ of $X$ (resp. $\left\{w_{1}, \ldots, w_{k}\right\}$ of $Y$ ), we identify $\GL(X)$ (resp. $\GL(Y)$) with $\GL_k(F)$. Hence we can define an isomorphism $i: \GL(X) \rightarrow \GL(Y) $ via these identifications. Put 
$$e=v_{1} \otimes w_{1}^{*}+\cdots+v_{k} \otimes w_{k}^{*} \in X \otimes Y^{\vee} \quad \mbox{and}\quad  e^{*}=v_{1}^{*} \otimes w_{1}+\cdots+v_{k}^{*} \otimes w_{k} \in X^{\vee} \otimes Y.$$
%Then $i(a)^{c} e=a^{*} e$ and $\left(i(a)^{c}\right)^{*} e^{*}=a e^{*}$ for $a \in \mathrm{GL}(X)$

For $\varphi\in \mathscr S =\mathscr S(X^{\vee}\otimes W)\otimes \mathscr S_0$, we define functions $\mathfrak{f}(\varphi), \hat{\mathfrak{f}}(\varphi)$ on $\SP(W)\times \mathrm{O}(V)$ with values in $\mathscr S_0$ by 
\begin{align*}
\mathfrak{f}(\varphi)(g ,h)&=\left(\omega(g,h) \varphi\right)\left(
\begin{array}{c}{\hspace{0.4em}e^*}\\ {0} \\ {0}\end{array}\right),\\
\hat{\mathfrak{f}}(\varphi)(g, h)&=\int_{X^{\vee} \otimes Y}\left(\omega(g ,h) \varphi\right)\left(\begin{array}{c}{x} \\ {0} \\ {0}\end{array}\right) \psi\left(\left\langle x, e\right\rangle\right) d x
\end{align*}
for $g\in \SP(W)$ and $h \in \mathrm{O}(V) $. Here, we write an element in $X^{\vee}\otimes W$ as a block
matrix
$$\left(\begin{array}{l}{y_{1}} \\ {y_{2}} \\ {y_{3}}\end{array}\right)$$
with $y_{1} \in X^{\vee} \otimes Y, y_{2} \in X^{\vee} \otimes W_0$ and $y_{3} \in X^{\vee} \otimes Y^{\vee}$. We also define functions $f(\varphi), \hat f(\varphi)$ on $\SP(W)\times \mathrm {O}(V)$ with values in $\mathscr S_{00}$ by 
$$\begin{aligned} 
f(\varphi)(g,h) &=\operatorname{ev}(\mathfrak{f}\left(\varphi(g, h)\right), \\ \hat{f}(\varphi)(g ,h) &=\operatorname{ev}\left (\hat{\mathfrak{f}}(\varphi)(g, h)\right), \end{aligned}$$
where
$$\operatorname{ev}: \mathscr S_0=\mathscr S(V_0\otimes Y^{\vee})\otimes \mathscr S_{00}\rightarrow \mathscr S_{00}$$
 is the evaluation at $0\in V_0\otimes Y^{\vee}$. If $f=f(\varphi )$ or $\hat f(\varphi)$, then 
 \begin{align*}
 f(u_{Q} g, u_{P} h)=&f(g, h),  &{u_{P} \in U_{P}, u_{Q} \in U_{Q}} ,\\ 
f(g_0g,h_0h) =&\omega_{00}(g_0,h_0) f(g,h),  h_0 \in \mathrm{O}(V_0), &g \in \mathrm{Sp}(W_0), \\ 
f(m_{Q}(i(a)) g, m_{P}(a) h)  =&\chi_{V}(\operatorname{det}(a))|\det(a)|_{F}^{\rho_P+\rho_Q} f\left(g,h\right),   &a \in \GL(X). 
 \end{align*}

Let $\tau$ be an irreducible (unitary) tempered representation of $\GL_{k}(F)$ on a space $\mathscr{V}_{\tau}$. We may regard $\tau$ as a representation of $\GL(X)$ or $\GL(Y)$ via the above identifications. Let $\sigma_0$ and $\pi_0$ be irreducible tempered representations of $\SP(W_0)$ and $\mathrm{O}(V_0)$ on spaces $\mathscr{V}_{\sigma_0}$ and $\mathscr{V}_{\pi_0}$, respectively. Fix nonzero invariant non-degenerate bilinear forms $\langle\cdot, \cdot\rangle$ on $\mathscr{V}_{\tau} \times \mathscr{V}_{\tau^ {\vee}}, \mathscr{V}_{\sigma_0} \times \mathscr{V}_{\sigma_{0}^{\vee}}$ and $\mathscr{V}_{\pi_0} \times \mathscr{V}_{\pi_{0}^{\vee}}$. Let
\begin{align*}
\langle\cdot, \cdot\rangle :\left(\mathscr{V}_{\tau} \otimes \mathscr{V}_{\sigma_{0}^{\vee}}\right) \times \mathscr{V}_{\tau^{\vee}} &\rightarrow \mathscr{V}_{\sigma_{0}^{\vee}},\\
\langle\cdot, \cdot\rangle :\left(\mathscr{V}_{\tau} \otimes \mathscr{V}_{\pi_{0}^{\vee}}\right) \times \mathscr{V}_{\tau^{\vee}} &\rightarrow \mathscr{V}_{\pi_{0}^{\vee}}
\end{align*}
be the induced maps. 

Now assume that 
$$
\sigma_0=\theta_{V_0,W_0,\psi}(\pi_0)\neq 0. 
$$
We fix a nonzero $\SP(W_0)\times \mathrm O(V_0)$-equivariant map 
$$
\mathcal{T}_{00} : \omega_{00} \otimes \sigma_{0}^{\vee} \rightarrow \pi_0. 
$$
For $\varphi\in \mathscr S, \Phi_{s}\in \Ind_{Q}^{\SP(W)}(\tau_s\chi_V\otimes \sigma_{0}^{\vee}), h\in \mathrm O(V),\check{v}\in \mathscr{V}_{\tau^{\vee}}$ and $\check{v}_{0}\in \mathscr{V}_{\pi_{0}^{\vee}}$, put 
\begin{align*}
&\left\langle\mathcal{T}_{s}\left(\varphi\otimes \Phi_{s}\right)(h), \check{v} \otimes \check{v}_{0}\right\rangle\\
&\qquad=L(s, \tau)^{-1} \times \int_{U_{Q} \SP(W_0) \backslash \SP\left(W\right)}\left\langle\mathcal{T}_{00}(\hat{f}(\varphi)\left(g, h\right)\otimes \left\langle\Phi_{s}\left(g\right), \check{v}\right\rangle), \check{v}_{0}\right\rangle d g.
\end{align*}
Note that $\left\langle\Phi_{s}\left(g\right), \check{v}\right\rangle \in \mathscr{V}_{\sigma_{0}^{\vee}}$. 

\begin{proposition}\label{intertwingmap}
	We have :
	\begin{enumerate}
		\item The integral $\left\langle\mathcal{T}_{s}\left(\varphi\otimes \Phi_{s}\right)\left(h\right), \check{v} \otimes \check{v}_{0}\right\rangle$ is absolutely convergent for $\Re s>0$ and admits a holomorphic continuation to $\mathbb C$. 
	    \item For $\Re(s)<1$, we have 
	    \begin{align*}
	    &\left\langle\mathcal{T}_{s}\left(\varphi\otimes \Phi_{s}\right)\left(h\right), \check{v} \otimes \check{v}_{0}\right\rangle\\
	    &\qquad=L(s, \tau)^{-1} \gamma(s, \tau, \psi)^{-1}\\
	    &\qquad\qquad \times\int_{U_{Q} \mathbb{\SP}(W_0) \backslash \SP\left(W\right)}\left\langle\mathcal{T}_{00}\left(f(\varphi )\left(g ,h\right)\otimes\left\langle\Phi_{s}\left(g\right), \check{v}\right\rangle\right), \check{v}_{0}\right\rangle d g.
	    \end{align*}
	    \item  The map 
	    $$\mathcal{T}_{s} : \omega\otimes \operatorname{Ind}_{Q}^{\SP\left(W\right)}\left(\tau_{s}\chi_{V} \otimes \sigma_{0}^{\vee}\right) \rightarrow \operatorname{Ind}_{P}^{\mathrm{O}\left(V\right)}\left(\tau_{s}  \otimes \pi_0\right)$$
	    is $\SP(W)\times \mathrm O(V)$-equivariant. 
	    \item For $\Phi \in \operatorname{Ind}_{Q}^{\mathrm{\SP}\left(W\right)}\left(\tau\chi_{V} \otimes \sigma_{0}^{\vee}\right)$ with $\Phi \neq 0,$ there exists $\varphi \in \mathscr S$ such that
	    $$\mathcal{T}_{0}(\varphi, \Phi) \neq 0.$$
	\end{enumerate}
\end{proposition}
\begin{proof}
	The proof is similar to that of \cite[Proposition 7.2]{MR3788848}. 
\end{proof}

\subsection{Compatibilities with intertwining operators}
Now we shall prove a key property of the equivariant map we have constructed. Having fixed $\tau,\pi_0$ and $\sigma_0 =\theta_{\psi, W_0, V_0}(\pi_0) \neq 0$, we let  
$$\mathcal{M}(\widetilde{w}_{c^\prime}, s)=\mathcal{M}\left(\widetilde{w}_{c^\prime}, \tau_{s} \otimes \pi_0\right) \quad \mbox{and} \quad \mathcal{M}\left(\widetilde{w}_1^{\prime}, s\right)=\mathcal{M}\left(\widetilde{w}_1^\prime, \tau_{s}\chi_{V} \otimes \sigma_{0}^{\vee}\right)$$
be the unnormalized intertwining operators defined in Subsection \ref{normalizingintertwing}. By the Howe duality, the diagram 
$$
\begin{CD}
\omega \otimes \operatorname{Ind}_{Q}^{\SP\left(W\right)}\left(\tau_{s}\chi_{V} \otimes \sigma_{0}^{\vee}\right) @>\mathcal T_{s}>> \operatorname{Ind}_{P}^{\mathrm{O}\left(V\right)}\left(\tau_{s}  \otimes \pi_0\right)\\
@VV1\otimes\mathcal{M}(\widetilde{w}_1^{\prime}, s) V @VV\mathcal{M}(\widetilde{w}_{c^\prime}, s)V\\
\omega \otimes \operatorname{Ind}_{Q}^{\SP\left(W\right)}\left(w^{\prime}(\tau_{s}\chi_{V} \otimes \sigma_{0}^{\vee})\right) @>\mathcal T_{-s}>> \operatorname{Ind}_{P}^{\mathrm{O}\left(V\right)}\left(w(\tau_{s} \otimes \pi_0)\right) 
\end{CD}
$$
commutes up to a scalar. The following proposition determines this constant of proportionality explicitly.
\begin{proposition}\label{commuteunnormalize}
	For $\varphi\in \mathscr S $ and $\Phi_{s}\in \Ind_{Q}^{\SP(W)}(\tau_s\chi_{V}\otimes \sigma_{0}^{\vee})$, we have 
	\begin{align*}
	\mathcal{M}\left(\widetilde{w}_{c^\prime}, s\right) \mathcal{T}_{s}\left(\varphi\otimes \Phi_{s}\right)=& \omega_{\tau}(-1/c^\prime) \times |c^\prime|_F^{-k(\rho_P+s)} \times \gamma_{V}^{k} \cdot\chi_{V}(-1)^{k}\cdot L(s, \tau)^{-1}\\ \times& L\left(-s, \tau^{\vee}\right)\times\gamma\left(-s, \tau^{\vee}, \psi\right) 
	\times \mathcal{T}_{-s}\left(\varphi\otimes \mathcal{M}(\widetilde{w}_1^{\prime}, s) \Phi_{s}\right). 
	\end{align*}
 \end{proposition}
\begin{proof}
	The proof is similar to that of \cite[Proposition 8.4]{MR3573972}. 
\end{proof}	

As a consequence of Proposition \ref{commuteunnormalize}, we deduce:
\begin{corollary}\label{123}
	For $\varphi\in \mathscr S $ and $\Phi_{s}\in \Ind_{Q}^{\SP(W)}(\tau_s\chi_{V}\otimes \sigma_{0}^{\vee})$, we have
	\begin{align*}
	\mathcal{R}_{\mathfrak W_{c^\prime}}\left(w, \tau_{s}\otimes \pi_0\right) \mathcal{T}_{s}\left(\varphi\otimes \Phi_{s}\right)=&\omega_{\tau}(-1/c^\prime)\times \chi_{V}(-c^\prime)^k\times  \alpha(s)\\
	\times&\mathcal{T}_{-s}\left(\varphi\otimes \mathcal{R}_{\mathfrak W^\prime_{\psi,1}}\left(w^{\prime}, \tau_{s}\chi_{V} \otimes \sigma_{0}^{\vee} \right)\Phi_{s}\right), 
	\end{align*}
	where
	$$
	\alpha(s)=|c^\prime|_{F}^{-ks} \times\frac{\varepsilon(-s,\tau^{\vee},\psi)}{\varepsilon(s,\tau,\psi)}.
	$$
	In particular, if $\tau\cong \tau^{\vee}$, then $\alpha(0)=1$ and 
$$
R_{\mathfrak W_{c^\prime}}\left(w, \tau\otimes \pi_0\right) \mathcal{T}_{0}\left(\varphi\otimes \Phi\right)=\omega_{\tau}(-1/c^\prime)\times \chi_{V}(-c^\prime)^k\times \mathcal{T}_{0}\left(\varphi\otimes R_{\mathfrak W^\prime_{\psi,1}}\left(w^{\prime}, \tau\chi_{V} \otimes \sigma_{0}^{\vee} \right)\Phi\right)
$$
for  $\Phi\in \Ind_{Q}^{\SP(W)}(\tau
\chi_{V}\otimes \sigma_{0}^{\vee})$. 
\end{corollary}
\begin{proof}
The corollary immediately follows from Proposition \ref{commuteunnormalize} and the fact that   
$$
\gamma_V^{-1}\times \lambda(E/F,\psi)=\epsilon(V)\times \chi_V(c) .
$$
%where $\gamma_V$ is the Weil constant associated to $V$ which appears on the explicit formula for the Weil representation, and $\lambda(E/F,\psi)$ is the Langlands constant which appears in the normalizing factors. 
\end{proof}

\section{The proof of local intertwining relation}
We begin to prove Theorem \ref{desideratumall} (9), i.e., the local intertwining relation, for our construction. We retain the notation in Subsection \ref{normalizingintertwing}. Let $\phi\in \Para(\mathrm O(V))$ be such that    
$$\phi=\phi_{\tau}+\phi_0+ \phi_{\tau}^{\vee},$$
where $\phi_{\tau}$ is an irreducible tempered representation of $\WD_F$ corresponding to $\tau\in \Irr(\GL_k(F))$ and $\phi_0\in \Para(\mathrm O(V_0))$. In this case, we have a natural embedding $\mathcal S_{\phi_0}\hookrightarrow \mathcal {S}_{\phi}$. Let $\pi_0\in \Pi_{\phi_0}$ be an irreducible tempered representation of $\mathrm O(V_0)$. Our goal is to analyze the induced representation $\Ind_{P}^{\mathrm O(V)}(\tau\otimes \pi_0)$. 

We divide our proof into two parts. In the first part, we analyze the $L$-parameter for each irreducible constituent $\pi$ of $\Ind_{P}^{\mathrm O(V)}(\tau\otimes \pi_0)$, and as a corollary, we get some information on the reducibility of $\Ind_{P}^{\mathrm O(V)}(\tau\otimes \pi_0)$. In the second part, we analyze the characters $\mathcal J_{\mathfrak W_{c^\prime}}(\pi)$ for a fixed choice of Whittaker data $\mathfrak W_{c^\prime}$.

\subsection{$L$-parameter and reducibility}
We first determine the $L$-parameter of $\pi\subseteq \Ind_{P}^{\mathrm O(V)}(\tau\otimes \pi_0)$.
\begin{lemma}\label{parameter}
	Let $\pi$ be an irreducible constituent of $\Ind_{P}^{\mathrm O(V)}(\tau\otimes \pi_0)$. Then the $L$-parameter of $\pi$ is $\phi$. 
\end{lemma}
\begin{proof}
	We divide it into two cases: 
	
	$\bullet$ \underline{Case I}: if $\theta_{W,V,\psi}(\pi)\neq 0 $, let $\sigma=\theta_{W,V,\psi}(\pi)$. Then by Lemma \ref{12}, we have 
	$$
	\sigma \subseteq \Ind_{Q}^{\SP(W)}\left(\tau\chi_{V}\otimes \sigma_{0}\right),
	$$
	where  
	$$\sigma_{0}=
	\theta_{W_0,V_0,\psi}(\pi_0). 
	$$
	 Let $\phi^+$ and $\phi_{0}^+$ be the $L$-parameters of $\sigma$ and $\sigma_{0}$. Then by Theorem \ref{llcsympletic} (6), we have
	$$
	\phi^+=(\phi_{\tau}\otimes \chi_{V})+\phi^+_{0} + (\phi_{\tau}^{\vee}\otimes \chi_{V}).
	$$
	On the other hand, it follows from Proposition \ref{prasad} that 
	\begin{align*}
	\phi^+=(\phi_{\pi}\otimes \chi_{V})+ \chi_{V}\quad \mbox{and}\quad
	\phi^+_{0}=(\phi_{0}\otimes\chi_{V})+ \chi_{V},
	\end{align*}
	where $\phi_{\pi}$ is the $L$-parameter of $\pi$. We deduce $\phi_{\pi}=\phi$ from these equalities. 
	
	$\bullet$ \underline{Case II}: if $\theta_{W,V,\psi}(\pi)= 0 $, then $\theta_{W,V,\psi}(\pi\otimes\det)\neq 0$ by the conservation relation (Theorem \ref{con2}). It follows from (\ref{59}) that 
	\begin{align*}
	\pi\otimes\det \subseteq \Ind_{P}^{\mathrm O(V)}(\tau\otimes (\pi_0\otimes\det)). 
	\end{align*} 
	 Since by Proposition \ref{det},
	\begin{align*}
	\mathcal L(\pi_0\otimes\det)=\mathcal L(\pi_0)=\phi_0, 
	\end{align*}
	we deduce from Case I that $\mathcal L(\pi\otimes\det)=\phi$. Then again by Proposition \ref{det}, we have 
	\begin{align*}
	\mathcal L(\pi)=\mathcal L(\pi\otimes\det)=\phi. 
	\end{align*}
	These complete the proof. 
\end{proof}

Next we analyze the reducibility of $\Ind_{P}^{\mathrm O(V)}(\tau\otimes \pi_0)$.
\begin{lemma}\label{multiplicityfree}
	The representation $\Ind_{P}^{\mathrm O(V)}(\tau\otimes \pi_0)$ is semisimple and multiplicity free. 
\end{lemma}
\begin{proof}
	Since $\tau$ and $\pi_{0}$ are unitary representations, so is $\Ind_{P}^{\mathrm O(V)}(\tau\otimes \pi_0)$. Hence it is semisimple. 
	
	Next we prove $\Ind_{P}^{\mathrm O(V)}(\tau\otimes \pi_0)$ is multiplicity free. Let $\pi\in \Irr(\mathrm O(V))$ be such that 
	\begin{align*}
	m(\pi)=\dim \Hom_{\mathrm O(V)}(\pi, \Ind_{P}^{\mathrm O(V)}(\tau\otimes \pi_0))\geq 1 .
	\end{align*}
	We shall prove $m(\pi)=1$. We separate the proof into two cases, depending on $\theta_{W,V,\psi}(\pi)=0$ or not: 

$\bullet$ \underline{Case I}: If $\theta_{W,V,\psi}(\pi)\neq 0$, let $\sigma=\theta_{ W,V,\psi}(\pi)$. As in the proof of Lemma \ref{12}, there is an injective map 
		\begin{align*}
		\Hom_{\mathrm O(V)}(\omega, \Ind_{P}^{\mathrm O(V)}(\tau\otimes \pi_0)) \hookrightarrow \left(\Ind_{Q}^{\SP(W)}(\tau^\vee\chi_{V}\otimes \sigma_{0})\right)^{\vee}.
		\end{align*}
		Then it is easy to see that 
		$$
		m(\pi) \leq m(\sigma) 
		$$
		where 
		$$
		m(\sigma) =  \dim \Hom_{\SP(W)}\left(\sigma, \Ind_{Q}^{\SP(W)}(\tau^\vee\chi_{V}\otimes \sigma_{0})\right). 
		$$
		By Theorem \ref{llcsympletic} , we have $m(\sigma) \leq 1$. Hence $m(\pi)=1$.

		\begin{comment}
		Let $\Ind_{P}^{\mathrm O(V)}(\tau\otimes \pi_0)=\pi_1\oplus \pi_2\cdots \oplus \pi_l$, since $\theta_{V,W,\psi}(\pi)\neq 0$, $\Hom_{\mathrm O(V)}(\omega, \pi)$ is non-zero and isomorphic to $\sigma^\vee$ 
		So 
		\begin{align*}
		\oplus_{i} \Hom_{\mathrm O(V)}(\omega, m(\pi)\pi)_{\infty} &\hookrightarrow \Hom_{\mathrm O(V)}(\omega, \Ind_{P}^{\mathrm O(V)}(\tau\otimes \pi_0))_{\infty}\\
		&\hookrightarrow \Ind_{Q}^{\SP(W)}(\chi_{V}\tau\otimes \sigma_{0})^{\vee}
		\end{align*}
		By LLC (add) for symplectic groups, the multiplicity of $\sigma^\vee$ in $\Ind_{Q}^{\SP(W)}(\chi_{V}\tau\otimes \sigma_{0})^{\vee}$ is one, hence the multiplicity $m(\pi)$ is at most one. 
		\end{comment}
		
$\bullet$ \underline{Case II}: If $\theta_{ W,V,\psi}(\pi)=0$, then by the conservation relation (Theorem \ref{con2}), $\theta_{ W,V,\psi}(\pi\otimes\det)\neq 0$. Note that by (\ref{59}), we have 
		\begin{align*}
		\pi\otimes\det \subseteq \Ind_{P}^{\mathrm O(V)}(\tau\otimes (\pi_0\otimes\det)). 
		\end{align*} 
		 It then follows from Case I that  
		\begin{align*}
		m(\pi\otimes\det)=\dim \Hom_{\mathrm O(V)}(\pi\otimes\det, \Ind_{P}^{\mathrm O(V)}(\tau\otimes (\pi_0\otimes\det)))=1. 
		\end{align*}
		On the other hand, by (\ref{59}), we have $m(\pi)=m(\pi\otimes \det)=1$. This proves Case II. 
\end{proof}
\begin{remark}
	This lemma also follows from the abelianess of the Knapp-Stein $R$-group and the induction in stages. The abelianess of the Knapp-Stein $R$-group was proved by Goldberg \cite{MR1296726} for the quasi-split special orthogonal groups. Later Choiy--Goldberg \cite[Appendix A]{MR3430367} proved that the Knapp-Stein $R$-group for the non quasi-split special orthogonal groups is isomorphic to its quasi-split pure inner forms. But their proof is under the assumption of the LLC for pure inner forms of quasi-split special orthogonal groups. Since the main purpose of this paper is to prove the LLC for pure inner forms, to avoid a circular reasoning, we give a proof of the lemma here. 
\end{remark}

Now we analyze the reducibility of $\Ind_{P}^{\mathrm O(V)}(\tau\otimes \pi_0)$. Recall that there is a natural embedding $\mathcal S_{\phi_{0}}\hookrightarrow \mathcal {S}_{\phi}$ of component groups. We consider two cases depending on the relative size of $\mathcal S_{\phi_0}$ and $\mathcal {S}_{\phi}$.  
\begin{corollary}\label{embeddingisisomorphism}
	Assume that $\phi_{\tau}\subseteq \phi_0$, so the natural embedding $\mathcal S_{\phi_{0}}\hookrightarrow \mathcal {S}_{\phi}$ is an isomorphism. Then the induced representation $\Ind_{P}^{\mathrm O(V)}(\tau\otimes \pi_0)$ is irreducible. 
\end{corollary}
\begin{proof}
We denote by $$\JH(\Ind_{P}^{\mathrm O(V)}(\tau\otimes \pi_0))$$ 
the set of irreducible constituents of $\Ind_{P}^{\mathrm O(V)}(\tau\otimes \pi_0)$. Consider the set 
\begin{align*}
\bigsqcup_{\pi_0\in \Pi_{\phi_0}} \JH(\Ind_{P}^{\mathrm O(V)}(\tau\otimes \pi_0)) .
\end{align*}
By the Howe duality, Lemma \ref{parameter} and Lemma \ref{multiplicityfree}, this set is a subset of $\Pi_{\phi}$. Hence 
\begin{align*}
|\Pi_{\phi}|\geq \left|\bigsqcup_{\pi_0\in \Pi_{\phi_0}} \JH(\Ind_{P}^{\mathrm O(V)}(\tau\otimes \pi_0)) \right|\geq |\Pi_{\phi_{0}}|. 
\end{align*}
On the other hand, by Proposition \ref{bijection}, we have 
\begin{align*}
|\Pi_{\phi}|=|\widehat{\mathcal {S}_{\phi}}|=|\widehat{\mathcal S_{\phi_0}}|=|\Pi_{\phi_{0}}|. 
\end{align*}
Hence we must have 
\begin{align*}
\left| \JH(\Ind_{P}^{\mathrm O(V)}(\tau\otimes \pi_0)) \right|=1
\end{align*}
for all $\pi_0\in \Pi_{\phi_{0}}$. This implies $\Ind_{P}^{\mathrm O(V)}(\tau\otimes \pi_0)$ is irreducible. 
\end{proof}

\begin{corollary}\label{notanisomorphism}
	If $\phi_\tau\not \subseteq \phi_0$, so the image of $\mathcal S_{\phi_{0}}$ inside $\mathcal {S}_{\phi}$ is an index $2$ subgroup, then $\Ind_{P}^{\mathrm O(V)}(\tau\otimes \pi_0)$ is a direct sum of two non-isomorphic irreducible representations. 
\end{corollary}
\begin{proof}
	It follows from Lemma \ref{multiplicityfree} that $\Ind_{P}^{\mathrm O(V)}(\tau\otimes \pi_0)$ is multiplicity free, so we only need to prove $\Ind_{P}^{\mathrm O(V)}(\tau\otimes \pi_0)$ is reducible of length two.  
	
	We first prove 
	\begin{align}\label{1236}
	\left|\JH(\Ind_{P}^{\mathrm O(V)}(\tau\otimes \pi_0))\right|\geq 2;
	\end{align}
	in other words, $\Ind_{P}^{\mathrm O(V)}(\tau\otimes \pi_0)$ is reducible. Let  
	\begin{align*}
	\phi^+&=(\phi\otimes \chi_{V})+ \chi_{V}, \\
	\phi^+_{0}&=(\phi_{0}\otimes\chi_{V})+ \chi_{V}.
	\end{align*}
	Depending on the relative size of $\mathcal S_{\phi^+_{0}}$ and $\mathcal S_{\phi^+}$, there are two cases: 
	
	$\bullet$ \underline{Case I}: If $\phi_\tau\neq \mathbbm 1$, then $\mathcal S_{\phi^+_{0}}$ is an index $2$ subgroup of $\mathcal S_{\phi^+}$. It follows from the conservation relation (Theorem \ref{con2}) that at least one of  
	$$\theta_{W_0,V_0,\psi}(\pi_0)\quad \mbox{and}\quad \theta_{W_0,V_0,\psi}(\pi_0\otimes\det)$$ 
	is non-zero. Note that we have the isomorphism 
	\begin{align*}
	\Ind_{P}^{\mathrm O(V)}\left(\tau \otimes (\pi_0\otimes\det)\right)\cong \Ind_{P}^{\mathrm O(V)}\left(\tau \otimes \pi_0\right)\otimes\det 
	\end{align*}
	in (\ref{59}), so the reducibility of $\Ind_{P}^{\mathrm O(V)}\left(\tau \otimes (\pi_0\otimes\det)\right)$ is the same as the reducibility of $\Ind_{P}^{\mathrm O(V)}\left(\tau \otimes \pi_0\right)$. Without loss of generality, we may assume that
	$$\theta_{W_0,V_0,\psi}(\pi_0)\neq 0.$$ 
	Put $$\sigma_{0}=\theta_{W_0,V_0,\psi}(\pi_0).$$ Then the $L$-parameter of $\sigma_0$ is $\phi_{0}^+$ by Proposition \ref{prasad}, so is $\sigma_{0}^\vee$ by (\ref{104}). It then follows from Theorem \ref{llcsympletic} (6) that 
	$$
	\Ind_{Q}^{\SP(W)}(\tau\chi_{V}\otimes\sigma_{0}^{\vee})
	$$
	is reducible and has two non-isomorphic irreducible constituents. Put 
	\begin{align*}
	\Ind_{Q}^{\SP(W)}(\tau\chi_{V}\otimes\sigma_{0}^{\vee})\cong \sigma_1^\vee\oplus \sigma_2^\vee. 
	\end{align*}
	Recall that we have constructed an $\SP(W)\times\mathrm O(V)$-equivariant map
	$$\mathcal{T}_{0} : \omega \otimes \Ind_{Q}^{\SP(W)}\left(\tau\chi_{V} \otimes \sigma_{0}^{\vee}\right) \rightarrow \Ind_{P}^{\mathrm O(V)}\left(\tau \otimes \pi_0\right)$$ 
	in Subsection \ref{constructionof}. By the non-vanishing result in Proposition \ref{intertwingmap}, we know that the restriction of $\mathcal T_0$ to $\omega\otimes\sigma_1^\vee$ and $\omega\otimes\sigma_2^\vee$ are both non-zero. Moreover, their images are irreducible and non-isomorphic to each other by the Howe duality. Hence $$\Ind_{P}^{\mathrm{O}\left(V\right)}\left(\tau  \otimes \pi_0\right)$$ is reducible.

	$\bullet$ \underline{Case II}: If $\phi_\tau= \mathbbm 1$, then the natural embedding $\mathcal S_{\phi^+_{0}}\hookrightarrow \mathcal S_{\phi^+}$ is an isomorphism. Our assumptions imply that $\mathbbm 1 \nsubseteq \phi_0$. It then follows from Corollary \ref{1inphi} that both  $$\theta_{W_0,V_0,\psi}(\pi_{0})\quad \mbox{and}\quad \theta_{W_0,V_0,\psi}(\pi_{0}\otimes\det)$$ are non-zero. Put 
	\begin{align*}
	\sigma_0^+=\theta_{W_0,V_0,\psi}(\pi_{0})\quad\mbox{and}\quad  \sigma_0^-=\theta_{W_0,V_0,\psi}(\pi_{0}\otimes\det). 
	\end{align*} 
	Then $\sigma_0^{\pm}$ both have the $L$-parameter $\phi_{0}^+$ by Proposition \ref{prasad}, so do $(\sigma_0^{\pm})^{\vee}$ by (\ref{104}). By Theorem \ref{llcsympletic} (6), both 
	$$\Ind_{Q}^{\SP(W)}\left(\tau\chi_{V}\otimes \left(\sigma_{0}^{+}\right)^{\vee}\right)\quad \mbox{and}\quad \Ind_{Q}^{\SP(W)}\left(\tau\chi_{V}\otimes \left(\sigma_{0}^{-}\right)^{\vee}\right)$$ are irreducible. Put 
	\begin{align*}
	\left(\sigma^{+}\right)^{\vee}= \Ind_{Q}^{\SP(W)}\left(\tau\chi_{V}\otimes \left(\sigma_{0}^{+}\right)^{\vee}\right) \quad \mbox{and}\quad 
	\left(\sigma^{-}\right)^{\vee}= \Ind_{Q}^{\SP(W)}\left(\tau\chi_{V}\otimes \left(\sigma_{0}^{-}\right)^{\vee}\right).
	\end{align*}
  Recall that we have constructed non-zero $\mathrm O(V)\times \SP(W)$-equvariant maps 
	\begin{align*}
	\mathcal{T}_0^{+} : \omega \otimes \Ind_{Q}^{\SP(W)}\left(\tau\chi_{V}\otimes \left(\sigma_{0}^{+}\right)^{\vee}\right)  \rightarrow \Ind_{P}^{\mathrm O\left(V\right)}\left(\tau  \otimes \pi_0\right)
	\end{align*}
	and 
	\begin{align*}
	\mathcal{T}_0^{-} : \omega \otimes \Ind_{Q}^{\SP(W)}\left(\tau\chi_{V}\otimes \left(\sigma_{0}^{-}\right)^{\vee}\right)  \rightarrow \Ind_{P}^{\mathrm O(V)}\left(\tau  \otimes (\pi_0\otimes\det)\right)
	\end{align*}
	in Subsection \ref{constructionof}. Let 
	\begin{align*}
	\pi^{+} \coloneqq \mbox{Im}( \mathcal{T}_0^+)\subseteq \Ind_{P}^{\mathrm O\left(V\right)}\left(\tau\otimes \pi_0\right)
	\end{align*}
	and 
	\begin{align*}
	\pi^{-} \coloneqq \mbox{Im}( \mathcal{T}_0^-)\subseteq \Ind_{P}^{\mathrm O\left(V\right)}\left(\tau\otimes (\pi_0\otimes\det)\right)=\Ind_{P}^{\mathrm O\left(V\right)}\left(\tau\otimes \pi_0\right)\otimes\det.
	\end{align*}
	Then $\pi^{\pm}$ are irreducible by the Howe duality and the $L$-parameter of $\pi^{\pm}$ are $\phi$ by Lemma \ref{parameter}.
	\begin{comment}
	\begin{align*}
	\theta_{V,W,\psi}(\sigma^+)&=\pi^+ \\
	\theta_{V,W,\psi}(\sigma^-)&=\pi^
	\end{align*}
	\end{comment}
	Since $\mathbbm 1\subseteq \phi$, it follows from Proposition \ref{1notinphi} and the conservation relation (Theorem \ref{con2}) that $\pi^+\not \cong \pi^-\otimes \det$. Note that both $\pi^+$ and $\pi^-\otimes\det$ lie in $\Ind_{P}^{\mathrm O\left(V\right)}\left(\tau\otimes \pi_0\right)$. Hence $\Ind_{P}^{\mathrm O\left(V\right)}\left(\tau \otimes \pi_0\right)$ is reducible.
	
	It remains to show that 
	$$\left|\JH(\Ind_{P}^{\mathrm O(V)}(\tau\otimes \pi_0))\right|=2.$$ 
	Again we consider the set 
	\begin{align*}
	\bigsqcup_{\pi_0\in \Pi_{\phi_0}} \JH(\Ind_{P}^{\mathrm O(V)}(\tau\otimes \pi_0)). 
	\end{align*}
	By the Howe duality, Lemma \ref{parameter} and Lemma \ref{multiplicityfree}, this set is a subset of $\Pi_{\phi}$. Hence by (\ref{1236}), we have 
	\begin{align*}
	|\Pi_{\phi}|\geq \left|\bigsqcup_{\pi_0\in \Pi_{\phi_0}} \JH(\Ind_{P}^{\mathrm O(V)}(\tau\otimes \pi_0)) \right|\geq 2|\Pi_{\phi_{0}}|. 
	\end{align*}
	On the other hand, by Proposition \ref{bijection}, we have 
	\begin{align*}
	|\Pi_{\phi}|=|\widehat{\mathcal {S}_{\phi}}|=2|\widehat{S_{\phi_0}}|=2|\Pi_{\phi_{0}}|. 
	\end{align*}
	Hence we must have 
	\begin{align*}
	\left| \JH(\Ind_{P}^{\mathrm O(V)}(\tau\otimes \pi_0)) \right|=2
	\end{align*}
	for all $\pi_0\in \Pi_{\phi_{0}}$. This completes the proof. 
\end{proof}

\subsection{Charater of component group}
For any irreducible representation $\pi\subseteq \Ind_{P}^{\mathrm O(V)}(\tau\otimes \pi_0)$, we have shown that the $L$-parameter of $\pi$ is $\phi$ in the previous subsection. In this subsection, we study the map $\mathcal J_{\mathfrak W_{c^\prime}}$ for a fixed Whittaker datum $\mathfrak W_{c^\prime}$. We divide it into three cases: 
\begin{itemize}
	\item  \underline{Case A}: $\mathbbm 1\not\subseteq \phi$;
	\item \underline{Case B}: $\mathbbm 1\subseteq \phi_{0}$;
	\item \underline{Case C}: $\mathbbm 1\not\subseteq \phi_{0}$ and $\phi_{\tau}= \mathbbm 1$.
\end{itemize}
Put 
\begin{align*}
\phi^+&=(\phi\otimes \chi_{V})+ \chi_{V} ,\\
\phi^+_{0}&=(\phi_{0}\otimes\chi_{V})+ \chi_{V}.
\end{align*}
Then the diagram
\begin{align}\label{101}
\begin{CD}
\mathcal S_{\phi_0} @>>> \mathcal S_{\phi^+_{0}} \\
@VVV @VVV \\
\mathcal {S}_{\phi} @>>> \mathcal S_{\phi^+}
\end{CD}
\end{align}
is commutative. 

\begin{proposition}\label{localintertwing}
Assume that we are in Case A or Case B. 
\begin{enumerate}[(i)]
	\item Put $\mathcal J_{\mathfrak W_{c^\prime}}(\pi_0)=\eta_0$ and $\mathcal J_{\mathfrak W_{c^\prime}}(\pi)=\eta$. Then 
	\begin{align*}
	\eta|_{\mathcal S_{\phi_0}}=\eta_0. 
	\end{align*}
	\item If we further assume that $\phi_{\tau}$ is self-dual and of orthogonal type, then the restriction of the normalized intertwining operator $R_{\mathfrak W_{c^\prime}}(w,\tau\otimes \pi_0)$ to $\pi$ is the scalar multiplication by $\eta(a)$, where $a\in \mathcal {S}_{\phi}$ corresponds to $\phi_{\tau}$. 
\end{enumerate} 
\end{proposition}
\begin{proof}
We first prove statement (i). 

In \underline{Case A}, it follows from Lemma \ref{parameter} that the $L$-parameter for $\pi$ is $\phi$. Since $\mathbbm 1\not \subseteq \phi$, by Corollary \ref{1inphi}, we have $$\theta_{W,V,\psi}(\pi)\neq 0.$$ 
Then Lemma \ref{12} implies that  $$\theta_{W_0,V_0,\psi}(\pi_0)\neq 0.$$ 
Let  
\begin{align*}
\sigma=\theta_{W,V,\psi}(\pi), \quad \sigma_0=\theta_{W_0,V_0,\psi}(\pi_0). 
\end{align*}
It follows from Proposition \ref{prasad} that 
\begin{align*}
\mathcal L(\sigma)=\phi^+\quad \mbox{and}\quad \mathcal L(\sigma_0)=\phi_{0}^+. 
\end{align*}
Put $\mathcal J_{\mathfrak W^\prime_{\psi,c^\prime}}(\sigma)=\eta^+$ and $\mathcal J_{\mathfrak W^\prime_{\psi,c^\prime}}(\sigma_0)=\eta^+_0$. We have  
	\begin{align*}
	\eta|_{\mathcal S_{\phi_0}}&= \left(\eta^+|_{ \mathcal {S}_{\phi}}\right)|_{\mathcal S_{\phi_0}} & \quad \quad \mbox{(by our construction of $\eta$)} \\
	&= \left(\eta^+|_{ \mathcal S_{\phi_0^+}}\right)|_{ \mathcal S_{\phi_0}} & \quad \quad  \mbox{(by the commutative diagram (\ref{101}))}\\
	&=\eta_{0}^+|_{ \mathcal S_{\phi_0}} & \mbox {(by Theorem \ref{llcsympletic} (6))} \\
	&=\eta_0.  & \mbox{(by our construction of $\eta_0$)} 
	\end{align*}
	
In \underline{Case B}, it follows from Lemma \ref{parameter} that the $L$-parameter for $\pi$ is $\phi$. Since $\mathbbm 1 \subseteq \phi$, by Proposition \ref{1notinphi}, exactly one of 
\begin{align*}
\theta_{W,V,\psi}(\pi) \quad \mbox{and}\quad \theta_{W,V,\psi}(\pi\otimes\det)
\end{align*}
is non-zero. If $\theta_{W,V,\psi}(\pi)\neq 0$, then we may apply the same argument as in Case A to $\pi$. So we may assume that $$\theta_{W,V,\psi}(\pi)= 0 \quad \mbox{and}\quad  \theta_{W,V,\psi}(\pi\otimes\det)\neq 0.$$ 
Then by Lemma \ref{12}, we have 
\begin{align*}
\theta_{W_0,V_0,\psi}(\pi_0)= 0 \quad \mbox{and}\quad  \theta_{W_0,V_0,\psi}(\pi_0\otimes\det)\neq 0.
\end{align*}
By a similar argument to Case $A$, we have 
\begin{align}\label{53}
\mathcal J_{\mathfrak W_{c^\prime}}(\pi\otimes\det)|_{\mathcal S_{\phi_0}}= \mathcal J_{\mathfrak W_{c^\prime}}(\pi_0\otimes\det).
\end{align}
On the other hand, since $\mathbbm 1\subseteq \phi_0$, it follows from our construction of $\mathcal J_{\mathfrak W_{c^\prime}}$ that 
\begin{equation}
\begin{aligned}\label{54}
\eta=\mathcal J_{\mathfrak W_{c^\prime}}(\pi)&= \mathcal J_{\mathfrak W_{c^\prime}}(\pi\otimes\det)\otimes\kappa_{\phi},\\
\eta_0=\mathcal J_{\mathfrak W_{c^\prime}}(\pi_0)&= \mathcal J_{\mathfrak W_{c^\prime}}(\pi_0\otimes\det)\otimes\kappa_{\phi_0}.
\end{aligned}
\end{equation}
Hence by (\ref{53}) and (\ref{54}), we have $\eta|_{\mathcal S_{\phi_0}}=\eta_{0}$. Here we use the fact that $\kappa_{\phi}|_{\mathcal S_{\phi_{0}}}=\kappa_{\phi_{0}}$. 

We then prove the statement (ii). It follows from Lemma \ref{multiplicityfree} and Schur's lemma that  $R_{\mathfrak W_{c^\prime}}(\omega,\tau\otimes \pi_0)$ acts on $\pi$ by scalar multiplication. 

In \underline{Case A}, we have 
$$\sigma=\theta_{W,V,\psi}(\pi)\neq 0\quad \mbox{and} \quad \sigma_0=\theta_{W_0,V_0,\psi}(\pi_0)\neq 0.$$ 
Recall that we have constructed the $\mathrm O(V)\times \SP(W)$-equivariant map 
$$\mathcal{T}_{0} : \omega\otimes \operatorname{Ind}_{Q}^{\SP(W)}\left(\tau\chi_{V} \otimes \sigma_{0}^{\vee}\right) \rightarrow \operatorname{Ind}_{P}^{\mathrm{O}\left(V\right)}\left(\tau  \otimes \pi_0\right)$$
in Subsection \ref{constructionof}. By Lemma \ref{12}, we have 
\begin{align*}
\sigma \subseteq \Ind_{Q}^{\SP(W)}(\tau\chi_{V}\otimes \sigma_{0}),
\end{align*}
and then 
$$\sigma^{\vee}\subseteq \Ind_{Q}^{\SP(W)}(\tau\chi_{V}\otimes \sigma_{0}^{\vee}).$$
Here we use the assumption that $\tau$ is self-dual and the fact that $\Ind_{Q}^{\SP(W)}(\tau\chi_{V}\otimes \sigma_{0})$ is semi-simple. It follows from the Howe duality and Proposition \ref{intertwingmap} that the restriction of $\mathcal T_0$ to $\omega\otimes \sigma^{\vee}$ gives an epimorphism 
	$$
	\mathcal T_0 :\omega\otimes \sigma^{\vee} \rightarrow \pi. 
	$$ 
Then by Corollary \ref{123} and the assumption that $\phi_\tau$ is self-dual, we have 

	\begin{equation}\label{81}
	\begin{aligned}
	R_{\mathfrak W_{c^\prime}}(w,\tau\otimes \pi_0)|_\pi&=\omega_{\tau}(-1/c^\prime)\times \chi_{V}(-c^\prime)^k\times  R_{\mathfrak W^\prime_{\psi,1}}(w^{\prime},\tau\chi_{V}\otimes \sigma_0) |_{\sigma^{\vee}}\\
	&=\omega_{\tau}(-c^\prime)\times \chi_{V}(-c^\prime)^k\times  R_{\mathfrak W^\prime_{\psi,1}}(w^{\prime},\tau\chi_{V}\otimes \sigma_0) |_{\sigma^{\vee}}.
	\end{aligned}
	\end{equation}
 It follows from Theorem \ref{llcsympletic} (6) and (\ref{104}) that 
 \begin{equation}\label{82}
 \begin{aligned}
 	R_{\mathfrak W^\prime_{\psi,1}}(w^{\prime},\tau\chi_{V}\otimes \sigma_0) |_{\sigma^{\vee}}=&\mathcal J_{\mathfrak W^\prime_{\psi,1}}(\sigma^\vee)(a^\prime)\\
 	=&\mathcal J_{\mathfrak W^\prime_{\psi,1}}(\sigma)(a^\prime)\times \eta_{\phi^+,-1}(a^\prime)\\
 =&\mathcal J_{\mathfrak W^\prime_{\psi,1}}(\sigma)(a^\prime)\times \omega_{\tau}(-1)\times \chi_{V}(-1)^k,
 \end{aligned}
 \end{equation}
	where $a^{\prime}\in \mathcal S_{\phi^+}$ corresponds to $\phi_{\tau}\otimes \chi_{V}$. On the other hand, by Theorem \ref{llcsympletic} (4), we have 
	\begin{align}\label{84}
	\mathcal J_{\mathfrak W^\prime_{\psi,c^\prime}}(\sigma)(a^{\prime})=\mathcal J_{\mathfrak W^\prime_{\psi,1}}(\sigma)(a^{\prime})\times \omega_{\tau}(c^\prime)\times \chi_{V}(c^\prime)^{ k}. 
	\end{align}
	Recall that by our construction of $\mathcal J_{\mathfrak W_{c^\prime}}$, we also have 
	\begin{align}\label{83}
	\mathcal J_{\mathfrak W_{c^\prime}}(\pi)(a)=\mathcal J_{\mathfrak W^\prime_{c^\prime}}(\sigma)(a^{\prime}).
	\end{align}
	Combining the equalities (\ref{81}), (\ref{82}), (\ref{84}) and (\ref{83}), we deduce  
	\begin{comment}
	\begin{align*}
	R_{c^{\prime}}(\omega,\tau\otimes \pi_0)|\pi=\omega_{\tau}(1/c^{\prime})\cdot \chi_{V}(c^{\prime})^k\cdot \eta(\sigma)(a^{\prime})\\
	=\eta_{c^{\prime}}(\sigma)(a^{\prime})= \eta_{c^{\prime}}(\pi)(a)
	\end{align*}
	\end{comment}
	\begin{align*}
	R_{\mathfrak W_{c^\prime}}(\omega,\tau\otimes \pi_0)|_\pi= \mathcal J_{\mathfrak W_{c^\prime}}(\pi)(a). 
	\end{align*}
	
	In \underline{Case B}, if $\theta_{W,V,\psi}(\pi)\neq 0$, then we may apply the same argument as in Case A to $\pi$. So we may assume that $$\theta_{W,V,\psi}(\pi)=0\quad \mbox{and}\quad \theta_{W,V,\psi}(\pi\otimes\det)\neq 0.$$ 
	By a similar argument to Case $A$, we have 
	\begin{align}\label{60}
	R_{\mathfrak W_{c^\prime}}(w,\tau\otimes (\pi_0\otimes\det))|_{\pi\otimes\det}= \mathcal J_{\mathfrak W_{c^\prime}}(\pi\otimes
	\det)(a).
	\end{align}
  Since $\mathbbm 1\subseteq\phi$, it follows from our construction for $\mathcal J_{\mathfrak W_{c^\prime}}$ in Section \ref{constructJc} that 
	\begin{equation}\label{555}
	\begin{aligned}
	\eta=\mathcal J_{\mathfrak W_{c^\prime}}(\pi)&= \mathcal J_{\mathfrak W_{c^\prime}}(\pi\otimes\det)\otimes\kappa_{\phi}.
	\end{aligned}
	\end{equation}
	On the other hand, by (\ref{comparenormalize}), we have 
	\begin{align}\label{61}
	R_{\mathfrak W_{c^\prime}}(w,\tau\otimes (\pi_0\otimes\det))|_{\pi\otimes\det}=(-1)^{\dim\tau}\cdot R_{\mathfrak W_{c^\prime}}(w,\tau\otimes \pi_0)|_\pi .
	\end{align}
	So by (\ref{60}), (\ref{555}) and (\ref{61}), we deduce 
	\begin{align*}
   \eta(a)=&\mathcal J_{\mathfrak W_{c^\prime}}(\pi)(a)\\
   =&\mathcal J_{\mathfrak W_{c^\prime}}(\pi\otimes
   \det)(a)\times \kappa_{\phi}(a)\\ 
   =& R_{\mathfrak W_{c^\prime}}(w,\tau\otimes (\pi_0\otimes\det))|_{\pi\otimes\det}\times (-1)^{\dim \tau}\\
   =&R_{\mathfrak W_{c^\prime}}(w,\tau\otimes \pi_0)|_\pi. 
   \end{align*}	
   This finishes the proof.
\end{proof}
\begin{remark}
	To apply the similar argument to Case C, we need the formula 
	\begin{align*}
	\mathcal J_{\mathfrak W_{c^\prime}}(\pi_0\otimes\det)=\mathcal J_{\mathfrak W_{c^\prime}}(\pi_0)\otimes \kappa_{\phi_{0}}\quad \mbox{for}\,\, \pi_{0}\in \Pi_{\phi_0}
	\end{align*}
	in the case when $\mathbbm 1\not\subseteq \phi_{0}$. But this does not directly follow from our construction. We will prove this formula in the next proposition.
\end{remark}

Let $\pi\in \Pi_{\phi}$. Then $\pi\otimes\det \in \Pi_{\phi}$ by Proposition \ref{det}. We compare $\mathcal J_{\mathfrak W_{c^\prime}}(\pi)$ with $\mathcal J_{\mathfrak W_{c^\prime}}(\pi\otimes\det)$. 
\begin{proposition}\label{determint}
Let  $\phi\in \Para(\mathrm O(V_{2n}))$ and $\pi\in \Pi_{\phi}$. Then we have 
\begin{align*}
\mathcal J_{\mathfrak W_{c^\prime}}(\pi\otimes\det)=\mathcal J_{\mathfrak W_{c^\prime}}(\pi)\otimes\kappa_{\phi},
\end{align*}
where $\kappa_{\phi}$ is defined in (\ref{kappa}). 
\end{proposition}
\begin{proof}
 The proof follows an idea in \cite[\S 7]{MR3788848}. If $\mathbbm 1\subseteq \phi$, then this follows from our construction of $\mathcal J_{\mathfrak W_{c^\prime}}$ in Subsection \ref{constructJc}. So we assume that $\mathbbm 1\not\subseteq \phi$, and write 
 \begin{align*}
 \phi=m_1\phi_1+ \cdots + m_l \phi_l + \varphi+ \varphi^\vee,
 \end{align*}
 where $\phi_i$ are pairwise distinct irreducible $k_i$-dimensional orthogonal representations of $\WD_{F}$ and $\varphi$ is a sum of irreducible tempered representations of $\WD_F$ which are not orthogonal. Then  
 \begin{align*}
 \mathcal {S}_{\phi}=\bigoplus_{i=1}^l (\mathbb Z/2\mathbb Z)a_i,
 \end{align*} 
 where $a_i$ corresponds to $\phi_i$. For any $i\in \{1,2\cdots l\}$, let $\tau_{i}$ be the irreducible unitary representation of $\GL_{k_{i}}(F)$ corresponding to  $\phi_{i}$. We consider the induced representations $$\Ind_{P^\prime}^{\mathrm O(V^\prime)}(\tau_i\otimes\pi)\quad \mbox{and}\quad \Ind_{P^\prime}^{\mathrm O(V^\prime)}(\tau_i\otimes(\pi\otimes \det)),$$
 where $V^\prime= V\oplus \mathbb H^k$  and $P^\prime$ is a parabolic subgroup of $\mathrm O(V^\prime)$ with Levi component $M_{P^\prime}\cong \mathrm O(V)\times \GL_{k_{i}}(F)$. Put
	\begin{align*}
	\phi^\prime =\phi_i + \phi + \phi_i^{\vee}.
	\end{align*}
	Since $\phi_i\subseteq \phi$, the inclusion $\mathcal {S}_{\phi}\hookrightarrow \mathcal S_{\phi^\prime}$ is an isomorphism. It then follows from Proposition \ref{det}, Lemma \ref{parameter} and Corollary \ref{embeddingisisomorphism} that both $\Ind_{P^\prime}^{\mathrm O(V^\prime)}(\tau_i\otimes\pi)$ and  $\Ind_{P^\prime}^{\mathrm O(V^\prime)}(\tau_i\otimes(\pi\otimes\det))$ are irreducible and have the $L$-parameter $\phi^\prime$. Write $$\pi^\prime = \Ind_{P^\prime}^{\mathrm O(V^\prime)}(\tau_i\otimes\pi).$$
	By (\ref{59}), we have 
  $$\pi^\prime\otimes\det \cong \Ind_{P^\prime}^{\mathrm O(V^\prime)}(\tau_i\otimes(\pi\otimes\det)).$$ If we identify $\mathcal S_{\phi^\prime}$ with $\mathcal {S}_{\phi}$ through the natural isomorphism, then it follows from Proposition \ref{localintertwing} (Case A) that 
\begin{equation}\label{57}
\begin{aligned}
\mathcal J_{\mathfrak W_{c^\prime}}(\pi)(a_i)=\mathcal J_{\mathfrak W_{c^\prime}}(\pi^\prime)(a_i)=&R_{\mathfrak W_{c^\prime}}(w,\tau_{i}\otimes\pi)|_{\pi^\prime},\\
\mathcal J_{\mathfrak W_{c^\prime}}(\pi\otimes\det)(a_i)=\mathcal J_{\mathfrak W_{c^\prime}}(\pi^\prime\otimes\det)(a_i)=&R_{\mathfrak W_{c^\prime}}(w,\tau_{i}\otimes(\pi\otimes\det))|_{\pi^\prime\otimes\det}.
\end{aligned}
\end{equation}
On the other hand, by (\ref{comparenormalize}), we know that 
	\begin{align}\label{58}
	 R_{\mathfrak W_{c^\prime}}(w,\tau_{i}\otimes(\pi\otimes\det))|_{\pi^\prime\otimes\det}=(-1)^{k_i} R_{\mathfrak W_{c^\prime}}(w,\tau_{i}\otimes\pi)|_{\pi^\prime}.
	\end{align}
Hence by (\ref{kappa}), (\ref{57}) and (\ref{58}), we have 
	\begin{align*}
	\mathcal J_{\mathfrak W_{c^\prime}}(\pi)(a_i)= \mathcal J_{\mathfrak W_{c^\prime}}(\pi\otimes\det)(a_i)\times (-1)^{k_i}=  \mathcal J_{\mathfrak W_{c^\prime}}(\pi\otimes\det)(a_i)\times \kappa_{\phi}(a_i). 
	\end{align*}
This finishes the proof. 	
\end{proof}

Now we can prove Proposition \ref{localintertwing} for Case C.
\begin{corollary}\label{CaseC}
		Assume that we are in Case C, i.e., $\mathbbm 1\not\subseteq \phi_0$ and $\phi_\tau=\mathbbm 1$. 
	\begin{enumerate}[(i)]
		\item Put $\mathcal J_{\mathfrak W_{c^\prime}}(\pi_0)=\eta_0$ and $\mathcal J_{\mathfrak W_{c^\prime}}(\pi)=\eta$. Then 
		\begin{align*}
		\eta|_{\mathcal S_{\phi_0}}=\eta_0.
		\end{align*}
		\item If we further assume $\phi_{\tau}$ is self-dual and of orthogonal type, then the restriction of the normalized intertwining operator $R_{\mathfrak W_{c^\prime}}(\omega,\tau\otimes \pi_0)$ to $\pi$ is the scalar multiplication by $\eta(a)$, where $a\in \mathcal {S}_{\phi}$ corrresponds to $\phi_{\tau}$. 
	\end{enumerate} 
\end{corollary}
\begin{proof}
	It follows by Proposition \ref{determint} that 
	$$\mathcal J_{\mathfrak W_{c^\prime}}(\pi_0\otimes\det)=\mathcal J_{\mathfrak W_{c^\prime}}(\pi_0)\otimes \kappa_{\phi_{0}}.$$
	So now we can apply the same argument in Proposition \ref{localintertwing} for Case B to this case.
\end{proof}

Combining Lemma \ref{parameter}, Corollary \ref{embeddingisisomorphism}, Corollary \ref{notanisomorphism}, Proposition \ref{localintertwing} and Corollary \ref{CaseC}, we deduce the local intertwing relation  for even orthogonal groups: 
\begin{corollary}\label{LIRfinal}
	The maps $\mathcal L$ and $\mathcal J_{\mathfrak W_{c^\prime}}$ we constructed satisfy the local intertwining relation in Theorem \ref{desideratumall} (9).  
\end{corollary}
\begin{comment}
\begin{proposition}\label{induceddecomposition}
There is a decomposition of the induced representation 
$$
\Ind_{P}^{\mathrm O(V)}(\tau\otimes \pi_0)=\bigoplus_{\eta}\pi_{c^{\prime}}(\phi,\eta)
$$
where the sum runs for all $\eta\in \widehat{\mathcal {S}_{\phi}^{+}}$ such that $\eta|_{S_{\phi_{0}}^{+}}=\eta_0$
\end{proposition}
This Lemma together with Lemma \ref{localintertwing} gives us the local intertwing relation for Even orthogonal groups, which is  Desideratum \ref{desideratumorthgonal}(7). 
\end{comment}

\section{Comparison with local Langlands corresponence \`a la Arthur}
%Fix a $d\in F^\times /F^{\times 2}$ and let $E=F(\sqrt d)$. 
%Let $V^+_{2n}=V_{2n}$ be the orthogonal space associated to $(d,1)$ and 
Fix $(d,c)\in (F^\times) ^2$ and let $E=F(\sqrt d)$. Let $V_{2n}=V_{2n}^+$ be the $2n$-dimensional orthogonal space associated to $(d,c)$. For any $\pi\in \Irr(\mathrm O(V_{2n}^+))$ and $c^\prime \in c N_{E/F}(E^\times)$, we associated a pair
\begin{align*}
\left(\phi= \mathcal L(\pi), \quad \eta=\mathcal J_{\mathfrak W_{c^\prime}}(\pi)\right)
\end{align*}
to $\pi$ in Section \ref{construction}. In Theorem \ref{Arthurorth}, Arthur and Atobe--Gan also associated a pair 
\begin{align*}
\left(\phi^A= \mathcal L^A(\pi), \quad \eta^A=\mathcal J^A_{\mathfrak W_c^\prime}(\pi)\right)
\end{align*}
to $\pi$. The following theorem shows that our classification coincides with Arthur's in the quasi-split case. 
%The following theorem says that our parameterization maps $\mathcal L$ and $\mathcal J_{\mathfrak W_{c^\prime}}$ equal to Arthur's parameterization maps $\mathcal L^A$ and $\mathcal J^A_{\mathfrak W_{c^\prime}}$ on $\Irr(\mathrm O(V_{2n}^+))$.
\begin{theorem}\label{comparearthur}
	We have 
	\begin{align*}
	\phi=\phi^A \quad \mbox{and}\quad \eta=\eta^A. 
	\end{align*}
\end{theorem}
\begin{proof}
This theorem is a corollary of \cite[Theorem 4.4]{MR3708200}. For the convenience of the reader, we inculde the proof here. By Proposition \ref{Langlandsclasification} and Theorem \ref{Arthurorth}, both two LLC are compatible with Langlands quotients. Without loss of generality, we may assume that $\pi$ is tempered. It follows by Lemma \ref{respectplancherel} and Theorem \ref{Arthurorth} that 
\begin{align*}
\mu_{\psi}(\tau_s\otimes \pi)=&\gamma(s,\phi_{\tau}\otimes \phi^{ \vee},\psi)\times \gamma(-s,\phi_{\tau}^{\vee}\otimes\phi, \psi_{-1} )\\
  \times &\gamma (2s, \wedge^2\circ \phi_{\tau}, \psi)\times \gamma(-2s, \wedge^{2} \circ \phi_{\tau}^{\vee},\psi_{-1})
\end{align*}
and 
\begin{align*}
\mu_{\psi}(\tau_s\otimes \pi)&=\gamma(s,\phi_{\tau}\otimes (\phi^{ A})^\vee,\psi)\times \gamma(-s,\phi_{\tau}^{\vee}\otimes\phi^A, \psi_{-1} )\\
& \times \gamma (2s, \wedge^2\circ \phi_{\tau}, \psi)\times \gamma(-2s, \wedge^{2} \circ \phi_{\tau}^{\vee},\psi_{-1})
\end{align*}
for any $\tau\in \Irr (\GL_k(F))$, where $\phi_{\tau}$ is the $L$-parameter of $\tau$. Hence
\begin{align*}
\gamma(s,\phi_{\tau}\otimes \phi^{ \vee},\psi)\times \gamma(-s,\phi_{\tau}^{\vee}\otimes\phi, \psi_{-1} )=\gamma(s,\phi_{\tau}\otimes (\phi^{ A})^\vee,\psi)\times \gamma(-s,\phi_{\tau}^{\vee}\otimes\phi^A, \psi_{-1} ).
\end{align*}
Then by Lemma \ref{gammadetermine}, we deduce that $\phi=\phi^A$. 

Next we prove $\eta=\eta^A$. Write 
	\begin{align*}
	\phi=m_1\phi_1+ \cdots + m_l \phi_l+ \varphi+ \varphi^\vee,
	\end{align*}
	where $\phi_i$ are pairwise distinct irreducible orthogonal representations of $\WD_{F}$ and $\varphi$ is a sum of irreducible tempered representations of $\WD_F$ which are not orthogonal. Then 
	\begin{align*}
	\mathcal {S}_{\phi}=\bigoplus_{i=1}^l (\mathbb Z/2\mathbb Z)a_i,
	\end{align*} 
	where $a_i$ corresponds to $\phi_i$. 
	
For any $i\in \{1,2,\cdots, l\}$, let $\tau_{i}$ be the irreducible unitary representation of $\GL_{k_{i}}(F)$ corresponding to $\phi_{i}$. We consider the induced representation $\widetilde{\pi}_i=\Ind_{P^\prime}^{\mathrm O(V^\prime)}(\tau_i\otimes\pi)$, where $V^\prime= V\oplus \mathbb H^k$ and $P^\prime$ is a parabolic subgroup of $\mathrm O(V^\prime)$ with Levi component $$M_{P^\prime}\cong \mathrm O(V)\times \GL_{k_{i}}(F).$$
It follows from Corollary \ref{LIRfinal} and Theorem \ref{Arthurorth} that $\widetilde{\pi}_i$ is irreducible, with $L$-parameter 
\begin{align*}
\phi^\prime =\phi_i + \phi + \phi_i^{\vee}
\end{align*}
and corresponds to 
\begin{align*}
\begin{cases*}
\eta \quad &\mbox{under $\mathcal J_{\mathfrak W_{c^\prime}}$},\\
\eta^A \quad &\mbox{under $\mathcal J^A_{\mathfrak W_{c^\prime}}$}. 
\end{cases*}
\end{align*}
Here we identify $\mathcal S_{\phi^\prime}$ with $\mathcal {S}_{\phi}$ through the natural isomorphism $\mathcal S_{\phi^\prime}\cong \mathcal {S}_{\phi}$. Let $R_{\mathfrak W_{c^\prime}}(w,\tau_{i}\otimes\pi)$ be the normalized intertwining operator defined in Subsection \ref{normalizingintertwing}. Apply Corollary \ref{LIRfinal} and Theorem \ref{Arthurorth} again, and we have 
\begin{equation}\label{102}
\eta(a_i)=R_{\mathfrak W_{c^\prime}}(w,\tau_{i}\otimes\pi)|_{\widetilde{\pi}_i}=\eta^A(a_i). 
\end{equation}
Hence 	
\begin{align*}
\eta=\eta^A.
\end{align*}
This completes the proof. 
\end{proof}

We then prove Theorem \ref{desideratumall} (6) holds for our construction: 
\begin{corollary}\label{plus}
Let $\phi\in \Phi(\mathrm O(V_{2n}))$ and $\pi\in \Pi_{\phi}$. For each $c^\prime \in F^\times $ and Whittaker datum $\mathfrak W_{c^\prime}$, $\pi\in \Pi_\phi(\mathrm O(V_{2n}^+))$ if and only if 
	\begin{align*}
	\mathcal J_{\mathfrak W_{c^\prime}}(\pi)(z_{\phi})=\chi_{V}(c^\prime/c).
	\end{align*}
\end{corollary}
\begin{proof}
	We divide the proof into two cases: 
	
	$\bullet$ \underline{Case I}: If $c^\prime \in c N_{E/F}(E^\times)$, then $\chi_{V}(c^\prime/c)=1$ and this case follows from Theorem \ref{Arthurorth} and Theorem \ref{comparearthur}.

	$\bullet$ \underline{Case II}: If $c^\prime \notin c N_{E/F}(E^\times)$, then this case follows from Case I and Proposition \ref{changeofwhittakerdatum}. 
\end{proof}

\section{Completion of the proof}
So far, we have proved that our construction of LLC satisfies Theorem \ref{desideratumall} (1), (2), (4), (5), (6), (9), (10), (11), (12). In this section, we shall prove our construction of LLC satisfies (3), (7) and (8) in Theorem \ref{desideratumall}, and hence complete the proof of Theorem \ref{desideratumall}. As in the previous sections, we fix $(d,c)\in (F^\times) ^2$. Let $V_{2n}=V_{2n}^+$ be the $2n$-dimensional orthogonal space associated to $(d,c)$. 

We first prove that our construction satisfies Theorem \ref{desideratumall} (8).
\begin{proposition}\label{dettwistfinal}
For any $\phi\in \Phi(\mathrm O(V_{2n}))$ and $\pi\in \Pi_{\phi}$, the determinant twist $\pi\otimes \det$ also belongs to $\Pi_{\phi}$, and 
$$\mathcal J_{\mathfrak W_{c^\prime}}(\pi \otimes \det)=\mathcal J_{\mathfrak W_{c^\prime}}(\pi)\otimes  \kappa_{\phi}.$$
\end{proposition}
\begin{proof}
	By our construction in Subsection \ref{nontempered}, it is enough to prove this when $\phi\in \Para(\mathrm O(V_{2n}))$. This follows from Proposition \ref{det} and Proposition \ref{determint}. 
\end{proof}

Before proving Theorem \ref{desideratumall} (3), we do some preparations. For any $c^\prime \in F^\times$, let $V_{2n}^\prime $ be the $2n$-dimensional orthogonal space associated to $(d,c^\prime)$ and $W_{2n}$ be the $2n$-dimensional symplectic space. 
%Let $\mathrm O(V_{2n})$ be a quasi-split even orthogonal group and $\SP(W_{2n})$ be a symplectic group. For any $c\in F^\times$, 
Let $U^\prime, \widetilde{U}$ and $\mu^\prime_{c^\prime}, \mu_{c^\prime}^+$ be those defined in Subsection \ref{whittaker}. The following lemma computes the Whittaker model of the Weil representation $\omega=\omega_{V^\prime_{2n},W_{2n},\psi}$.  
\begin{lemma}\label{Jacquetmodule}
	 We have 
	\begin{align*}
	\omega_{(U^\prime,\mu^\prime_{c^\prime})}\cong \ind_{\widetilde{U}}^{\mathrm O(V^\prime_{2n})}\mu_{c^\prime}^+. 
	\end{align*}
\end{lemma}
\begin{proof}
	See \cite{MR1454260} and \cite{MR1738175} for an analogous computation. We omit the details. 
\end{proof}
Let $\pi\in \Irrt(\mathrm O(V^\prime_{2n}))$. We calculate $$\Hom_{\mathrm O(V^\prime_{2n})\times U^\prime}(\omega, \pi\boxtimes \mu_{c^\prime}^\prime)$$ in two different ways. On one hand, we have 
\begin{align*}
\Hom_{\mathrm O(V^\prime_{2n})\times U^\prime}(\omega, \pi\boxtimes \mu_{c^\prime}^\prime)&\cong \Hom_{U^\prime}(\Theta_{W_{2n},V^\prime_{2n}, \psi}(\pi), \mu_{{c^\prime}}^\prime)\\
&\cong  \Hom_{U^\prime}(\theta_{W_{2n},V^\prime_{2n}, \psi}(\pi), \mu_{{c^\prime}}^\prime),
\end{align*}
where the last equality follows from Lemma \ref{temperedtotempered}. On the other hand, by Lemma \ref{Jacquetmodule}, we have 
\begin{align*}
\Hom_{\mathrm O(V^\prime_{2n})\times U^\prime}(\omega, \pi\boxtimes \mu_{c^\prime}^\prime)&\cong \Hom_{\mathrm O(V^\prime_{2n})}\left(\ind_{\widetilde{U}}^{\mathrm O(V^\prime_{2n})}\mu_{c^\prime}^+, \pi\right)\\
&\cong \Hom_{\mathrm O(V^\prime_{2n})}\left(\pi^\vee, \Ind_{\widetilde{U}}^{\mathrm O(V^\prime_{2n})}(\mu_{c^\prime}^+)^\vee\right)\\
&\cong \Hom_{\widetilde{U}}\left(\pi^\vee, (\mu_{c^\prime}^+)^\vee\right)\\
&\cong \Hom_{\widetilde{U}}\left(\pi, \mu_{c^\prime}^+\right).
\end{align*}
Here the last equality follows from the fact that both $\pi$ and $\mu_{c^\prime}^+$ are unitary. So we deduce the following proposition.  
\begin{proposition}\label{generic to generic}
	For any ${c^\prime}\in F^\times$, $\pi\in \Irrt(\mathrm O(V^\prime_{2n}))$ is $\mathfrak W_{c^\prime}^+$-generic if and only if $\sigma=\theta_{W_{2n},V^\prime_{2n},\psi}(\pi)$ is $\mathfrak W_{\psi,{c^\prime}}$-generic.   	
\end{proposition}

We then prove Theorem \ref{desideratumall} (3): 
\begin{proposition}\label{generic}
	Assume that $\phi\in \Para(\mathrm O(V_{2n}))$. Then for each Whittaker datum $\mathfrak W_{c^\prime}$,
	\begin{itemize}
		\item there is a unique $\mathfrak W^+_{{c^\prime}}$-generic representation $\pi$ in $\Pi_{\phi}$ and $\mathcal J_{\mathfrak W_{c^\prime}}(\pi)$ is the trivial representation of $\mathcal {S}_{\phi}$; 
		\item there is a unique $\mathfrak W^-_{{c^\prime}}$-generic representation $\pi$ in $\Pi_{\phi}$ and $\mathcal J_{\mathfrak W_{c^\prime}}(\pi)=\kappa_{\phi}$.  
	\end{itemize}
\end{proposition}
\begin{proof}
	The first statement follows from Theorem \ref{llcsympletic} (3),  Proposition \ref{generic to generic} and our construction of $\mathcal J_{\mathfrak W_{c^\prime}}$. Note that $\pi$ is $\mathfrak W^+_{{c^\prime}}$-generic if and only if $\pi\otimes\det$ is $\mathfrak W^{-}_{{c^\prime}}$-generic, so the second statement follows from the first statement and Proposition \ref{determint}.    
\end{proof}

Finally we prove Theorem \ref{desideratumall} (7): 
\begin{proposition}
Under the LLC we constructed, the following are equivalent: 
	\begin{itemize}
		\item 	$\phi \in \Phi^{\epsilon}\left(\mathrm{O}\left(V_{2 n}\right)\right)$;
		\item some $\pi\in \Pi_{\phi}$ satisfies $\pi \otimes \det \neq \pi$;
		\item 	
		all $\pi\in \Pi_{\phi}$ satisfy $\pi \otimes\det 
		\neq \pi$. 
	\end{itemize} 
\end{proposition}
\begin{proof}
	Note that for $\phi\in \Phi(\mathrm O(V_{2n}))$, we have 
	\begin{align*}
	\kappa_{\phi}\neq  1 \quad \mbox{if and only if $\phi\in \Phi^{\epsilon}(\mathrm O(V_{2n}))$}.
	\end{align*}
	Then this proposition follows from Theorem \ref{desideratumall} (2) and (8), which we have proved in Corollary \ref{bijection2} and Proposition \ref{dettwistfinal}. 
\end{proof}

So we have completed the proof of Theorem \ref{desideratumall}.

\appendix
\section{Local Langlands correspondence for special even orthogonal groups}
In \cite{MR3135650}, Arthur established a weaker version LLC for quasi-split special even orthogonal groups from the LLC for quasi-split even orthogonal groups. This was explicated by Atobe--Gan \cite{MR3708200}. Since we now construct the LLC for even orthogonal groups, following Arthur's idea, we can deduce a weaker version LLC for special even orthogonal groups. We shall do it in this appendix. 

Let $V=V_{2n}$ be a $2n$-dimensional orthogonal space and $\chi_V$ be the discriminant character of $V$. By \cite[\S 8]{MR3202556} and \cite[\S 3]{MR3708200}, we define
\begin{align*}
\Phi(\SO(V_{2n}))=\{\phi: \WD_{F} \rightarrow \mathrm{O}(2n, \mathbb{C}) | \det(\phi)=\chi_{V}\} /(\SO(2n, \mathbb{C})\mbox{-conjugacy}).
\end{align*} 
and call an element $\phi\in \Phi(\SO(V_{2n}))$ an $L$-parameter for $\SO(V_{2n})$. Note that $\Phi(\SO(V_{2n}))$ is different from $\Phi(\mathrm O(V_{2n}))$ since we consider the $\SO(2n,\mathbb C)$-conjugacy rather than $\mathrm O(2n,\mathbb C)$-conjugacy here. There is a natural surjective map 
\begin{align}\label{523}
\Phi(\SO(V_{2n}))\twoheadrightarrow \Phi(\mathrm O(V_{2n})).
\end{align} 
We define $\Phi^\epsilon(\SO(V_{2n}))$ to be the preimage of $\Phi^\epsilon(\mathrm O(V_{2n}))$. It is easy to check that the map (\ref{523}) is bijective on the subset $\Phi^\epsilon(\SO(V_{2n}))$ and is a two-to-one map on $\Phi(\SO(V_{2n}))\setminus \Phi^\epsilon(\SO(V_{2n}))$. 

%As in Subsection \ref{whittaker}, for every $c\in F^\times/F^{\times 2}$, we have a Whittaker datum $\mathfrak W_{c^\prime}$.
 Next we state the local Langlands correspondence for special even orthogonal groups. The reader can consult \cite[\S 3.3]{MR3708200} for a detailed description. 
\begin{desideratum}\label{desideratumallspecial1}
	Fix $(d,c)\in (F^\times)^2$. Let $V_{2n}=V_{2n}^+$ be the orthogonal space associated to $(d,c)$, and $\chi_V=(\cdot, d)_F$ be the discriminant character of $V_{2n}$. 	
%	Fix $d\in F^\times /F^{\times 2}$. Let $\chi=(\cdot, d)_F$ be the character of $F^\times$ associated to $E=F(\sqrt d)$. Let $V=V_{2n}$ be a $2n$-dimensional orthogonal space with discriminant character $\chi_V=\chi$.
%Let $V=V_{2n}$ be an orthogonal space and $\chi_V$ be the discriminant character of $V$.
	\begin{enumerate}[(1).]
		\item There exists a surjective map
		$$
		\mathcal L: \bigsqcup_{\delta\in \{\pm 1\}} \Irr \left(\SO(V_{2n}^\delta)\right) \longrightarrow \Phi(\SO(V_{2n}))$$
		which is finite-to-one. For any $\phi \in \Phi(\SO(V_{2n}))$, we denote $\mathcal L^{-1}(\phi)$ by $\Pi_{\phi}$ and call it the $L$-packet of $\phi$. We also write $\Pi_{\phi}(\SO(V_{2n}))=\Pi_{\phi}\cap \Irr\left(\SO(V_{2n})\right)$.
		\item For each Whittaker datum $\mathfrak W_{c^\prime}$, there exists a canonical bijection
		\begin{align}\label{522}
		\mathcal J_{\mathfrak W_{c^\prime}} : \Pi_{\phi} \longrightarrow \widehat {\mathcal S^+_{\phi}}.
		\end{align}
		\item The $L$-packet $\Pi_{\phi}$ and the bijection $\mathcal J_{\mathfrak W_{c^\prime}}$ satisfy the analogues of Theorem \ref{desideratumall} (3)-(12). 
	\end{enumerate}
\end{desideratum}
Desideratum \ref{desideratumallspecial1} has not been established. However, following what Arthur did in \cite{MR3135650}, we can deduce a weaker version of Desideratum \ref{desideratumallspecial1} as follows.

We introduce an equivalence relation $\sim_{\epsilon}$ on $\Irr(\SO (V_{2n}^\delta))$. Choose an element $\epsilon\in \mathrm O(V_{2n}^\delta)$ such that $\det(\epsilon)=-1$. For $\pi_0\in \Irr(\SO (V_{2n}^\delta))$, we define its conjugate $\pi_0^\epsilon$ by $$\pi_0^\epsilon(h)=\pi_0(\epsilon^{-1}h\epsilon) \quad \mbox{for $h\in \mathrm O(V_{2n}^\delta)$}.$$ 
Then the equivalence relation $\sim_{\epsilon}$ on $\Irr(\SO (V_{2n}^\delta))$ is defined by 
\begin{align*}
\pi_0 \sim_{\epsilon} \pi_0^\epsilon.
\end{align*}
We denote by $[\pi_0]$ the image of a representation $\pi_0\in \Irr(\SO (V_{2n}^\delta))$ under the canonical map $\Irr(\SO (V_{2n}^\delta))\rightarrow \Irr(\SO (V_{2n}^\delta))/\sim_{\epsilon}$. We say that $[\pi_0]\in \Irr(\SO (V_{2n}^\delta))/\sim_{\epsilon}$ is tempered (resp. discrete) if some (and hence any) representative $\pi_0$ is tempered (resp. discrete). We also define an equivalence relation $\sim_{\det}$ on $\Irr(\mathrm  O(V_{2n}^\delta))$ by 
\begin{align*}
\pi\sim_{\det} \pi\otimes\det \quad \mbox{for $\pi\in \Irr(\mathrm O(V_{2n}^\delta))$}.
\end{align*}
Then the restriction and the induction gives a canonical bijection 
\begin{align*}
\Irr(\mathrm O(V_{2n}^\delta))/\sim_{\det} \longleftrightarrow \Irr(\SO  (V_{2n}^\delta))/\sim_{\epsilon}. 
\end{align*}
\begin{comment}
On the other hand, we introduce an equivalence relation $\sim_{\epsilon}$ on $\Phi(\SO(V_{2n}))$. For $\phi,\phi^\prime\in \Phi(\SO(V_{2n}))$, we write $\phi\sim_{\epsilon}\phi^\prime$ if $\phi$ is $\mathrm O(2n,\mathbb C)$-conjugate to $\phi^\prime$. The equivalence class of $\phi$ is also denoted by $\phi$.
\end{comment}
We state the weaker version of LLC for $\SO(V_{2n})$ as follows: 
\begin{theorem}[Weak LLC for special even orthogonal groups]\label{desideratumallspecial}
	Fix $(d,c)\in (F^\times)^2$. Let $V_{2n}=V_{2n}^+$ be the orthogonal space associated to $(d,c)$, and $\chi_V=(\cdot, d)_F$ be the discriminant character of $V_{2n}$. 
%	Fix $d\in F^\times /F^{\times 2}$. Let $\chi=(\cdot, d)_F$ be the character of $F^\times$ associated to $E=F(\sqrt d)$. Let $V=V_{2n}$ be a $2n$-dimensional orthogonal space with discriminant character $\chi_V=\chi$.
	\begin{enumerate}[(1).]
		\item There exists a surjective map
		$$
		\mathcal L: \bigsqcup_{\delta\in \{\pm 1\}} \left(\Irr \left(\SO(V_{2n}^\delta)\right)/\sim_{\epsilon}\right) \longrightarrow \Phi(\mathrm O(V_{2n})),$$
		which is finite-to-one. For $\phi \in \Phi(\mathrm O(V_{2n}))$, we denote $\mathcal L^{-1}(\phi)$ by $\Pi^0_{\phi}$ and call it the $L$-packet of $\phi$. We also write $\Pi^0_{\phi}(\SO(V_{2n}))=\Pi^0_{\phi}\cap \Irr\left(\SO(V_{2n})\right)$.
		\item For each Whittaker datum $\mathfrak W_{c^\prime}$, there exist a canonical bijection
		\begin{align}\label{524}
		\mathcal J^0_{\mathfrak W_{c^\prime}} : \Pi^0_{\phi} \longrightarrow \widehat {\mathcal S^+_{\phi}}.
		\end{align}
		\item The $L$-packet $\Pi^0_{\phi}$ and the bijection $\mathcal J^0_{\mathfrak W_{c^\prime}}$ satisfy the analogues of Theorem \ref{desideratumall} (3)-(12). 
		\item For $\phi\in \Phi(\mathrm O(V_{2n})),$ let $\Pi_{\phi}$ be the $L$-packet defined in Theorem \ref{desideratumall}. Then the image of $\Pi_{\phi}$ under the map 
		\begin{align*}
		\Irr(\mathrm O(V_{2n}^\delta))\longrightarrow \left(\Irr(\mathrm O(V_{2n}^\delta))/\sim_{\det} \right) \longrightarrow \left(\Irr(\SO(V_{2n}^\delta))/\sim_{\epsilon}\right)
		\end{align*}
		is the packet $\Pi_{\phi}^0$ and the diagram 
		$$
		\begin{CD}
		\Pi_{\phi} @>\mathcal J_{\mathfrak W_{c^\prime}}>> \widehat{\mathcal {S}_{\phi}} \\
		@VVV  @VV\ell^* V \\
		\Pi_{\phi}^0 @>\mathcal J^0_{\mathfrak W_{c^\prime}}>>  \widehat{\mathcal S^+_{\phi}}   
		\end{CD}
		$$
		is commutative for any Whittaker datum $\mathfrak W_{c^\prime}$, where $\ell:\mathcal S^+_{\phi}\rightarrow \mathcal {S}_{\phi}$ is the natural embedding. 
	%	\item The following are equivalent: 
	%	\begin{itemize}
	%		\item $\phi\in \Phi^\epsilon(\mathrm O(V_{2n}))$;
	%		\item some $[\pi_0]\in \Pi_{\phi}^0$ satisfies     $\pi_0^\epsilon \cong \pi_0$;
	%		\item all $[\pi_0]\in \Pi_{\phi}^0$ satisfy $\pi_0^\epsilon \cong \pi_0$. 
	%	\end{itemize}
	\end{enumerate}
\end{theorem}
\begin{proof}
	This follows from Theorem \ref{desideratumall}; see also Atobe--Gan \cite[\S 3.5]{MR3708200} for an explication. 
\end{proof}

\section{The Plancherel measures and normalized intertwining operators}\label{plancherelmeasure}
We recall the definition of Plancherel measures and prove Lemma \ref{homomorphicatsequal0} in this appendix. 
 
We retain the notation in Subsection \ref{normalizingintertwing}. Let $\overline{P}=M_{P} U_{\overline P}$ be the opposite parabolic subgroup to $P$ of $\mathrm O(V)$ and consider the induced representation
$$\Ind_{\overline{P}}^{\mathrm O\left(V\right)}\left(\tau_{s} \otimes \pi_0 \right).$$ 
Similarly to those of Subsection \ref{normalizingintertwing}, they are realized on the space of smooth functions 
$$\overline{\Psi}_{s} : \mathrm O(V) \rightarrow \mathscr{V}_{\tau} \otimes \mathscr{V}_{\pi_0}$$
such that
\begin{align*} \overline{\Psi}_{s}\left( u_{\overline P} m_{P}(a) h_0 h\right) &=|\det(a)|_{F}^{s+\rho_{\overline P}} \tau(a) \pi_0(h_0)  \overline{\Psi}_{s}\left(h\right) 
\end{align*}
for any $ u_{\overline P} \in U_{\overline P}, a\in \GL(X), h_0\in \mathrm{O}(V_0), h \in \mathrm{O}(V)$. As in \cite[\S 12]{MR3166215}, we define the 
standard intertwining operator
\begin{align*}
J_{\overline{P}|P}(\tau_{s}\otimes \pi_0): \Ind_{P}^{\mathrm O(V)}(\tau_{s}\otimes \pi_0) &\longrightarrow \Ind_{\overline{P}}^{\mathrm O(V)}(\tau_{s}\otimes \pi_0)
\end{align*}
by (the meromorphic continuations of) the integrals
\begin{align*}
J_{\overline{P}|P}(\tau_{s}\otimes \pi_0)\Psi_s(h)&= \int_{U_{\overline P}} \Psi_s(\bar uh)d\bar u 
\end{align*}
for $\Psi_s\in \Ind_{P}^{\mathrm O(V)}(\tau_{s}\otimes \pi_0)$. Similarly, we have the 
standard intertwining operator
\begin{align*}
J_{P|\overline P}(\tau_{s}\otimes \pi_0): \Ind_{\overline P}^{\mathrm O(V)}(\tau_{s}\otimes \pi_0) &\longrightarrow \Ind_{P}^{\mathrm O(V)}(\tau_{s}\otimes \pi_0). 
\end{align*}
By \cite[\S 12]{MR3166215}, the Plancherel measure associated to 
$\Ind_{ P}^{\mathrm O(V)}(\tau_{s}\otimes \pi_0)$ is a rational function $\mu(\tau_{s}\otimes\pi_0)$ such that 
\begin{align*}\label{planch}
J_{P|\overline{P}}(\tau_{s}\otimes \pi_0)\circ J_{\overline{P}|P}(\tau_{s}\otimes \pi_0)= \mu(\tau_{s}\otimes\pi_0)^{-1}. 
\end{align*}
At this point, the Plancherel measure $\mu(\tau_{s}\otimes\pi_0)$ is only well-defined up to a scalar since it depends on the choice of Haar measures on $U_P$ and $U_{\overline P}$. We refer the reader to \cite[Appendix B.2]{MR3166215} for the choice of these Haar measures we used here. 

\begin{comment}
Put
\begin{align*}
\rho_{\bar P|P}(\tau_{s}\otimes\pi_0)&= \gamma(s,\phi_{\tau}\otimes \phi_{\pi_0}^{\vee},\psi)^{-1}\gamma(2s,\wedge^{2}\circ \phi_{\tau},\psi)^{-1}\\
\rho_{P|\bar P}(\tau_{s}\otimes\pi_0)&= \gamma(-s,\phi_{\tau}^{\vee}\otimes \phi_{\pi_0},\psi_{-1})^{-1}\gamma(-2s,\wedge^{2}\circ \phi_{\tau}^{\vee},\psi_{-1})^{-1}\\
\end{align*}
Here $\phi_{\tau},\phi_{\pi_0},\phi_{\sigma_0}$  is the $L$-parameter of $\tau,\pi_0$ and $\sigma_0$ respectively. 
\end{comment} 
Fix a Whittaker datum $\mathfrak W_{c^\prime}$. Let $\widetilde{w}_{c^\prime}$ be the lift of $w$ in (\ref{131}). Then there is a intertwining isomorphism 
\begin{align*}
\ell(w_{c^\prime},\tau_s\otimes \pi_0): \Ind_{\overline P}^{\mathrm O(V)}(\tau_{s}\otimes \pi_0)\rightarrow \Ind_P^{\mathrm O(V)}(w(\tau_{s}\otimes \pi_0))
\end{align*}
given by left translation 
\begin{align*}
\left(\ell(w_{c^\prime},\tau_s\otimes \pi_0)\overline{\Psi}_{s}\right)(h) =\overline{\Psi}_s(\widetilde{w}_{c^\prime}^{-1}h)
\end{align*}
for $\overline{\Psi}_s\in \Ind_{\overline P}^{\mathrm O(V)}(\tau_{s}\otimes \pi_0)$ and $h\in \mathrm O(V)$. It is easy to check that the following diagram 
\begin{equation}\label{120}
\begin{CD}
\Ind_{P}^{\mathrm O(V)}(\tau_{s}\otimes \pi_0) @>J_{\overline P| P}(\tau_{s}\otimes \pi_0)>> \Ind_{\overline P}^{\mathrm O(V)}(\tau_{s}\otimes \pi_0) \\
@| @VV\ell(w_{c^\prime},\tau_s\otimes \pi_0)V\\
\Ind_{P}^{\mathrm O(V)}(\tau_{s}\otimes \pi_0) @>\mathcal M(\widetilde{w}_{c^\prime},\tau_{s}\otimes \pi_0)>> \Ind_{P}^{\mathrm O(V)}(w(\tau_{s}\otimes \pi_0))
\end{CD}
\end{equation}
is commutative, where $\mathcal M(\widetilde{w}_{c^\prime},\tau_{s}\otimes \pi_0)$ is the unnormalized intertwining operator defined in Subsection 7.1. 
\begin{comment}
 It is easy to check that 
$$
\ell(w_c,\tau_s\otimes \pi_0)\circ J_{\bar{P}|P}(\tau_{s}\otimes \pi_0)= \mathcal M(\widetilde{w}_{c},\tau_{s}\otimes \pi_0)
$$

Moreover, the following diagram 
$$
\begin{tikzcd}
\Ind_{P}^{\mathrm O(V)}(\tau_{s}\otimes \pi_0) \arrow[r,"J_{\bar P|P}"] \arrow[d, "\mathbbm 1"] & \Ind_{\bar P}^{\mathrm O(V)}(\tau_{s}\otimes \pi_0) \arrow[r, "J_{P|\bar P}"] \arrow[d,"\ell(w)"] &\Ind_{P}^{\mathrm O(V)}(\tau_{s}\otimes \pi_0)  \arrow[d,"\ell(w^2)"] \\
\Ind_{P}^{\mathrm O(V)}(\tau_{s}\otimes \pi_0) \arrow[r,"\mathcal M(\widetilde{w}_{c^\prime})"] & \Ind_{P}^{\mathrm O(V)}(w(\tau_{s}\otimes \pi_0)) \arrow[r]&\Ind_{P}^{\mathrm O(V)}(\tau_{s}\otimes \pi_0) 
\end{tikzcd}
$$
is commutative (rewrite this commutative diagram). 
It is easy to check that 
$$
\ell(w^2_c,\tau_s\otimes \pi_0)\circ J_{P|\bar P}(\tau_{s}\otimes \pi_0)= \mathcal M(\widetilde{w}_{c},w(\tau_{s}\otimes \pi_0))\circ \ell(w_c,\tau_s\otimes \pi_0)
$$
\end{comment}
Note that 
\begin{align*}
\widetilde{w}_{c^\prime}^2= 
m_P((-\mathbf 1)^{k-1})\cdot \mathbf 1_{V_0}.
\end{align*}
Hence $\widetilde{w}_{c^\prime}^2$ lies in the center of $M_P$. We have an intertwining isomorphism  
\begin{align*}
\ell(w^2_{c^\prime},\tau_s\otimes \pi_0): \Ind_{P}^{\mathrm O(V)}(\tau_{s}\otimes \pi_0)\rightarrow \Ind_P^{\mathrm O(V)}(w^2(\tau_{s}\otimes \pi_0))=\Ind_P^{\mathrm O(V)}(\tau_{s}\otimes \pi_0)
\end{align*}
given by left translation 
\begin{align}\label{129}
\left(\ell(w_{c^\prime}^2,\tau_s\otimes \pi_0)\Psi_{s}\right)(h) = \Psi_s((\widetilde{w}_{c^\prime})^{-2}h)=\omega_{\tau}(-1)^{k-1}\Psi_s(h).
\end{align}
for $\Psi_s\in \Ind_{P}^{\mathrm O(V)}(\tau_{s}\otimes \pi_0)$ and $h\in \mathrm O(V)$. Here $\omega_\tau$ is the central character of $\tau$. Then we have the following commutative diagram 
\begin{equation}\label{122}
\begin{CD}
\Ind_{\overline P}^{\mathrm O(V)}(\tau_{s}\otimes \pi_0) @>J_{P|\overline P}(\tau_{s}\otimes \pi_0)>> \Ind_{P}^{\mathrm O(V)}(\tau_{s}\otimes \pi_0) \\
@VV\ell(w_{c^\prime},\tau_s\otimes \pi_0)V @VV\ell(w^2_{c^\prime},\tau_s\otimes \pi_0)V\\
\Ind_{P}^{\mathrm O(V)}(w(\tau_{s}\otimes \pi_0)) @>\mathcal M(\widetilde{w}_{c^\prime},w(\tau_{s}\otimes \pi_0))>> \Ind_{P}^{\mathrm O(V)}(\tau_{s}\otimes \pi_0).
\end{CD}
\end{equation}
Combining (\ref{120}), (\ref{129}) and (\ref{122}), we have 
\begin{equation}\label{1235}
\begin{aligned}
&\mathcal M(\widetilde{w}_{c^\prime},w(\tau_{s}\otimes \pi_0))\circ \mathcal M(\widetilde{w}_{c},\tau_{s}\otimes \pi_0)\\
=& \ell(w^2_{c^\prime},\tau_s\otimes \pi_0)\circ J_{P|\overline P}(\tau_{s}\otimes \pi_0)\circ J_{\overline P|P}(\tau_{s}\otimes \pi_0)\\
=& \omega_{\tau}(-1)^{k-1}\times \mu(\tau_{s}\otimes\pi_0)^{-1}.
\end{aligned}
\end{equation}

Now we begin to prove Lemma \ref{homomorphicatsequal0} for even orthogonal groups. 
\begin{proposition}\label{appendixa1}
Let $\mathcal R_{\mathfrak W_{c^\prime}}\left(w, \tau_{s} \otimes \pi_0\right)$ and be the normalized intertwining operator defined in Subsection \ref{normalizingintertwing}. Then we have  
\begin{enumerate}[(i)]
	\item $	\mathcal R_{\mathfrak W_{c^\prime}}\left(w, w(\tau_{s} \otimes \pi_0)\right)\circ \mathcal R_{\mathfrak W_{c^\prime}}\left(w, \tau_{s} \otimes \pi_0\right)=1 $;
	\item $\mathcal R_{\mathfrak W_{c\prime}}\left(w, w(\tau_{-\bar s} \otimes \pi_0)\right)^*=\mathcal R_{\mathfrak W_{c^\prime}}\left(w,\tau_{s} \otimes \pi_0 \right)$. 
\end{enumerate}
In particular, $\mathcal R_{\mathfrak W_{c^\prime}}\left(w,\tau_{s} \otimes \pi_0 \right)$ is unitary when $s$ is purely imaginary. Hence $\mathcal R_{\mathfrak W_{c^\prime}}\left(w,\tau_{s} \otimes \pi_0 \right)$ is holomorphic at $s=0$. 
\end{proposition}
\begin{proof}
	We first prove (i). Let $\phi_\tau$ and $\phi_{0}$ be the $L$-parameters of $\tau$ and $\pi_0$. Then by (\ref{1235}) and Corollary \ref{LIRfinal}, we have 
	\begin{align*}
	 &\quad\mathcal R_{\mathfrak W_{c^\prime}}\left(w, w(\tau_{s} \otimes \pi_0)\right)\circ \mathcal R_{\mathfrak W_{c^\prime}}\left(w, \tau_{s} \otimes \pi_0\right)\\
	&=\epsilon(V)^{2k}\times \chi_{V}(c^\prime/c)^{2k}\times r(w,w(\tau_s\otimes\pi_{0}))^{-1}\times r(w,\tau_s\otimes\pi_{0})^{-1}\\
	& \quad \times|c^\prime|_F^{k\rho_P}\times |c^\prime|_F^{k\rho_{\overline P}}\times \mathcal M(\widetilde{w}_{c^\prime},w(\tau_{s}\otimes \pi_0))\circ \mathcal M(\widetilde{w}_{c^\prime},\tau_{s}\otimes \pi_0)\\
	&=r(w,w(\tau_s\otimes\pi_{0}))^{-1}\times r(w,\tau_s\otimes\pi_{0})^{-1}\times \omega_{\tau}(- 1)^{k-1}\times \mu(\tau_{s}\otimes\pi_0)^{-1}\\
	&=r(w,w(\tau_s\otimes\pi_{0}))^{-1}\times r(w,\tau_s\otimes\pi_{0})^{-1}\times \omega_{\tau}(- 1)^{k-1}\times \gamma(s,\phi_{\tau}\otimes \phi_{0}^{\vee},\psi)^{-1}\\
	&\quad \times\gamma(-s,\phi_{\tau}^{\vee}\otimes\phi_0, \psi_{-1} )\times
   \gamma (2s, \wedge^2\circ \phi_{\tau}, \psi)^{-1}\times \gamma(-2s, \wedge^{2} \circ \phi_{\tau}^{\vee},\psi_{-1})^{-1}\\
	&=1,
\end{align*}
where the last equality follows from (\ref{normalizationfactor}), (\ref{12345}) and the following formulas 
	\begin{align*}
	\gamma(s,\phi,\psi)&= \frac{\varepsilon(s,\phi,\psi)\times L(1-s,\phi^\vee)}{L(s,\phi)} ,\\
	\varepsilon(s,\phi,\psi_{-1})&=\det(\phi)(-1)\times  \varepsilon(s,\phi,\psi)
	\end{align*}
	for any representation $\phi$ of $\WD_F$. 
	 
	The second statement follows from a similar argument in \cite[Proposition 2.3.1]{MR3135650}, we omit the details here. 
\end{proof}

\bibliographystyle{alpha}
\nocite{*}
\bibliography{LLC4ort}

\end{document}